\documentclass[preprint,12pt]{article}
\pdfoutput=1
\usepackage[top=1in, bottom=1in, left=0.9in, right=0.9in]{geometry}
\usepackage{amsmath,amsfonts,amssymb,amsthm}
\usepackage{hyperref}
\usepackage{mathrsfs,dsfont}
\usepackage{graphicx,color}
\usepackage{colortbl,dcolumn}
\usepackage{psfrag}
\usepackage{booktabs}
\usepackage{cite}
\usepackage{subfigure}
\usepackage{url}
\usepackage{enumerate}
\allowdisplaybreaks[4]
\numberwithin{equation}{section}
\usepackage[justification=centering]{caption}
\usepackage{authblk}
\usepackage{booktabs}
\usepackage{threeparttable}

\newtheorem{Theo}{Theorem}
\newtheorem{Lem}[Theo]{Lemma}
\newtheorem{Prop}{Proposition}
\newtheorem{Defi}{Definition}
\newtheorem{Assu}{Assumption}
\newtheorem{Rem}{Remark}

\numberwithin{equation}{section}

\begin{document}
\title{Stochastic theta methods for free stochastic differential equations}

\author[a]{Yuanling Niu}
\author[a]{Jiaxin Wei \thanks{Corresponding author: jiaxinwei@csu.edu.cn}}
\author[a]{Zhi Yin}
\author[a]{Dan Zeng}
\affil[a]{School of Mathematics and Statistics, HNP-LAMA, Central South University, Changsha, 410083, China.}
\renewcommand*{\Affilfont}{\small \it}
\date{}

\maketitle

\begin{abstract}
      {\rm
We introduce free probability analogues of the stochastic theta methods for free stochastic differential equations in this work. Assume that the drift coefficient of the free stochastic differential equations is operator Lipschitz and the  diffusion coefficients are locally operator Lipschitz, we prove the strong convergence of the numerical methods. Moreover, we investigate the exponential stability in mean square of the equations and the numerical methods.
In particular, the free stochastic theta methods with $\theta \in [1/2, 1]$ can inherit the exponential stability of original equations for any given step size. Our methods offer better stability than the free Euler-Maruyama method. Numerical results are reported to confirm these theoretical findings and show the efficiency of our methods compared with the free Euler-Maruyama method.      } \\

      \textbf{AMS subject classification: }
      {\rm {65C30, 60H35, 46L54, 46L55}}\\
      
\textbf{Key Words: }{\rm free stochastic differential equations, stochastic theta methods, strong convergence, exponential stability in mean square }
\end{abstract}

\maketitle

\section{Introduction}
Free probability, introduced by Voiculescu \cite{D94}, is a non-commutative analogue of classical probability, where the random variables are considered as operators in a non-commutative probability space. Free stochastic calculus was initiated by Biane and Speicher \cite{kummererStochasticIntegrationCuntz1992, bianeStochasticCalculusRespect1998, bianeFreeDiffusionsFree2001, speicher2001free} and is devoted to studying the operator-valued process.  
Let $(\mathscr{F}, \psi)$ be a non-commutative probability space, where $\mathscr{F}$ is a von Neumann algebra and $\psi: \mathscr{F} \to \mathbb{C}$ is a faithful unital normal trace on $\mathscr{F}.$ Denote the set of all self-adjoint operators in $\mathscr{F}$ by $\mathscr{F}^{sa}.$
We will consider the following free analogues of the stochastic differential equations \cite{karginFreeStochasticDifferential2011}:
\begin{equation}\label{eq:fSDE-0}
{\rm d} U_t =\alpha (U_t) {\rm dt} + \sum_{i=1}^k \beta^i (U_t)  {\rm d} W_t \gamma^i (U_t),
\end{equation}
where $\{U_t\}_t$ is a $\mathscr{F}^{sa}$-valued process, $\{W_t\}_t$ is the free Brownian motion \cite{bianeStochasticCalculusRespect1998} on $\mathscr{F}^{sa}$, and $\alpha, \beta^i,  \gamma^i: \mathscr{F}^{sa} \rightarrow \mathscr{F}^{sa}$ are operator-valued functions for $i=1,2,\ldots,k,$ $k\in \mathbb{N}$. Usually, $\alpha$ and $\beta^i,  \gamma^i$ are called the drift and diffusion coefficients, respectively. One motivation for studying free stochastic differential equations (free SDEs) dates back to stochastic calculus on matrices. 
Let $\{U_t^{(N)}\}_t$ be a $M_N(\mathbb{C})^{sa}$-valued process, where $M_N(\mathbb{C})^{sa}$ is the set of all $N \times N$ self-adjoint matrices,  and let $\{W_t^{(N)}\}_t$ be the ${\rm GUE}_N$ Brownian  motion \cite{Kemp2016}. Then the free SDEs \eqref{eq:fSDE-0} can be considered as the large $N$-limit of the following matrix-valued SDEs:
\begin{equation}
{\rm d} U_t^{(N)} =\alpha (U_t^{(N)}) {\rm dt} +\sum_{i.=1}^k \beta^i (U_t^{(N)})  {\rm d} W_t^{(N)} \gamma^i (U_t^{(N)}).
\end{equation}
We refer to \cite{Kemp2016} for the proof of the convergence of $\{ W_t^{(N)}\}_t$ to the free Brownian motion $\{W_t\}_t$. Hence, solving free SDEs \eqref{eq:fSDE-0} helps to study the limiting eigenvalue distribution of $U_t^{(N)}$ \cite{karginFreeStochasticDifferential2011, Kemp2016, hoBrownMeasuresFree2022}.

However, finding the exact solution to free SDEs is a daunting task due to the non-commutativity of the operators. By consulting the classical SDEs,  one can provide numerical solutions to free SDEs as a means of approximating exact solutions. For classical SDEs, there are many experts contributing themselves to exploring various numerical methods; see \cite{kloeden1992stochastic,milstein2021stochastic,highamStrongConvergenceEulerType2002, buckwarStochasticRungeKutta2010,MR2985171}.  On the contrary, studying the numerical methods for free SDEs is an emerging field. Biane and Speicher initially developed an Euler scheme to approximate the solution of a peculiar free SDE in the operator norm  \cite{bianeFreeDiffusionsFree2001}. Very recently, Schl\"{u}chtermann and Wibmer introduced a free analogue of the Euler-Maruyama method (free EM) for the free SDEs \eqref{eq:fSDE-0} \cite{schluechtermannNumericalSolutionFree2022a}. More precisely, for $P \in \mathbb{N},$ let  $ 0=t_0 \le t_1\le \cdots \le t_P=T $  be a partition of $[0, T]$ with a constant step size $ h= T/P,$ define 
	\begin{equation}
	U_{n+1}=U_n+\alpha(U_n)h+\sum_{i=1}^k \beta^i(U_n)\Delta W_n\gamma^i(U_n),\ n=0,1,\ldots,P-1,
	\end{equation}
with the starting value $ U_0,$ where $ \Delta W_n:=W_{t_{n+1}}-W_{t_n}.$ $ U_n $ is called the numerical approximation of $ U_t $ at time point $ t_n $. They proved the strong and weak convergence order of free EM by assuming that $\alpha,\beta^i,\gamma^i$ are locally operator Lipschitz in the $L^2 (\psi)$ norm (we refer to Definition \ref{defi:Lip} for the definition).

As indicated in \cite{schluechtermannNumericalSolutionFree2022a}, to implement free EM as a numerical method on a computer, it is necessary to consider the set of $N \times N$ matrices instead of the abstract von Neumann algebra $\mathscr{F}$ and let $N \to \infty$. However, the corresponding computational cost may be extremely expensive, so developing algorithms without step size restrictions for stability is meaningful.  Implicit methods usually have better stability properties for classical SDEs \cite{Higham2007,higham2003E,MR2926259,chenConvergenceStabilityBackward2020,MR4589707}.  Owing to these facts, we aim to establish the numerical approximations of implicit numerical methods for free SDEs. Motivated by the pioneering work of Schl\"{u}chtermann and Wibmer \cite{schluechtermannNumericalSolutionFree2022a} and the work of stochastic theta methods applied to classical SDEs \cite{MR2660864,MR1981636,MR1781202,MR2926259,higham2003E,MR4132905}, we introduce the following free analogues of stochastic theta methods (free STMs): let $0\leq \theta \leq 1,$
define
\begin{equation}\label{eq:stm-0}
	\bar{U}_{n+1}=\bar{U}_n+(1-\theta)\alpha(\bar{U}_{n})  h+\theta\alpha(\bar{U}_{n+1})  h+ \sum_{i=1}^k \beta^i (\bar{U}_n)  \Delta W_n\gamma^i(\bar{U}_n),\ n=0,1,\ldots,P-1,
	\end{equation} 
with the starting value $U_0\in \mathscr{F}^{sa}.$ We note that the free STMs reduce to the free EM when $\theta =0$. However, the free STMs with $0<\theta \leq 1$ are implicit methods which require some conditions for the existence of the numerical solution. For instance, suppose that $\alpha$ is operator Lipschitz, i.e., there is a constant $L_0>0$ such that
\begin{equation*}
\left\| \alpha(U) -\alpha (V) \right\| \leq L_0 \left\| U- V \right\|
\end{equation*}
for any operators $U, V \in \mathscr{F}^{sa}.$ Then $\bar{U}_n$ uniquely exists if $h <1/ (\theta L_0)$ by a fixed point argument.  

Our paper contributes two main results:  First, we will prove the strong convergence property of the free STMs. Just as in the classical case, strong convergence is generally used to judge the validity of numerical methods. Suppose that $\alpha$ is operator Lipschitz in the $L^2 (\psi)$ norm and $ \beta^i,\gamma^i\ (i=1, \ldots, k)$ are locally operator Lipschitz in the $L^2 (\psi)$ norm, then the free STMs \eqref{eq:stm-0} strongly converge to the exact solution to the free SDEs $\eqref{eq:fSDE-0}$ with order $1/2$; see Theorem \ref{thm:converges-stm}. We remark that our condition for the drift coefficient $\alpha$ is stronger than the one for free EM \cite{schluechtermannNumericalSolutionFree2022a}. The second result is devoted to the stability of the numerical solution, which is essential to confirming whether the numerical methods can inherit the stability of original equations. To the best of our knowledge, the stability of the exact and numerical solution to free SDEs has not been studied. Suppose that there exist constants $ \bar{L}, \bar{K}>0$ such that ${\rm (a)} \;  \psi( U \alpha(U))+\frac{k}{2}\sum_{i=1}^k \|\beta^i(U) \|_2^2 \cdot \|\gamma^i(U) \|_2^2 \leq -\bar{L} \| U \|_2^2$ and ${\rm (b)} \; \|\alpha(U)\|_2^2\leq \bar{K}\|U \|_2^2$ for any $ U \in \mathscr{F}^{sa}.$ Here  $ \|\cdot \|_2:=\psi[|\cdot|^2]^{1/2},$ where $|U|= (U^*U)^{1/2}$ for any $U \in \mathscr{F}.$ For $\theta \in [0, 1/2),$ the free STMs \eqref{eq:stm-0} are exponentially stable in mean square whenever the step size $0< h  < \frac{2\bar{L}}{(1-2\theta)\bar{K}}$; for $\theta \in [1/2, 1],$ the free STMs are exponentially stable in mean square for any given step size; see Theorem \ref{12wdx}. We emphasize that the case for $\theta =1,$ which is called the free backward Euler method (free BEM), offers better stability and efficiency than free EM. 
 At first glance, our results are similar to the classical case; however, our proofs involve many non-trivial applications of the free It\^{o} formula {\cite{anshelevichItoFormulaFree2002,bianeStochasticCalculusRespect1998} and free Burkholder-Gundy inequality  \cite{bianeStochasticCalculusRespect1998}, which have their own independent interests.  Finally, we verify our results by conducting numerical experiments on several specific free SDEs. 

The rest of the paper is organized as follows: Section \ref{sec:pre} collects some preliminary results. Section \ref{sec:main} concerns the strong convergence order and exponential stability in mean square of the numerical solution, which offer our main results. In the last section, we provide some numerical experiments to verify our theoretical results.  

\section{Preliminaries}\label{sec:pre}

In this section, we collect some results in free probability and free stochastic calculus, and refer to \cite{NS2006, mingoFreeProbabilityRandom2017} for more details. We will also discuss the solution of free SDEs, which refers to \cite{karginFreeStochasticDifferential2011,MR1475218}.


\subsection{Free probability}

 A non-commutative probability space $(\mathscr{F}, \psi)$ is a von Neumann algebra $\mathscr{F}$, with a faithful unital normal trace $\psi: \mathscr{F} \to \mathbb{C}.$ Here are two examples of non-commutative probability spaces:
\begin{enumerate}[{\rm (a)}]
\item $(L^\infty(\mathbb{R}, \mu), \mathbb{E}),$ where $\mu$ is a probability measure supporting on $\mathbb{R},$ and $\mathbb{E}$ is the expectation defined by
$\mathbb{E}[f]:= \int_{\mathbb{R}} f(t) {\rm d \mu(t)}.$
\item $(M_N(\mathbb{C}), {\rm tr}_N),$ where $M_N(\mathbb{C})$ is the set of $N \times N$ matrices, and ${\rm tr}_N:={\rm Tr}/N$ is the normalized trace.
\end{enumerate}

The operator $U \in \mathscr{F}$ is called a non-commutative probability random variable, and it is called centered if $\psi (U) =0.$ Instead of ``classical independence", ``free independence" takes a central role in free probability. Let $\mathscr{F}_1, \mathscr{F}_2, \ldots, \mathscr{F}_m$ be subsets of $\mathscr{F}.$ Denote $W^*(\mathscr{F}_i)$ by the von Neumann sub-algebra generated by $\mathscr{F}_i.$ We call $\mathscr{F}_1, \mathscr{F}_2, \ldots, \mathscr{F}_m$ are freely independent, if for any centered random variables $U_j \in W^* (\mathscr{F}_{i_j})~( i_j \in \{1, \ldots, m\}, j=1, \ldots, k),$
\begin{equation*}
\psi(U_1 \cdots U_k)=0,
\end{equation*}
whenever we have $i_j \neq i_{j+1}$  $(j=1, \ldots, k-1).$ Roughly speaking, free independence provides a systematic way to calculate the joint distribution of random variables from their marginal distributions. For instance, if $\{U_1, U_2\}, \{U_3\}$ are freely independent, then $\psi (U_1 U_3 U_2) = \psi (U_1 U_2) \cdot \psi (U_3).$

A self-adjoint random variable $S \in \mathscr{F}$ is called a semicircular element (or has a semicircular distribution with mean 0 and variance $\sigma^2$) \cite{MR3645423}, if 
\begin{equation*}
\psi \left(S^k \right) = \int_{\mathbb{R}} t^k {\rm d\mu_{sc} (t)}
\end{equation*}
for any $k \in \mathbb{N},$ where $\mu_{sc}$ is a probability measure supporting on $\mathbb{R}$ with density ${\rm d \mu_{sc}(t)} = \frac{1}{2\pi \sigma^2} \sqrt{4\sigma^2- t^2} \cdot {\bf 1}_{\{ |t| \leq 2\sigma\}} (t){\rm dt}.$

For any $U \in \mathscr{F},$ define 
\begin{equation*}
\left\| U \right\|_p:= \left( \psi \left( |U|^p\right) \right)^{1/p}
\end{equation*}
for $1\leq p <\infty,$ where $|U| = (U^* U)^{1/2}.$  The completion of $\mathscr{F}$ with respect to the norm $\|\cdot \|_p$ is called a non-commutative $ L^p $ space \cite{PX2003}, denoted by $ L^p (\psi).$ 


\subsection{Free stochastic integral}

Let $(\mathscr{F}, \psi) $ be a non-commutative probability space, and a filtration $\{\mathscr{F}_t\} _{t \geq 0}$ be a family of sub-algebras of $\mathscr{F}$ such that $ \mathscr{F}_s \subset \mathscr{F}_t $ for $s \le t$. A free stochastic process is a family of elements $\{U_t\}_{t\geq 0}$ for which the increments $U_t -U_s$ are freely independent with respect to the sub-algebra $\mathscr{F}_s$. $\{U_t\}_{t\geq 0}$ is called adapted to the filtration $\{\mathscr{F}_t\} _{t \geq 0}$ if $U_t \in \mathscr{F}_t$ for all $t \geq 0.$ A family of self-adjoint elements $ \{W_t\} _{t \ge 0} $ in $\mathscr{F}$ is called a free Brownian motion if it satisfies 
\begin{enumerate}[{\rm (a)}]
\item $ W_0 = 0;$
\item The increments $ W_t - W_s $ are freely independent from $ \mathscr{W}_s $ for all $ 0 \le s < t $, where $ \mathscr{W}_s$ is the von Neumann algebra generated by $\{W_\tau: 0 \leq \tau \leq s\}$;	
\item The increment $ W_t - W_s $  has a semicircular distribution with mean 0 and variance $ t-s $ for all $ 0 \le s < t $.	
 \end{enumerate}

It is clear that $\{ \mathscr{W}_t\}_{t \geq 0}$ is a filtration in $\mathscr{F},$ and the free Brownian motion $\{W_t\}_{t\geq 0}$ is adapted to $\{ \mathscr{W}_t\}_{t \geq 0}.$
Given $T>0,$ suppose that $ \beta_t,  \gamma_t \in \mathscr{W}_t $ and $\|\beta_t\| \cdot \|\gamma_t\| \in L^2([0,T])$. Let $0= t_0 \leq t_1 \leq \cdots \leq t_n=T$ be a partition of $ [0,T] $, denoted by $ \Delta_n $. 
Consider the finite sum 
\begin{equation*}
I_n:=\sum_{i=1}^{n}\beta_{t_{i-1}}(W_{t_i}-W_{t_{i-1}})\gamma_{t_{i-1}}.
\end{equation*}
It is known \cite{bianeStochasticCalculusRespect1998} that $I_n$ converges in the operator norm as $ d(\Delta_n)\rightarrow 0,$ where $ d(\Delta_n):=\max_{1\le k\le n}(t_{k}-t_{k-1}),$  and the convergence does not depend on the partition. Hence, we have the following definition:

\begin{Defi}[Free stochastic integral \cite{bianeStochasticCalculusRespect1998,karginFreeStochasticDifferential2011}] 
Given $T>0,$ suppose that $ \beta_t,  \gamma_t \in \mathscr{W}_t $ and $\|\beta_t\| \cdot \|\gamma_t\| \in L^2([0,T])$. The free stochastic integral of $\{\beta_t\}_{t\geq 0}$ and $\{\gamma_t\}_{t\geq 0}$ is defined as the limit (in the operator norm) of $I_n$, denoted by
\begin{equation*}
I=\int_{0}^{T}\beta_t {\rm d} W_t \gamma_t.
\end{equation*}
\end{Defi}

A key ingredient to show the convergence of $I_n$ is the following free Burkholder-Gundy (B-G) inequality \cite{bianeStochasticCalculusRespect1998}, which will be used frequently in our paper: 
\begin{equation}\label{eq:B-G}
\left\|\int_0^T \beta_t {\rm d} W_t \gamma_t \right\| \leq 2 \sqrt{2} \left( \int_{0}^T \|\beta_t\|^2 \cdot \|\gamma_t\|^2 {\rm dt} \right)^{1/2}, 
\end{equation}  
and for the $L^2(\psi)$-norm, we have the following free $L^2(\psi)$-isometry:
\begin{equation}\label{eq:isometry}
\left\|\int_0^T \beta_t {\rm d} W_t \gamma_t \right\|_2 =\left( \int_{0}^T \|\beta_t\|^2_2 \cdot \|\gamma_t\|^2_2 {\rm dt} \right)^{1/2}. 
\end{equation}  

Another important tool is the free It\^{o} formula \cite{bianeStochasticCalculusRespect1998, anshelevichItoFormulaFree2002}. In terms of formal rules, it can be written as follows:
\begin{equation}\label{eq:free-ito}
\begin{split}
&\alpha_t {\rm dt} \cdot \beta_t {\rm dt}  = \alpha_t {\rm dt} \cdot \beta_t {\rm d} W_t \gamma_t = \alpha_t {\rm d} W_t \beta_t \cdot  \gamma_t {\rm dt}=0,\\
&\alpha_t {\rm d} W_t \beta_t \cdot \gamma_t {\rm d} W_t \zeta_t  =\psi \left(\beta_t \gamma_t \right)  \alpha_t \zeta_t  {\rm dt}
\end{split}
\end{equation}
for any $\alpha_t, \beta_t, \gamma_t, \zeta_t \in \mathscr{W}_t.$


\subsection{Free stochastic differential equations}

Given operator-valued functions $\alpha, \beta^i,  \gamma^i: \mathscr{F}^{sa} \rightarrow \mathscr{F}^{sa}, i=1,2,\ldots,k,$ $k\in \mathbb{N}.$ A general form of free SDEs with an initial value of $U_0 \in \mathscr{F}^{sa}$ is given by \cite{karginFreeStochasticDifferential2011}
\begin{equation}\label{eq:fSDE}
{\rm d} U_t =\alpha (U_t) {\rm dt} +\sum_{i=1}^k \beta^i (U_t)  {\rm d} W_t \gamma^i(U_t),
\end{equation}
which is a convenient shortcut notation for the following integral equations:
\begin{equation}\label{eq:solution}
U_t = U_0 +\int_{0}^{t} \alpha(U_s ) {\rm ds} + \sum_{i=1}^k \int_{0}^{t} \beta^i (U_s ) {\rm d} W_s \gamma^i(U_s).   
\end{equation} 

Motivated by the classical SDEs \cite{MR1475218}, we have the following definition for the solution of free SDEs:
\begin{Defi}\label{djie}
For $T>0,$ a free stochastic process $\{U_t\}_{0\le t \le T}$ adapted to the filtration $\{\mathscr{W}_t\}_{t\geq 0}$ is called a (strong) solution to the free SDEs \eqref{eq:fSDE}, if the following conditions hold:
\begin{enumerate}[{\rm (a)}]
\item $U_t \in \mathscr{F}^{sa}$  for all $0\le t \le T$; 
\item $\|\alpha(U_t)\| \in L^1([0,T])$ and $\|\beta^i(U_t)\| \cdot \|\gamma^i(U_t)\| \in L^2([0,T]), i=1,2,\ldots,k$; 
\item $U_t$ fulfills equation \eqref{eq:solution} for all $0\le t \le T.$ 
\end{enumerate}
\end{Defi}

Recall that an operator-valued function $f: \mathscr{F}^{sa} \rightarrow \mathscr{F}^{sa}$ is called locally operator Lipschitz, if for all $M>0,$ 
there is a constant $L_f(M)>0$ such that
\begin{equation*}
\left\| f(U) -f (V) \right\| \leq L_f(M) \left\| U- V \right\|
\end{equation*}
for any $U, V \in \mathscr{F}^{sa}$ and $\| U \|, \| V \| \leq M.$ Moreover, if there is a constant $L_f>0$ such that
\begin{equation*}
\left\| f(U) -f (V) \right\| \leq L_f \left\| U- V \right\|
\end{equation*}
for any $U, V \in \mathscr{F}^{sa}.$ Then $f$ is called operator Lipschitz.

\begin{Rem}
Let $f: \mathbb{R} \to \mathbb{R}$ be a locally bounded, measurable real-valued function. We may define an operator-valued function $\tilde{f}: \mathscr{F}^{sa} \to \mathscr{F}^{sa}$ given by $\tilde{f}(U) = f(U)$ for any $U \in \mathscr{F}^{sa},$ where $f(U) \in \mathscr{F}^{sa}$ is obtained by the functional calculus. In this sense, the real-valued function $f$ is called (locally) operator Lipschitz if $\tilde{f}$ is (locally) operator Lipschitz. It is well-known that $f$ is Lipschitz does not imply that it is operator Lipschitz \cite{AP2016}. For instance, $f(x)= |x|$ is not operator Lipschitz.  
\end{Rem}

It was shown by Kargin that there exists a local solution to the free SDEs \eqref{eq:fSDE} if $\alpha, \beta^i,$ and $\gamma^i$ $(i=1,2,\ldots,k)$ are locally operator Lipschitz. More precisely, we have the following theorem:
\begin{Theo}[Theorem 3.1, \cite{karginFreeStochasticDifferential2011}]\label{thm:solution}
Suppose that $\alpha, \beta^i,$ and $\gamma^i$ $(i=1,2,\ldots,k)$ are locally operator Lipschitz. Then there exist a sufficient small $T_0>0$ and a family of operators $U_t$ defined for
all $t \in [ 0, T_0)$, such that $U_t$ is a unique solution to \eqref{eq:fSDE} for $t < T_0.$
Moreover, the solution $\{U_t\}_{t \in [0, T_0)}$ is uniformly bounded, i.e., 
there exists a constant $ A>0,$ such that
\begin{equation*}
\sup_{0\le t < T_0}\|U_t\| \le A. 
\end{equation*}  
\end{Theo}

However, the existence of the global solution to \eqref{eq:fSDE} is not clear by only assuming that $\alpha, \beta^i,$ and $\gamma^i$ are locally operator Lipschitz. We provide some examples for which the global solution exists. 
\begin{enumerate}[{\rm (a)}]
\item Free Ornstein-Uhlenbeck equation \cite{karginFreeStochasticDifferential2011}:
\begin{equation*}
{\rm d} U_t = \mu U_t {\rm dt} + \sigma  {\rm d} W_t, \; \mu, \sigma \in \mathbb{R}.
\end{equation*}
\item Free Geometric Brownian motion I \cite{karginFreeStochasticDifferential2011}:
\begin{equation*}
{\rm d} U_t = \mu U_t {\rm dt} + \sqrt{U_t}  {\rm d} W_t \sqrt{U_t}, \; \mu \in \mathbb{R}.
\end{equation*}
\item Free Geometric Brownian motion II \cite{karginFreeStochasticDifferential2011}:
\begin{equation*}
{\rm d} U_t = \mu U_t {\rm dt} + U_t  {\rm d} W_t + {\rm d}W_t U_t, \; \mu \in \mathbb{R}.
\end{equation*}
\item Free Cox-Ingersoll-Ross equation \cite{Free-CIR, bouchaudFinancialApplicationsRandom2015}:
\begin{equation*}
{\rm d} U_t = (\alpha- \beta U_t)  {\rm dt} + \frac{\sigma}{2} \sqrt{U_t}  {\rm d} W_t + \frac{\sigma}{2} {\rm d}W_t \sqrt{U_t}, 
\end{equation*}
where $\alpha, \beta, \sigma >0$ and satisfy the Feller condition $2 \alpha \geq \sigma^2.$
\end{enumerate}
Examples 1, 3, and 4 can be generalized to the following model: 
\begin{equation}\label{fSDEdanbian}
{\rm d} U_t =\alpha (U_t) {\rm dt} +\sum_{i=1}^k \left( \beta^i(U_t)  {\rm d} W_t + {\rm d} W_t \gamma^i (U_t) \right).
\end{equation}
If $\alpha, \beta^i,$ and $\gamma^i$ $(i=1,2,\ldots,k)$ are operator Lipschitz, then there exists a unique global solution. The proof is similar to the classical case  \cite{MR1475218}, and we omit the details.


\section{Free analogues of the stochastic theta methods}\label{sec:main}

Inspired by the pioneering work \cite{schluechtermannNumericalSolutionFree2022a}, we can define the free analogues of stochastic theta methods (free STMs). 
Recall that $\alpha, \beta^i,$ and $\gamma^i: \mathscr{F}^{sa} \to \mathscr{F}^{sa}$ $(i=1,2,\ldots,k)$ are operator-valued functions.

\begin{Defi}\label{theta}
Given $T>0$ and $\theta \in [0,1].$ For $P \in \mathbb{N},$ let $0= t_0 \le t_1\le \cdots \le t_P=T $ be a partition of $[0, T]$ with constant step size $ h= T/P.$
Define the one step free STMs approximation $\bar{U}_n$ of the solution $U_t$ to
the free SDEs \eqref{eq:fSDE} on $[0, T]$ by
	\begin{equation}\label{eqtheta}
	\bar{U}_{n+1}=\bar{U}_n+(1-\theta)\alpha(\bar{U}_{n})  h+\theta\alpha(\bar{U}_{n+1})  h+ \sum_{i=1}^k \beta^i (\bar{U}_n)  \Delta W_n \gamma^i (\bar{U}_n), 
\end{equation}
$n =0, 1, \ldots, P-1,$ with start value $\bar{U}_0=U_0 \in \mathscr{F}^{sa}$ and where $\Delta W_n:=W_{t_{n+1}} -W_{t_n}$ is the increment of the free Brownian motion $\{W_t\}_{t\geq 0}.$  Here $\bar{U}_n \in \mathscr{F}^{sa}$ is called the numerical approximation of $U_t$ at the time point $t_n.$
\end{Defi}

For $\bar{U}_n$ to be self-adjoint, it requires that the sum $ \sum_{i=1}^k \beta^i (\bar{U}_n)  \Delta W_n \gamma^i (\bar{U}_n)$ is self-adjoint for all $n.$ If $\theta=0,$ free STM reduces to the free Euler-Maruyama method (free EM) introduced in \cite{schluechtermannNumericalSolutionFree2022a}, and $\bar{U}_{n+1}$ is explicitly determined by $\bar{U}_n.$ However, for $0< \theta \leq 1,$ Equation \eqref{eqtheta} is implicit. Since the locally operator Lipschitz property of $\alpha, \beta^i,$ and $\gamma^i$ is not enough to guarantee the well-posedness of free SDEs and free STMs. We require the following assumption:
\begin{Assu}[Well-posedness of free SDEs and free STMs]\label{ass0}
Given $T>0$, we assume that there exists a unique solution $\{U_t\}_{t\in [0, T]}$ to the free SDEs \eqref{eq:fSDE}, satisfying $\sup_{t \in [0, T]} \| U_t\| < \infty.$ Moreover, suppose that
there exists a unique solution $\{\bar{U}_n\}_{n =0,1, \ldots, P}$ to the free STMs \eqref{eqtheta}, such that $\{\bar{U}_n\}_{n=0,1, \ldots, P}$ is adapted to $\{\mathscr{W}_{t_n}\}_{n=0,1, \ldots, P}.$
\end{Assu}

\begin{Rem}
We remark that one simple condition for the well-posedness of the free STMs can be easily obtained by using a fixed-point argument.  Define
\begin{equation*}
F(U):=\bar{U}_n+(1-\theta)\alpha(\bar{U}_n)h+\theta\alpha(U)  h+\sum_{i=1}^k \beta^i (\bar{U}_n)  \Delta W_n\gamma^i(\bar{U}_n).  
\end{equation*} 
We assume that the function $\alpha: \mathbb{R} \to \mathbb{R}$ is operator Lipschitz with Lipschitz constant $L_0.$
If $h<1/ (\theta L_0)$, then we have 
\begin{equation*}
		\|F(U)-F(V)\|\le\theta h\|\alpha(U)-\alpha(V)\|\le L_0 \theta h\|U-V\|< \|U-V\|,
\end{equation*}
for any $U, V \in \mathscr{F}^{sa}.$ Hence, $F(\cdot)$ is a contraction, and the fixed-point theorem implies the existence and uniqueness of $\bar{U}_{n+1}$. Moreover, define 
$f: \mathbb{R} \to \mathbb{R}$ by
\begin{equation*}
f (x):= x- \theta h \alpha(x).
\end{equation*}
Then, 
\begin{equation*}
f(x_1) -f(x_2) \geq (1 - \theta L_0 h) (x_1 -x_2) 
\end{equation*}
for $x_1 >x_2.$ Hence, $f$ is uniformly monotone for $h<1/ (\theta L_0)$, and the inverse $f^{-1}$ exists for every $h<1/ (\theta L_0)$. 
By the functional calculus, we have
\begin{equation*}
\bar{U}_{n+1} = f^{-1} \left( \bar{U}_n+(1-\theta)\alpha(\bar{U}_n)h+\sum_{i=1}^k \beta^i (\bar{U}_n)  \Delta W_n\gamma^i(\bar{U}_n)\right).
\end{equation*}
Therefore, $\{\bar{U}_n\}_{n=0,1, \ldots, P}$ is adapted to $\{\mathscr{W}_{t_n}\}_{n=0,1, \ldots, P}$ whenever the step size $h<1/ (\theta L_0).$
\end{Rem}

For the convergence of the free STMs, we need the following additional assumption. For simplicity, we assume $k=1$, and the general case can be obtained easily by the triangular inequality.
\begin{Defi}[\cite{AP2016, schluechtermannNumericalSolutionFree2022a}]\label{defi:Lip}
An operator-valued function $f: \mathscr{F}^{sa} \rightarrow \mathscr{F}^{sa}$ is called locally operator Lipschitz in the $L^2 (\psi)$ norm, if for all $M>0,$ 
there is a constant $L_f(M)>0$ such that
\begin{equation*}
\left\| f(U) -f (V) \right\|_2 \leq L_f(M) \left\| U- V \right\|_2
\end{equation*}
for any $U, V \in \mathscr{F}^{sa}$ and $\left\| U\right\|_2, \left\| V \right\|_2 \leq M.$ 

Moreover, $f$ is called operator Lipschitz in the $L^2 (\psi)$ norm, if there is a constant $L_f>0$ such that
\begin{equation*}
\left\| f(U) -f (V) \right\|_2 \leq L_f\left\| U- V \right\|_2
\end{equation*}
for any $U, V \in \mathscr{F}^{sa}.$ 
\end{Defi}

\begin{Assu}\label{ass1}
We assume that $ \alpha$ is operator Lipschitz in the $L^2(\psi)$ norm, and $\beta,\gamma$ are locally operator Lipschitz in the $L^2 (\psi)$ norm. 
\end{Assu}

It is clear that there exists a constant $K_\alpha>0$, only depending on $L_\alpha$ and $U_0,$ such that for any $U\in \mathscr{F}^{sa},$ 
\begin{equation}\label{lin}
		\left \|  \alpha(U)  \right \|_2   \le K_\alpha \left(1+\left \| U \right \|_2 \right). 
\end{equation}
We will also frequently use the following condition:
\begin{equation}\label{lin-1}
\left \|  \alpha(U)  \right \|^2_2  \le K'_\alpha \left(1+\left \| U \right \|^2_2 \right),
\end{equation}
where $K'_\alpha>0$ is a constant only depending on $L_\alpha$ and $U_0.$ Moreover, there exist constants $K_\beta(M), K'_\beta(M)>0$, only depending on $L_\beta(M)$ and $U_0,$ such that for any $U\in \mathscr{F}^{sa}$ and $\left\| U\right\|_2 \leq M,$
\begin{equation}\label{lin-3}
		\left \|  \beta(U)  \right \|_2   \le K_\beta(M) \left(1+\left \| U \right \|_2 \right) \; \text{and} \; \left \|  \beta(U)  \right \|^2_2  \le K'_\beta(M) \left(1+\left \| U \right \|^2_2 \right).
\end{equation}
And similar results hold for $\gamma.$ 

In the rest of this section, we will use the following notations:
\begin{equation}
\begin{split}
L_M&: = \max\{ L_\alpha, L_\beta(M), L_\gamma(M)\}, \\
K_M&: = \max\{ K_\alpha, K_\beta(M), K_\gamma(M)\}, 
K'_M: = \max\{ K'_\alpha, K'_\beta(M), K'_\gamma(M)\}.
\end{split}
\end{equation}


\subsection{Uniformly boundness of the numerical solution in a given interval}

We introduce a piecewise constant process $\{\bar{U}_t \}_{t \in [0, T]}$ defined by 
\begin{equation}\label{piecewise}
\bar{U}_t:= \bar{U}_n
\end{equation}
for $t_n \leq t < t_{n+1},$ where $\{\bar{U}_n\}_{n=0,1, \ldots, P}$ is a numerical solution on $[0, T]$ given by \eqref{eqtheta}. Suppose that $\{U_t\}_{t \in [0, T]}$
is a solution to the free SDEs \eqref{eq:fSDE}, it is not difficult to show that $\| U_t \|_2$ is uniformly bounded; see \cite[Remark 3.3]{schluechtermannNumericalSolutionFree2022a}.

Firstly, we need the following technical lemma (see also \cite[Lemma 6.5]{schluechtermannNumericalSolutionFree2022a}).
\begin{Lem}\label{productlocal}
Assume that $\beta,\gamma$ are locally operator Lipschitz in the $L^2(\psi)$ norm. Suppose that $\|\bar{U}_n\|_2 \leq M.$ Then 
there is a constant $K'_{\beta\gamma}(M)>0$ such that 
\begin{equation*}
\left\| \beta(\bar{U}_n)  \Delta W_n \gamma(\bar{U}_n) \right\|_2^2  \leq K'_{\beta\gamma}(M) \left(1+ \left\| \bar{U}_n\right\|_2^2\right) h.
\end{equation*}
\end{Lem}

\begin{proof}  
Since $\bar{U}_n$ and $\Delta W_n$ are freely independent, by the free cumulant-moment formula (see \cite{mingoFreeProbabilityRandom2017}), we have 
\begin{equation}\label{eq:freecmf}
\begin{split}
\left\| \beta(\bar{U}_n)  \Delta W_n \gamma(\bar{U}_n) \right\|_2^2 
 = &\psi \left( \beta(\bar{U}_n)  \Delta W_n \gamma(\bar{U}_n)  \beta(\bar{U}_n)  \Delta W_n \gamma(\bar{U}_n) \right)\\
 =&  (\psi \left( \gamma(\bar{U}_n) \beta(\bar{U}_n) \right))^2 \cdot \psi \left( (\Delta W_n )^2 \right)+ \psi \left( (\gamma(\bar{U}_n) \beta(\bar{U}_n) )^2 \right) \cdot (\psi \left( \Delta W_n \right))^2\\
&- (\psi \left( \gamma(\bar{U}_n) \beta(\bar{U}_n) \right))^2 \cdot (\psi \left( \Delta W_n \right))^2.
\end{split}
\end{equation}
Note that 
\begin{equation*}
\psi \left( (\Delta W_n )^2 \right) = \left\| \Delta W_n  \right\|_2^2 = \left\| \int_{t_n}^{t_{n+1}} {\rm d} W_s \right\|_2^2 =  \int_{t_n}^{t_{n+1}} 1 {\rm d s}  = h,
\end{equation*}
and $\psi \left( \Delta W_n \right)=0.$ It implies that
\begin{equation*}
\begin{split}
\left\| \beta(\bar{U}_n)  \Delta W_n \gamma(\bar{U}_n) \right\|_2^2 =  (\psi \left( \gamma(\bar{U}_n) \beta(\bar{U}_n) \right))^2 h.
\end{split}
\end{equation*}

Define an operator-valued function $\zeta: U \mapsto \psi \left( \gamma(U) \beta(U) \right)\cdot{\bf 1},$ where ${\bf 1}$ is the identity operator in $\mathscr{F}$. For any $U, V \in \mathscr{F}^{sa}$ and $\left\| U\right\|_2, \left\| V \right\|_2 \leq M,$ 
\begin{equation*}
\begin{split}
\|\psi \left( \gamma(U) \beta(U) \right)&\cdot{\bf 1}-\psi \left( \gamma(V) \beta(V) \right)\cdot{\bf 1}\|_2=|\psi \left( \gamma(U) \beta(U) \right)-\psi \left( \gamma(V) \beta(V) \right)|\\
&\leq |\psi \left((\gamma(U)-\gamma(V)) \beta(U) \right)|+|\psi \left( \gamma(V) (\beta(U)-\beta(V)) \right)|\\
&\leq \|\gamma(U)-\gamma(V)\|_2\|\beta(U)\|_2+\|\gamma(V)\|_2\|\beta(U)-\beta(V)\|_2\\
&\leq (L_{\beta}(M)K_{\gamma}(M)+L_{\gamma}(M)K_{\beta}(M))(1+M)\|U-V\|_2\\
&=:L_{\beta\gamma}(M)\|U-V\|_2,
\end{split}
\end{equation*}
where we have used the non-commutative Cauchy-Schwarz inequality \cite{PX2003} for the second inequality. Hence, $\zeta$ is locally operator 
Lipschitz in the $L^2(\psi)$ norm. Then there exists a constant $K'_{\beta\gamma}(M)>0$, only depending on $L_{\beta\gamma}(M)$ and $U_0,$ such that for any $U\in \mathscr{F}^{sa}$ and $\left\| U\right\|_2 \leq M,$
\begin{equation*}
(\psi \left( \gamma(U) \beta(U) \right))^2 = \left \|  \psi \left( \gamma(U) \beta(U) \right)\cdot{\bf 1}  \right \|_2^2   \le K'_{\beta\gamma}(M) \left(1+\left \| U\right \|_2^2 \right),
\end{equation*}
which completes our proof.
\end{proof}

\begin{Prop}\label{prop:uniform-bound}
Fix $0< C_0 <1.$  For $0 < \theta \leq1,$ suppose that Assumption \ref{ass0}, \ref{ass1} hold. Define $\tilde{A}= A+2,$ where $A = \sup_{t \in [0, T]} \| U_t \|_2^2.$ Denote 
$K' _0: = \max\{ K'_\alpha, K'_{\beta \gamma}(\tilde{A}) \}.$  Then there exists a  
time point $0< \tilde{T} \leq T,$ such that 
\begin{equation}
\sup_{t \in [0, \tilde{T}]}\left\| \bar{U}_t \right\|_2^2 \leq \tilde{A}
\end{equation}
whenever $0<h \leq \frac{C_0}{\theta (1+ K'_0)}.$
\end{Prop}

\begin{proof}
Since $\left\| U_0\right\|_2^2 \leq A < \tilde{A},$ there exists a time point $T^{(h)} \in (0, T],$ depending on $h$, such that 
$\| \bar{U}_t \|_2^2 \leq \tilde{A}$ for all $t \in [0, T^{(h)}].$ Consider a partition 
\begin{equation*}
 0=t_0\le t_1\le \cdots\le t_{n_0}=T^{(h)} \leq t_{n_0+1} \le \cdots \le t_P = T. 
 \end{equation*} 
 Without losing generality, we assume that $\|\bar{U}_{t_{n_0+1}}\|_2^2 >\tilde{A}.$

For $n=0, 1, \ldots, n_0,$ moving the term $\theta\alpha(\bar{U}_{n+1})  h$ in \eqref{eqtheta} to the left hand side, we obtain 
\begin{equation}\label{eq:0}
\bar{U}_{n+1}- \theta\alpha(\bar{U}_{n+1})  h =\bar{U}_n+(1-\theta)\alpha(\bar{U}_{n})  h+ \beta(\bar{U}_n)  \Delta W_n \gamma(\bar{U}_n).
\end{equation} 

On one hand, thanks to the non-commutative Cauchy-Schwarz inequality \cite{PX2003}, we have
\begin{equation*}
2\psi \left( \bar{U}_{n+1} \alpha(\bar{U}_{n+1})\right) \leq 2 \left\| \bar{U}_{n+1} \right\|_2 \cdot \left\|\alpha(\bar{U}_{n+1}) \right\|_2 \leq \left\| \bar{U}_{n+1} \right\|_2^2 + \left\|\alpha(\bar{U}_{n+1}) \right\|_2^2.
\end{equation*}
Then, it follows that
\begin{equation}\label{eq:2}
\begin{split}
\left\| \bar{U}_{n+1}- \theta\alpha(\bar{U}_{n+1})  h \right\|_2^2&  = \left\| \bar{U}_{n+1}\right\|_2^2 -2 \theta h \psi \left( \bar{U}_{n+1} \alpha(\bar{U}_{n+1})\right) + \theta^2 h^2 \left\| \alpha(\bar{U}_{n+1})\right\|_2^2\\
& \geq \left\| \bar{U}_{n+1}\right\|_2^2 -\theta h \left\|\alpha(\bar{U}_{n+1}) \right\|_2^2 - \theta h \left\| \bar{U}_{n+1}\right\|_2^2\\
& \geq \left(1- \theta h - K'_0 \theta h \right) \left\| \bar{U}_{n+1}\right\|_2^2 - K'_0 \theta h.
\end{split}
\end{equation}
On the other hand, 
\begin{equation}\label{eq:1}
\begin{split}
& \left\| \bar{U}_n+(1-\theta)\alpha(\bar{U}_{n})  h+ \beta(\bar{U}_n)  \Delta W_n \gamma(\bar{U}_n) \right\|_2^2\\
&  = \left\| \bar{U}_n\right\|_2^2  + (1-\theta)^2 h^2 \left\| \alpha(\bar{U}_n)\right\|_2^2 + \left\| \beta(\bar{U}_n)  \Delta W_n \gamma(\bar{U}_n) \right\|_2^2 + 2 (1-\theta) h \psi \left( \bar{U}_n \alpha( \bar{U}_n) \right) \\
&  \quad+ 2\psi \left( \bar{U}_n \beta (\bar{U}_n) \Delta W_n \gamma(\bar{U}_n)\right)  + 2(1-\theta) h \psi \left( \alpha(\bar{U}_n) \beta(\bar{U}_n)  \Delta W_n \gamma(\bar{U}_n) \right).
\end{split}
\end{equation}
Since $\bar{U}_n$ is freely independent from $\Delta W_n,$ we have 
\begin{equation*}
\psi \left( \bar{U}_n \beta (\bar{U}_n) \Delta W_n \gamma(\bar{U}_n)\right) = \psi \left(\bar{U}_n \beta (\bar{U}_n)  \gamma(\bar{U}_n)\right) \cdot \psi \left( \Delta W_n\right)=0,
\end{equation*}
where we have used $\psi \left( \Delta W_n \right)=0.$ Similarly, we have $ \psi \left( \alpha(\bar{U}_n) \beta(\bar{U}_n)  \Delta W_n \gamma(\bar{U}_n) \right)=0.$ 
Putting the above estimations together and using Lemma \ref{productlocal}, for $n=0, 1, \ldots, n_0,$
{\small\begin{equation}\label{eq:3}
\begin{split}
\eqref{eq:1}& \leq \left\| \bar{U}_n\right\|_2^2  + (1-\theta)^2 h^2 \left\| \alpha(\bar{U}_n)\right\|_2^2 +  K'_0 \left(1+ \left\| \bar{U}_n\right\|_2^2\right) h + 2 (1-\theta) h \psi \left( \bar{U}_n \alpha( \bar{U}_n) \right)\\
&\leq \left\| \bar{U}_n\right\|_2^2  + (1-\theta)^2 h^2 \left\| \alpha(\bar{U}_n)\right\|_2^2 + K'_0 \left(1+ \left\| \bar{U}_n\right\|_2^2\right) h+ (1-\theta) h \left\| \bar{U}_n\right\|_2^2+ (1-\theta) h \left\|\alpha( \bar{U}_n) \right\|_2^2\\
& \leq f\left(\theta, h, K'_0\right) + \left[ 1+ g\left(\theta, h, K'_0\right)\right] \cdot \left\| \bar{U}_n \right\|_2^2,
\end{split}
\end{equation}}
where we denote $f\left(\theta, h, K'_0 \right):= \left[ 1+(1-\theta) +(1-\theta)^2 h \right] K'_0 h$ and $g\left(\theta, h, K'_0\right):=(1-\theta)h + f\left(\theta, h, K'_0\right).$ Combining \eqref{eq:0}, \eqref{eq:2}, and \eqref{eq:3}, we obtain (note that $h \leq \frac{C_0}{\theta (1+ K'_0)} < \frac{1}{\theta (1+ K'_0)} $)
\begin{equation*}
\left(1- \theta h - K'_0 \theta h \right) \left\| \bar{U}_{n+1}\right\|_2^2 - K'_0 \theta h \leq f\left(\theta, h, K'_0\right) + \left[ 1+ g\left(\theta, h, K'_0\right)\right] \cdot \left\| \bar{U}_n \right\|_2^2,
\end{equation*}
which is equivalent to 
\begin{equation*}
\begin{split}
\left\| \bar{U}_{n+1}\right\|_2^2 & \leq \frac{f\left(\theta, h, K'_0\right) +K'_0\theta h}{1- \theta h - K'_0 \theta h } + \frac{1+ g\left(\theta, h, K'_0\right)}{1- \theta h - K'_0 \theta h } \cdot \left\| \bar{U}_n \right\|_2^2\\
& =\bar{f} \left(\theta, h, K'_0\right) h + \left( 1+ \bar{g}\left(\theta, h, K'_0\right) h \right) \cdot \left\| \bar{U}_n \right\|_2^2
\end{split}
\end{equation*}
for $n=0, 1, \ldots, n_0,$ where we denote
\begin{equation*}
\bar{f} \left(\theta, h, K'_0\right):=  \frac{ \left[ 2+(1-\theta)^2 h \right] K'_0}{1- \theta h - K'_0 \theta h } \leq  \frac{ \left[ 2+(1-\theta)^2 T \right] K'_0}{1-C_0}=:B_1
\end{equation*}
and
\begin{equation*}
\bar{g} \left(\theta, h, K'_0 \right):= \frac{1+ \left[ 2+(1-\theta)^2 h \right] K'_0}{1- \theta h - K'_0 \theta h } \leq \frac{1+ \left[ 2+(1-\theta)^2 T \right] K'_0}{1- C_0} =: B_2.
\end{equation*}

Note that $B_1,B_2$ are independent of $h$ and $0\leq B_1\leq B_2.$ The discrete Gronwall inequality implies that for any $n \geq 0$
\begin{equation*}
\begin{split}
\left\| \bar{U}_{n} \right\|_2^2 &\leq \left( \left\| U_0\right\|_2^2 + \frac{B_1}{B_2}\right) \cdot e^{B_2 n h} - \frac{B_1}{B_2} \leq \left( 1+ \left\| U_0\right\|_2^2\right) \cdot e^{B_2 n h} \leq (1+A) \cdot e^{B_2 n h}.
\end{split}
\end{equation*}
Since $\bar{U}_t$ is a piecewise constant process, we have $\bar{U}_t = \bar{U}_{n}$ for $t \in [t_{n}, t_{n+1}).$ So we arrive at
\begin{equation*}
\left\| \bar{U}_t \right\|_2^2\leq (1+A) \cdot e^{B_2 t_n}\leq(1+A) \cdot e^{B_2 t},
\end{equation*}
for $t \in [t_n, t_{n+1}).$ Therefore, we have 
\begin{equation*}
\left\| \bar{U}_t \right\|_2^2\leq (1+A) \cdot e^{B_2 t},
\end{equation*}
for $t \in [0, t_{n_0+1}].$
Since $1+A < \tilde{A},$ we define $\tilde{T}: = \frac{1}{B_2} \ln\left( \frac{\tilde{A}}{1+A}\right)>0.$ Then we must have $\tilde{T} < t_{n_0+1}.$ Thus, $\sup_{t \in [0, \tilde{T}]}\| \bar{U}_t\|_2^2 \leq \tilde{A}$ for any $h \leq \frac{C_0}{\theta (1+ K'_0)}.$
\end{proof}


\subsection{Strong convergence of the free STMs}\label{section:strongconvergence}

To establish the strong convergence of the free STMs, we first give the following definition of strong convergence of a numerical method which produces the numerical solution $\{\bar{U}_n\}_{n=0,1, \ldots, P}$:
\begin{Defi}[Definition 6.1, \cite{schluechtermannNumericalSolutionFree2022a}]
The numerical approximation is said to strongly converge to the solution $\{U_t\}_{t \in [0, T]}$ to the free SDEs \eqref{eq:fSDE} with 
order $p>0$ if there exists a constant $C>0$ independent of $h,$ such that
\begin{equation}\label{eq:strong-cov}
\psi \left( \left| U_{t_n}-\bar{U}_n\right| \right) \leq C h^p
\end{equation}
for $n=0, 1, \ldots, P.$ $U_{t_n}$ denotes the solution $U_t$ evaluated at time point $t_n.$
\end{Defi}

In this subsection, we will prove the strong convergence order of the free STMs \eqref{eqtheta}, namely, we have the following theorem:
\begin{Theo}\label{thm:converges-stm}
Suppose that Assumption \ref{ass0}, \ref{ass1} hold. Then, there exists a constant $A'>0$, does not depend on $h,$ such that $\sup_n \|\bar{U}_n\|_2^2 \leq A'$ whenever the step size $h$ is sufficiently small. Moreover, the free STMs \eqref{eqtheta} strongly converge to the solution $\{U_t\}_{t \in [0, T]}$ to the free SDEs \eqref{eq:fSDE} with order $p=1/2$ as $h \to 0,$ i.e., there exists a constant $C>0$ independent of $h,$ such that
\begin{equation}\label{slj}
\|U_{t_n}-\bar{U}_n\|_2 \leqslant C h^{1/2}~\text{as}~h\to 0
\end{equation}
for $n =0, 1, \ldots, P.$
\end{Theo}

We emphasize that \eqref{slj} is a stronger result compared to \eqref{eq:strong-cov} for $p= 1/2$. Since by the non-commutative Cauchy-Schwarz inequality \cite{PX2003}, we have
\begin{equation*}
\psi \left( \left| U_{t_n}-\bar{U}_n\right| \right) \leq \psi({\bf 1}^2)^{\frac{1}{2}} \cdot \|U_{t_n}-\bar{U}_n\|_2 =\|U_{t_n}-\bar{U}_n\|_2,
\end{equation*}
where ${\bf 1}$ is the identity operator in $\mathscr{F}$.

To show this theorem, we introduce the following intermediate terms, which are the values obtained after just one step of free STMs \eqref{eqtheta}:
\begin{equation}\label{uyibu}
\bar{U}(t_{n+1}):=U_{t_n} +\theta \alpha\left(U_{t_{n+1}}\right)  h+(1-\theta) \alpha\left(U_{t_n}\right)  h+\beta\left(U_{t_n} \right)  \Delta W_n \gamma\left(U_{t_n} \right)
		\end{equation} 
for $n=0,1,\ldots, P-1.$
Moreover, define the local truncation error 
 \begin{equation}\label{d}
		\delta_{n}:=U_{t_{n}} -\bar{U}(t_{n}), 
\end{equation}  
and the global error 
\begin{equation}\label{v}
		\varepsilon _n:=U_{t_n} -\bar{U}_n. 
		\end{equation}
Hence, proving the above theorem reduces to bounding the global error. We need the following technical lemmas: 

\begin{Lem}\label{prop:holder}
Suppose that Assumption \ref{ass0}, \ref{ass1} hold. Denote $\sup_{t \in [0, T]} \|U_t\|_2^2 \leq A.$ Then, there exists a constant $ C_{\ref{prop:holder}}>0,$ depending on $K_A, T,$ and $A$, such that for any $t, t' \in [0, T]$
\begin{equation}\label{hlianxu}
		\|U_t -U_{t'} \|^2_2 \le C_{\ref{prop:holder}} |t-t'|. 
\end{equation}  
In this sense, the solution $\{U_t\}_{t \in [0, T]}$ to the free SDEs \eqref{eq:fSDE} is $1/2$-H\"older continuous.
\end{Lem}

\begin{proof}
Without losing of generality, we assume $ 0 \leq t' < t \leq T.$ Recalling \eqref{eq:solution}, we write  
		\begin{equation*}
		U_t -U_{t'}  = \int_{t'}^t \alpha\left(U_s \right)  {\rm ds} +\int_{t'}^t \beta\left(U_s \right)  {\rm d} W_s \gamma\left(U_s \right). 
		\end{equation*}   
By the assumption and the free $L^2(\psi)$-isometry \eqref{eq:isometry}, we have 		
\begin{equation*}
\begin{split}
		\left\|U_t -U_{t'} \right\|^2_2
		&\leq 2\left\|\int_{t'}^t \alpha\left(U_s \right)  {\rm ds} \right\|^2_2 +
		2\left\|\int_{t'}^t \beta\left(U_s \right)  {\rm d} W_s \gamma\left(U_s \right)  \right\|^2_2\\
		& \leq 2 \left( \int_{t'}^t \left\| \alpha\left(U_s \right) \right\|_2 {\rm ds}\right)^2 + 2 \int_{t'}^t \left\| \beta\left(U_s \right) \right\|^2_2  \cdot  \left\| \gamma\left(U_s \right) \right\|^2_2 {\rm ds}\\
		& \leq 2\left( \int_{t'}^t K_A(1+ \left\| U_s  \right\|_2 ){\rm ds}\right)^2+ 2  \int_{t'}^t K_A^4(1+ \left\| U_s  \right\|_2 )^4{\rm ds} \\
		& \leq (1+ \sqrt{A})^2K_A^2 (2 T+ 2(1+\sqrt{A})^2 K_A^2) (t-t'),		
		\end{split}
		\end{equation*}
		where we have used $\sup_{t \in [0, T]} \|U_t\|^2_2 \leq A.$ Letting $ C_{\ref{prop:holder}}:=(1+ \sqrt{A})^2K_A^2 (2 T+ 2(1+\sqrt{A})^2 K_A^2)$, we complete the proof. 
	\end{proof}
	
	We recall the definition of conditional expectation as follows \cite{PX2003}:
 Let $\mathscr{G}$ be a sub-algebra of $\mathscr{F}.$ There exists a map $\mathbb{E}_{\mathscr{G}}: \mathscr{F} \rightarrow \mathscr{G},$ such that
\begin{enumerate}[(i)]
\item $\mathbb{E}_{\mathscr{G}} (y) = y$ for all $y \in \mathscr{G};$
\item $\mathbb{E}_{\mathscr{G}} (y_1 x y_2) = y_1 \mathbb{E}_{\mathscr{G}} (x) y_2$ for all $x \in \mathscr{F}, y_1, y_2 \in \mathscr{G}.$
\end{enumerate}
$\mathbb{E}_{\mathscr{G}}$ is called the conditional expectation (with respect to $\mathscr{G}$).

\begin{Lem}\label{lem3. 3}
Suppose that Assumption \ref{ass0}, \ref{ass1} hold. There exists a constant $C_{\ref{lem3. 3}}>0,$ depending on $C_{\ref{prop:holder}}, K'_A,$ and $L_A,$ such that
\begin{equation}\label{delta}
		\| \delta_{n+1} \|_2^2 \leq C_{\ref{lem3. 3}} h^2, 
\end{equation}
for $n=0, 1, \ldots, P-1.$ Moreover, we denote $\mathbb{E}_n: = \mathbb{E}_{\mathscr{W}_{t_n}}$ be the conditional expectation onto the sub-algebra $\mathscr{W}_{t_n}$ for the time point $t_n$. Then, there exists a constant $c_{\ref{lem3. 3}}>0,$ depending on $C_{\ref{prop:holder}}$ and $L_A$, such that
\begin{equation}\label{delta-E}
		\left\| \mathbb{E}_n [\delta_{n+1} ] \right\|_2^2 \leq c_{\ref{lem3. 3}} h^3,
\end{equation}
for $n =0, 1, \ldots, P-1.$
\end{Lem}
\begin{proof}
By a direct computation, we have
\begin{equation}\label{lem3. 2. 1}
		\begin{split}
		\delta_{n+1}=
		& \theta \int_{t_n}^{t_{n+1}} ( \alpha(U_s) -\alpha(U_{t_{n+1}}) ) {\rm ds}
		+ (1-\theta)\int_{t_n}^{t_{n+1}} ( \alpha(U_s) -\alpha(U_{t_{n}}) ) {\rm ds}
		\\ 
		&+\int_{t_n}^{t_{n+1}}  ( \beta(U_s) -\beta(U_{t_n}) )  {\rm d} W_s \gamma(U_s)+ \int_{t_n}^{t_{n+1}} \beta(U_{t_n}){\rm d} W_s (\gamma(U_s ) -\gamma(U_{t_n} ) ). 
		\end{split}
		\end{equation}
By using the free $L^2(\psi)$-isometry \eqref{eq:isometry}, 
{\small
\begin{equation}\label{j}
		\begin{split}
		\|\delta_{n+1}\|_2^2&\leq
		4\theta^2 \left\| \int_{t_n}^{t_{n+1}} ( \alpha(U_s) -\alpha(U_{t_{n+1}}) ) {\rm ds} \right\|_2^2
		+ 4(1-\theta)^2 \left\| \int_{t_n}^{t_{n+1}} ( \alpha(U_s) -\alpha(U_{t_{n}}) ) {\rm ds}\right\|_2^2
		\\
		&\quad + 4\left\| \int_{t_n}^{t_{n+1}}  ( \beta(U_s) -\beta(U_{t_n}) )  {\rm d} W_s\gamma(U_s) \right\|_2^2 + 4\left\| \int_{t_n}^{t_{n+1}}\beta(U_{t_n}) {\rm d} W_s (\gamma(U_s ) -\gamma(U_{t_n} ) ) \right\|_2^2\\
		&\leq 4 \theta^2 \left( \int_{t_n}^{t_{n+1}} \left\|  \alpha(U_s) -\alpha(U_{t_{n+1}})  \right\|_2 {\rm ds}\right)^2
		+ 4(1-\theta)^2  \left( \int_{t_n}^{t_{n+1}} \left\| \alpha(U_s) -\alpha(U_{t_{n}}) \right\|_2 {\rm ds}\right)^2
		\\
		&\quad + 4\int_{t_n}^{t_{n+1}}  \left\| \beta(U_s) -\beta(U_{t_n}) \right\|^2_2 \cdot \left\|\gamma(U_s)\right\|_2^2 {\rm ds}  +  4 \int_{t_n}^{t_{n+1}}  \left\| \gamma(U_s ) -\gamma(U_{t_n} ) \right\|^2_2 \cdot \left\| \beta(U_{t_n})\right\|_2^2 {\rm d s}.
		\end{split}
		\end{equation} }
Then, by Lemma \ref{prop:holder}, 
\begin{equation*}
\begin{split}
\int_{t_n}^{t_{n+1}} \left\|  \alpha(U_s) -\alpha(U_{t_{n+1}})  \right\|_2 {\rm ds} & \leq L_A\left( \int_{t_n}^{t_{n+1}} \left\| U_s - U_{t_{n+1}} \right\|^2_2 {\rm ds}\right)^{1/2} \cdot \sqrt{h}  \\
& \leq  L_A\left( \int_{t_n}^{t_{n+1}} C_{\ref{prop:holder}} |s- t_{n+1}|{\rm ds} \right)^{1/2} \cdot \sqrt{h}\\
&= \frac{ \sqrt{C_{\ref{prop:holder}}}}{\sqrt{2}} L_A h^{3/2}. 
\end{split}
\end{equation*}
Moreover, 
\begin{equation*}
\begin{split}
\int_{t_n}^{t_{n+1}}  \left\| \beta(U_s) -\beta(U_{t_n}) \right\|^2_2 \cdot \left\|\gamma(U_s)\right\|_2^2 {\rm ds}& \leq L_A^2 K'_A \int_{t_n}^{t_{n+1}}  \left\| U_s - U_{t_n} \right\|^2_2 \cdot (1+\left\| U_s\right\|_2^2) {\rm ds}  \\
& \leq L_A^2 K'_A(1+A) \int_{t_n}^{t_{n+1}}  C_{\ref{prop:holder}} |s-t_n| {\rm ds} \\
&= \frac{C_{\ref{prop:holder}}}{2} L_A^2 K'_A (1+A) h^2.
\end{split}
\end{equation*}
Similarly, 
\begin{equation*}
\int_{t_n}^{t_{n+1}}  \left\| \gamma(U_s ) -\gamma(U_{t_n} ) \right\|^2_2 \cdot \left\| \beta(U_{t_n})\right\|_2^2 {\rm d s} \leq \frac{C_{\ref{prop:holder}}}{2} L_A^2 K'_A (1+A) h^2.
\end{equation*}
Therefore,
\begin{equation*}
\begin{split}
	\|\delta_{n+1}\|_2^2 & \leq 2\left[ \theta^2+ (1-\theta)^2\right] C_{\ref{prop:holder}} L_A^2 h^3 + 4 C_{\ref{prop:holder}} L_A^2 K'_A (1+A) h^2\\
	& \leq \left\{2\left[ \theta^2+ (1-\theta)^2\right] T + 4K'_A (1+A) \right\} C_{\ref{prop:holder}} L_A^2 h^2.
	\end{split}
\end{equation*}
By letting $C_{\ref{lem3. 3}}: = \left\{2\left[ \theta^2+ (1-\theta)^2\right] T + 4K'_A (1+A) \right\} C_{\ref{prop:holder}} L_A^2,$ we prove the bound \eqref{delta} for $\|\delta_{n+1}\|_2^2$.

To prove the bound for $\left\| \mathbb{E}_n[\delta_{n+1}] \right\|_2^2,$ we note that
\begin{equation*}
\mathbb{E}_n \left[ \int_{t_n}^{t_{n+1}}  ( \beta(U_s) -\beta(U_{t_n}) )  {\rm d} W_s \gamma(U_s)+ \int_{t_n}^{t_{n+1}} \beta(U_{t_n}){\rm d} W_s (\gamma(U_s ) -\gamma(U_{t_n} ) )\right]=0.
\end{equation*}
This is due to the fact that
\begin{equation*}
t \rightarrow \int_0^t ( \beta(U_s) -\beta(U_{t_n}) )  {\rm d} W_s\gamma(U_s)~\text{and}~t \rightarrow \int_0^t  \beta(U_{t_n}) {\rm d} W_s ( \gamma(U_s) -\gamma(U_{t_n}) )
\end{equation*} 
are martingales with respect the filtration $\{\mathscr{W}_t\}_{t \geq 0}$; see \cite[Proposition 3.2.3]{bianeStochasticCalculusRespect1998}. Hence, it follows from \eqref{lem3. 2. 1} that
\begin{equation*}
\mathbb{E}_{n} \left[ \delta_{n+1}\right]= \theta \mathbb{E}_{n}\int_{t_n}^{t_{n+1}} ( \alpha(U_s) -\alpha(U_{t_{n+1}}) ) {\rm ds}
		+ (1-\theta) \mathbb{E}_n \int_{t_n}^{t_{n+1}} ( \alpha(U_s) -\alpha(U_{t_{n}}) ) {\rm ds}.
\end{equation*}
By the contractivity of conditional expectation \cite{PX2003}, we have 
\begin{equation*}
\begin{split}
\left\| \mathbb{E}_{n} \left[ \delta_{n+1}\right] \right\|_2^2 
& \leq \theta^2 \left\| \mathbb{E}_{n}\int_{t_n}^{t_{n+1}} ( \alpha(U_s) -\alpha(U_{t_{n+1}}) ) {\rm ds} 
\right\|_2^2 + (1-\theta)^2 \left\|  \mathbb{E}_n \int_{t_n}^{t_{n+1}} ( \alpha(U_s) -\alpha(U_{t_{n}}) ) {\rm ds}\right\|_2^2\\
& \leq \theta^2 \left\| \int_{t_n}^{t_{n+1}} ( \alpha(U_s) -\alpha(U_{t_{n+1}}) ) {\rm ds} 
\right\|_2^2 + (1-\theta)^2 \left\| \int_{t_n}^{t_{n+1}} ( \alpha(U_s) -\alpha(U_{t_{n}}) ) {\rm ds}\right\|_2^2\\
& \leq  \left[\theta^2 + (1-\theta)^2 \right] \frac{C_{\ref{prop:holder}}}{2}  L_A^2 h^3.
\end{split}
\end{equation*}
By letting $c_{\ref{lem3. 3}}: =\left[\theta^2 + (1-\theta)^2 \right] C_{\ref{prop:holder}} L_A^2/2$, we complete our proof.
\end{proof} 

\begin{Lem}\label{lem:u}
Suppose that Assumption \ref{ass0}, \ref{ass1} hold. Denote
{\small\begin{equation*}
\begin{split}
 u_n:=&\theta(\alpha(U_{t_{n+1}} ) -\alpha(\bar{U}_{n+1}) )  h+(1-\theta)(\alpha(U_{t_{n}} ) -\alpha(\bar{U}_{n}) )  h  +\beta(U_{t_{n}} )\Delta W_n \gamma (U_{t_n})-\beta(\bar{U}_n) \Delta W_n \gamma(\bar{U}_{n}).
\end{split}
\end{equation*}}
Without losing of generality, we assume that there exists $n_0$ such that $t_{n_0} = \tilde{T},$ where $\tilde{T}$ is given in Proposition \ref{prop:uniform-bound}.
Then, there exists a constant $C_{\ref{lem:u}} >0,$ depending on $L_{\tilde{A}}, K'_{\tilde{A}}, A,$ and $\tilde{A},$ such that
\begin{equation}
\|u_n\|_2^2 \leq C_{\ref{lem:u}} h \| \varepsilon_n\|_2^2 + 4 \theta^2 L^2_{\tilde{A}} h^2 \|\varepsilon_{n+1}\|_2^2
\end{equation}
and
\begin{equation}
\left\| \mathbb{E}_n [u_n] \right\|_2^2 \leq (1-\theta) L_{\tilde{A}} h \| \varepsilon_n\|_2 + \theta L_{\tilde{A}} h \|\varepsilon_{n+1}\|_2
\end{equation}
for $n=0, 1, \ldots, n_0,$ whenever $0<h \leq \frac{C_0}{\theta (1+ K'_0)}.$
\end{Lem}

\begin{proof}
We write $u_n = u_n^{(1)} + u_n^{(2)},$ where
\begin{equation*}
u_n^{(1)} = \theta(\alpha(U_{t_{n+1}} ) -\alpha(\bar{U}_{n+1}) )  h+(1-\theta)(\alpha(U_{t_{n}} ) -\alpha(\bar{U}_{n}) )  h
\end{equation*}
and
\begin{equation*}
u_n^{(2)} = \beta(U_{t_{n}} )\Delta W_n \gamma (U_{t_n})  -\beta(\bar{U}_n) \Delta W_n \gamma(\bar{U}_{n}).
\end{equation*}

By the operator Lipschitz property of $\alpha,$ we have (note that $L_\alpha \leq L_{\tilde{A}}$)
\begin{equation*}
\left\| \alpha(U_{t_{n+1}}) - \alpha( \bar{U}_{n+1}) \right\|_2^2 \leq  L^{2}_\alpha \left\| U_{t_{n+1}} - \bar{U}_{n+1} \right\|_2^2 \leq  L^{2}_{\tilde{A}}\left\| U_{t_{n+1}} - \bar{U}_{n+1} \right\|_2^2.
\end{equation*}
Therefore,
\begin{equation*}
\begin{split}
\left\|u_n^{(1)} \right\|_2^2 & \leq 2 \theta^2 h^2 \left\| \alpha(U_{t_{n+1}} ) -\alpha(\bar{U}_{n+1}) \right\|_2^2 + 2(1-\theta)^2 h^2 \left\| \alpha(U_{t_{n}} ) -\alpha(\bar{U}_{n}) \right\|_2^2\\
& \leq 2 \theta^2 L_{\tilde{A}}^2 h^2 \left\| U_{t_{n+1}} - \bar{U}_{n+1} \right\|_2^2 + 2(1-\theta)^2 L_{\tilde{A}}^2 h^2 \left\| U_{t_{n}} - \bar{U}_{n} \right\|_2^2\\
& = 2 \theta^2 L_{\tilde{A}}^2 h^2 \|\varepsilon_{n+1}\|_2^2 + 2(1-\theta)^2 L^2_{\tilde{A}} h^2 \|\varepsilon_n\|_2^2.
\end{split}
\end{equation*}
For $\| u_n^{(2)}\|_2,$ by Proposition \ref{prop:uniform-bound}, $\| \bar{U}_{n} \|_2^2$ is uniformly bounded by $\tilde{A}$ for $n=0, 1, \ldots, n_0,$ whenever $0<h\leq \frac{C_0}{\theta (1+ K'_0)}.$ On the other hand, $\| U_{t_n} \|_2^2 \leq A \leq \tilde{A},$ then by the locally operator Lipschitz property of $\beta$ and $\gamma,$ we have 
\begin{equation*}
\begin{split}
\left\|u_n^{(2)} \right\|_2^2 & = \left\|( \beta(U_{t_{n}} )- \beta(\bar{U}_n)) \Delta W_n \gamma (U_{t_n}) + \beta(\bar{U}_n)  \Delta W_n( \gamma (U_{t_n}) -\gamma(\bar{U}_n))\right\|_2^2\\
& \leq 2 \left\| ( \beta(U_{t_{n}} )- \beta(\bar{U}_n)) \Delta W_n \gamma (U_{t_n}) \right\|_2^2 + 2 \left\|  \beta(\bar{U}_n)  \Delta W_n( \gamma (U_{t_n}) -\gamma(\bar{U}_n)) \right\|_2^2.
\end{split}
\end{equation*}
Note that $\psi((\Delta W_n)^2)=h.$ Similar to \eqref{eq:freecmf}, we deduce that
\begin{equation}\label{UWU}
\begin{split}
\left\| ( \beta(U_{t_{n}} )- \beta(\bar{U}_n)) \Delta W_n \gamma (U_{t_n}) \right\|_2^2 
&=(\psi\left( \gamma (U_{t_n})  ( \beta(U_{t_{n}} )- \beta(\bar{U}_n)) \right))^2\psi((\Delta W_n)^2)\\
&\leq\left\| \beta(U_{t_{n}} )- \beta(\bar{U}_n)\right\|_2^2 \left\| \gamma (U_{t_n}) \right\|_2^2 h\\
& \leq L^2_{\tilde{A}} K'_{\tilde{A}}(1+A) h \left\| \varepsilon_n\right\|_2^2.
\end{split}
\end{equation}
Similarly, 
\begin{equation*}
\left\|  \beta(\bar{U}_n)  \Delta W_n( \gamma (U_{t_n}) -\gamma(\bar{U}_n)) \right\|_2^2 \leq L^2_{\tilde{A}} K'_{\tilde{A}}(1+\tilde{A}) h \left\| \varepsilon_n\right\|_2^2.
\end{equation*}
It follows that
\begin{equation*}
\left\|u_n^{(2)} \right\|_2^2 \leq 2 L_{\tilde{A}}^2 K'_{\tilde{A}}(2+A+ \tilde{A}) h \left\| \varepsilon_n\right\|_2^2.
\end{equation*}
Finally, we obtain 
\begin{equation*}
\begin{split}
\left\|u_n\right\|_2^2  & \leq 2 \left\|u_n^{(1)}\right\|_2^2 + 2 \left\|u_n^{(2)}\right\|_2^2\\
& \leq 4L_{\tilde{A}}^2 [ K'_{\tilde{A}}(2+ A+ \tilde{A}) + (1-\theta)^2 \tilde{T} ] h \left\| \varepsilon_n\right\|_2^2 + 4 \theta^2 L^2_{\tilde{A}} h^2 \left\|\varepsilon_{n+1}\right\|_2^2\\
& \leq C_{\ref{lem:u}} h \left\| \varepsilon_n\right\|_2^2 + 4 \theta^2 L^2_{\tilde{A}} h^2 \left\|\varepsilon_{n+1}\right\|_2^2,\\
\end{split}
\end{equation*}
where we denote $C_{\ref{lem:u}}:=4L_{\tilde{A}}^2 [ K'_{\tilde{A}}(2+ A+ \tilde{A}) + (1-\theta)^2 \tilde{T}].$ 

Now we turn to bound $\left\| \mathbb{E}_n [u_n] \right\|_2^2.$ Note that
\begin{equation*}
\mathbb{E}_n \left[ \beta(U_{t_{n}} )\Delta W_n \gamma (U_{t_n})  -\beta(\bar{U}_n) \Delta W_n \gamma(\bar{U}_{n}) \right] =\beta(U_{t_{n}}) \mathbb{E}_n [\Delta W_n ]\gamma (U_{t_n})- \beta(\bar{U}_n) \mathbb{E}_n [\Delta W_n] \gamma(\bar{U}_{n}) =0.
\end{equation*}
Then we have $\mathbb{E}_n [u_n] = u_n^{(1)},$ it follows that
\begin{equation*}
\begin{split}
\left\| \mathbb{E}_n [u_n]\right\|_2 = \left\|u_n^{(1)} \right\|_2 & \leq  \theta h \left\| \alpha(U_{t_{n+1}} ) -\alpha(\bar{U}_{n+1}) \right\|_2 + (1-\theta) h \left\| \alpha(U_{t_{n}} ) -\alpha(\bar{U}_{n}) \right\|_2\\
& \leq  \theta L_{\tilde{A}} h \left\|\varepsilon_{n+1} \right\|_2 + (1-\theta) L_{\tilde{A}} h \left\|\varepsilon_n \right\|_2,
\end{split}
\end{equation*}
which completes the proof.
\end{proof}

Now we are ready to prove Theorem \ref{thm:converges-stm}.
\begin{proof}[Proof of Theorem \ref{thm:converges-stm}]
{\it Step 1} -- We decompose the global error $ \varepsilon_{n+1} $ as follows:
		\begin{equation*}\label{th3. 4. 1}
		\varepsilon_{n+1}
		=U_{t_{n+1}} -\bar{U}_{n+1}=U_{t_{n+1}} -\bar{U}(t_{n+1}) +\bar{U}(t_{n+1}) -\bar{U}_{n+1}  =\delta_{n+1}+\varepsilon_n+u_n,		
		\end{equation*}
for $n=0, 1, \ldots, n_0,$ where $t_{n_0} = \tilde{T}.$
Taking $L^2(\psi)$-norm on both sides, we obtain 
\begin{equation*}
\left\|\varepsilon_{n+1}\right\|_2^2 = \left\|\delta_{n+1}\right\|_2^2 + \left\|\varepsilon_n\right\|_2^2 +\left\|u_n\right\|_2^2 + 2\psi(\delta_{n+1} \varepsilon_n) + 2\psi(\varepsilon_n u_n)+ 2\psi(\delta_{n+1} u_n).
\end{equation*}
For sufficiently small step size $0<h \le \frac{C_0}{\theta (1+ K'_0)},$ applying Lemma \ref{lem3. 3} and Lemma \ref{lem:u}, we have 
\begin{equation*}
\begin{split}
2\psi(\delta_{n+1} \varepsilon_n) &=2 \psi (\mathbb{E}_n [\delta_{n+1} \varepsilon_n]) = 2 \psi (\mathbb{E}_n [\delta_{n+1}] \varepsilon_n)  \leq 2 \left\| \mathbb{E}_n [\delta_{n+1}] \right\|_2 \cdot \left\| \varepsilon_n\right\|_2 \\
& \leq 2 \sqrt{c_{\ref{lem3. 3}}} h^{3/2} \|\varepsilon_n\|_2 \leq h^2 + c_{\ref{lem3. 3}}h \|\varepsilon_n\|_2^2,
\end{split}
\end{equation*}
\begin{equation*}
\begin{split}
2\psi(\varepsilon_nu_n) &=2 \psi (\mathbb{E}_n [\varepsilon_n u_n]) = 2 \psi ( \varepsilon_n \mathbb{E}_n [u_n])  \leq 2 \left\| \mathbb{E}_n [u_n] \right\|_2 \cdot \left\| \varepsilon_n\right\|_2 \\
& \leq 2 (1-\theta)L_{\tilde{A}}  h \| \varepsilon_n\|_2^2 + 2 \theta L_{\tilde{A}} h \|\varepsilon_{n+1}\|_2 \cdot \|\varepsilon_n\|_2\\
& \leq 2 (1-\theta)L_{\tilde{A}}  h \| \varepsilon_n\|_2^2 + \theta L_{\tilde{A}}  h \left(\|\varepsilon_{n+1}\|_2^2 +\|\varepsilon_n\|_2^2 \right)\\
& =  (2-\theta)L_{\tilde{A}}  h \| \varepsilon_n\|_2^2 + \theta L_{\tilde{A}}  h \|\varepsilon_{n+1}\|_2^2,
\end{split}
\end{equation*}
and 
\begin{equation*}
\begin{split}
2\psi(\delta_{n+1} u_n) & \leq 2 \left\| \delta_{n+1} \right\|_2 \cdot \left\| u_n\right\|_2 \leq \left\| \delta_{n+1} \right\|_2^2 +  \left\| u_n\right\|_2^2   \leq C_{\ref{lem:u}} h \| \varepsilon_n\|_2^2 + 4 \theta^2 L_{\tilde{A}} ^2 h^2 \|\varepsilon_{n+1}\|_2^2 + C_{\ref{lem3. 3}} h^2.
\end{split}
\end{equation*}
Therefore, we obtain 
\begin{equation*}
\begin{split}
\|\varepsilon_{n+1}\|_2^2  \leq &(1+ 2 C_{\ref{lem3. 3}})h^2 + \left[ 1+2 C_{\ref{lem:u}} h + c_{\ref{lem3. 3}}h + (2-\theta) L_{\tilde{A}} h \right] \|\varepsilon_n\|_2^2  \\
&+ \left[ 8\theta^2 L_{\tilde{A}} ^2 h^2+ \theta L_{\tilde{A}} h \right] \| \varepsilon_{n+1}\|_2^2.
\end{split}
\end{equation*}
We assume that $0< h < \frac{\sqrt{33}-1}{16 \theta L_{\tilde{A}}},$ which implies $1-\theta hL_{\tilde{A}} - 8\theta^2 L_{\tilde{A}} ^2 h^2>0.$ It follows that 
\begin{equation}\label{varepsilondd}
\begin{split}
\left\|\varepsilon_{n+1}\right\|_2^2 & \leq \frac{1+2 C_{\ref{lem3. 3}}}{1-\theta hL_{\tilde{A}} - 8\theta^2 L_{\tilde{A}} ^2 h^2} h^2 + \left(1+ \frac{(2C_{\ref{lem:u}} +2L_{\tilde{A}}+ c_{\ref{lem3. 3}} )h+ 8 \theta^2 L_{\tilde{A}} ^2 h^2}{1-\theta hL_{\tilde{A}} -8\theta^2 L_{\tilde{A}} ^2 h^2} \right)\left\|\varepsilon_n\right\|_2^2\\
&:= A(h) h^2 + \left(1+ B(h)h \right) \|\varepsilon_n\|_2^2,
\end{split}
\end{equation}
where we denote $A(h) = \frac{1+2 C_{\ref{lem3. 3}}}{1-\theta hL_{\tilde{A}} - 8\theta^2 L_{\tilde{A}} ^2 h^2} $ and $B(h) =  \frac{(2C_{\ref{lem:u}} +2L_{\tilde{A}}+ c_{\ref{lem3. 3}} )+ 8 \theta^2 L_{\tilde{A}} ^2 h}{1-\theta hL_{\tilde{A}} -8\theta^2 L_{\tilde{A}} ^2 h^2}.$

Set $h_{MAX}$ such that $h\leq h_{MAX}< \min\{\frac{C_0}{\theta (1+ K'_0)},\frac{\sqrt{33}-1}{16 \theta L_{\tilde{A}}}\}.$ Let $A=A(h_{MAX})$ and $B=B(h_{MAX}).$ With abbreviation $h_n:=\|\varepsilon_n\|^2_2$, it follows from \eqref{varepsilondd} that
\begin{equation*}
h_{n+1}\leq A h^2+(1+Bh)h_n.
\end{equation*}
The Gronwall inequality then yields $h_{n}\leq e^{BT} h_0+A/B(e^{BT} -1)h.$ 
Since $h_0=0$, we conclude that
\begin{equation*}
\begin{split}
\|\varepsilon_{n}\|_2^2 & \leq O(h) ~\text{as}~h\to 0
\end{split}
\end{equation*}
for $n=1, \ldots, n_0.$

{\it Step 2} --It is similar to the proof of \cite[Theorem 6.4]{schluechtermannNumericalSolutionFree2022a}. If $\tilde{T}=T,$ i.e., $n_0 = P,$ then we are finished. 
If $\tilde{T}< T,$ by {\it Step 1}, we can conclude that $\| \bar{U}_{n_0} \|_2^2 \rightarrow \| U_{\tilde{T}} \|_2^2$ as $h \rightarrow 0.$ It implies that
$\| \bar{U}_{n_0} \|_2^2 \leq \| U_{\tilde{T}} \|_2^2\leq A< \tilde{A}.$ Hence, considering $t_{n_0}$ as a starting point, and by repeating the above proof, there exists $\tilde{T}^{(1)} \in (\tilde{T},  T],$ such that $\tilde{T}^{(1)}=2 \tilde{T},$ and $\| \bar{U}_n \|_2^2$ is uniformly bounded by $\tilde{A},$ whenever the step size $h$ is sufficiently small. 
Hence, repeating this procedure, we can find $\tilde{T}^{(k)}$ (after $k$ times), such that $\tilde{T}^{(k)} \geq T.$ Therefore, we can conclude that
$\| \bar{U}_n \|_2^2$ is uniformly bounded by $\tilde{A}$ for all $n=0, 1, \ldots, P,$ 
and \eqref{slj} follows from {\it Step 1}.
\end{proof}


\subsection{Exponential stability in mean square of the free STMs}

Motivated by the classical case, we introduce the exponentially stability in mean square of free SDEs. Usually, the stability is formulated in terms of the trivial solution, i.e., $U_t \equiv 0,$ if the trivial solution is available (see \cite{MR1475218} and the reference therein). However, we prefer to define the stability in terms of the SDEs (and later the corresponding numerical methods), rather than the trivial solution, as this allows for possible perturbation of the trivial solution under discretization \cite{higham2003E}.

\begin{Defi}
The free SDEs \eqref{eq:fSDE} are said to be exponentially stable in mean square if there exist constants $c_1, c_2 >0$ such that, for all initial value $U_0 \in \mathscr{F}^{sa},$
\begin{equation*}
		\left\|U_t \right\|_2^2 \le  c_1 e^{-c_2 t} \left\|U_0 \right\|_2^2,
		\end{equation*}
for all $t \geq 0.$
\end{Defi}

\begin{Assu}\label{ass2}	
There exists a constant $ \bar{L} >0,$ such that for any $ U \in \mathscr{F}^{sa},$ 
	\begin{equation}\label{ohtj}
	\psi\left( U \alpha(U) \right)+\frac{k}{2}\sum_{i=1}^k \left\|\beta^i(U) \right\|_2^2 \cdot \left\|\gamma^i(U) \right\|_2^2 \leq -\bar{L} \left\| U \right\|_2^2.
	\end{equation}
\end{Assu}

	\begin{Prop}\label{prop:stable-solution}
		Let $\{U_t\}_{t\geq 0}$ be a solution to the free SDEs \eqref{eq:fSDE}. Suppose that Assumption \ref{ass2} holds.
		Then there exists a constant $C'_{\ref{prop:stable-solution}}>0,$ such that 
		for all $ t\geq 0 $ 
		\begin{equation}\label{eq:solution-stable}
		\left\|U_t \right\|_2^2\le e^{-C'_{\ref{prop:stable-solution}} t} \left\|U_0 \right\|_2^2.
		\end{equation}
	\end{Prop}

\begin{proof}
For simplicity, we assume $k=1,$ and the general case can be easily proved by using the triangular inequality. 
Let $\{U_t\}_{t\geq 0}$ be a solution to the free SDEs \eqref{eq:fSDE} and $C'_{\ref{prop:stable-solution}}:=2\bar{L}>0$.
For all $t \geq 0,$ the free It\^{o} formula \eqref{eq:free-ito} readily implies that (see also \cite[Lemma 2.5]{Kemp2016})
		\begin{equation}\label{ito}
		\begin{split}
		{\rm d} (e^{C'_{\ref{prop:stable-solution}} t} U^2_t) &=e^{C'_{\ref{prop:stable-solution}} t} \cdot {\rm d} U^2_t + U^2_t \cdot {\rm d} e^{C'_{\ref{prop:stable-solution}} t}  \\
		&=e^{C'_{\ref{prop:stable-solution}} t} \left(U_t \cdot {\rm d} U_t + {\rm d} U_t \cdot U_t + {\rm d} U_t \cdot {\rm d} U_t \right)+ U^2_t \cdot {\rm d} e^{C'_{\ref{prop:stable-solution}} t}\\
		&=e^{C'_{\ref{prop:stable-solution}} t} \left(U_t \alpha(U_t)  {\rm dt}+\alpha(U_t) U_t {\rm dt}+ \psi \left(\gamma(U_t) \beta(U_t) \right) \cdot \beta (U_t) \gamma(U_t) {\rm dt}\right. \\
		&\left.\quad +U_t \beta(U_t)  {\rm d} W_t \gamma(U_t) + \beta(U_t)  {\rm d} W_t \gamma(U_t) U_t + C'_{\ref{prop:stable-solution}} U^2_t {\rm d} t\right).
		\end{split}
		\end{equation}
		Integrating from $0$ to $t$ on both sides of \eqref{ito} gives
		\begin{equation}\label{ito2}
		\begin{split}
		e^{C'_{\ref{prop:stable-solution}} t} U^2_t -U^2_0 
		&=\int_0^t e^{C'_{\ref{prop:stable-solution}} s} U_s  \alpha(U_s)  {\rm d s} +\int_0^t e^{C'_{\ref{prop:stable-solution}} s} \alpha(U_s )  U_s  {\rm d s}+ \int_0^t  e^{C'_{\ref{prop:stable-solution}} s} \psi \left(\gamma(U_s) \beta(U_s) \right) \cdot \beta(U_s) \gamma(U_s) {\rm ds}\\
		&\quad+\int_0^t e^{C'_{\ref{prop:stable-solution}} s} U_s  \beta(U_s)  {\rm d} W_s \gamma(U_s)  + \int_0^t e^{C'_{\ref{prop:stable-solution}} s}\beta(U_s)  {\rm d} W_s \gamma(U_s) U_s + C'_{\ref{prop:stable-solution}} \int_0^t e^{C'_{\ref{prop:stable-solution}} s} U^2_s  {\rm d s}.
		\end{split}
		\end{equation}
Note that
\begin{equation*}
\psi\left( \int_0^t U_s  \beta(U_s)  {\rm d} W_s \gamma(U_s)\right) = \psi\left(  \int_0^t \beta(U_s)  {\rm d} W_s \gamma(U_s) U_s\right) =0.
\end{equation*}
It follows from taking the trace in both side of \eqref{ito2} that
\begin{equation*}
		\begin{split}
		&\psi\left(e^{C'_{\ref{prop:stable-solution}} t} U^2_t \right) -\psi\left(U^2_0 \right) \\
		&=\psi\left(\int_0^t e^{C'_{\ref{prop:stable-solution}} s} U_s  \alpha(U_s)  {\rm ds} \right) +\psi\left(\int_0^t e^{C'_{\ref{prop:stable-solution}} s} \alpha(U_s)  U_s {\rm d s}\right) +\psi \left(C'_{\ref{prop:stable-solution}} \int_0^t e^{C'_{\ref{prop:stable-solution}} s} U^2_s  {\rm d s}\right)\\
		&\quad + \psi \left(  \int_0^t  e^{C'_{\ref{prop:stable-solution}} s} \psi \left(\gamma(U_s) \beta(U_s) \right) \cdot \beta(U_s) \gamma(U_s) {\rm ds} \right)\\
		& = 2 \int_0^t e^{C'_{\ref{prop:stable-solution}} s} \psi\left(U_s  \alpha(U_s) \right)  {\rm d s}+ \int_0^t  e^{C'_{\ref{prop:stable-solution}} s} \psi \left(\gamma(U_s) \beta(U_s) \right) \cdot \psi\left(\beta(U_s) \gamma(U_s)\right) {\rm ds}+
		C'_{\ref{prop:stable-solution}} \int_0^t e^{C'_{\ref{prop:stable-solution}} s} \psi \left(U^2_s \right){\rm d s}.
		\end{split}
		\end{equation*}			
Now using the non-commutative Cauchy-Schwarz inequality, \eqref{ohtj} results in
{\small\begin{equation*}
\begin{split}
e^{C'_{\ref{prop:stable-solution}} t} \psi\left(U^2_t \right) -\psi\left(U^2_0 \right) 
& \leq  2\int_0^t e^{C'_{\ref{prop:stable-solution}} s} \psi\left(U_s  \alpha(U_s) \right)  {\rm d s}+\int_0^t e^{C'_{\ref{prop:stable-solution}} s} \left\|\beta(U_s) \right\|_2^2 \cdot \left\|\gamma(U_s) \right\|_2^2 {\rm d s}+C'_{\ref{prop:stable-solution}} \int_0^t e^{C'_{\ref{prop:stable-solution}} s} \left\|U_s\right\|_2^2{\rm d s}\\
& \leq -2 \bar{L} \int_0^t e^{C'_{\ref{prop:stable-solution}} s} \|U_s \|_2^2 {\rm d s}+C'_{\ref{prop:stable-solution}} \int_0^t e^{C'_{\ref{prop:stable-solution}} s} \left\|U_s\right\|_2^2{\rm d s}=0. 
\end{split}
\end{equation*}}
Hence, our proof completes. 
\end{proof}	

\begin{Rem}
Suppose that the initial value $U_0=0.$ It is clear that \eqref{eq:solution-stable} forces the free SDEs to have the trivial solution. However, let $\mu<0$ and $\sigma>0,$ for the free Ornstein-Uhlenbeck equation ${\rm d} U_t = \mu U_t {\rm dt} + \sigma  {\rm d} W_t$ with $U_0 =0,$ we have $\|U_t\|_2^2 \to \sigma \sqrt{\frac{2}{|\mu|}}$ as $t \to \infty$ \cite{karginFreeStochasticDifferential2011}. Hence, the free Ornstein-Uhlenbeck equation does not have a trivial solution. Therefore, it motives us to introduce the exponential stability in mean square of the perturbation solution. We postpone it to Appendix \ref{sec:appendix}. 
\end{Rem}

Analogue to the classical SDEs \cite{higham2003E,MR2926259,MR4589707}, we have the following definition of exponential stability in mean square for a numerical method which produces the numerical solution $\bar{U}_n$:
\begin{Defi}
For a given step size $h>0$, a numerical method is said to be exponentially stable in mean square if there exist constants $C_1, C_2>0$, such that with initial value $U_0\in \mathscr{F}^{sa}$,
		\begin{equation} \label{wdx}
		\left\|\bar{U}_n \right\|_2^2\le C_1 e^{- C_2 nh} \left\|U_0 \right\|_2^2,
		\end{equation}
		for any $ n \in \mathbb{N}.$
	\end{Defi}

\begin{Theo}\label{12wdx}
Let $\{\bar{U}_n\}$ be a numerical solution to the free SDEs \eqref{eq:fSDE} given by the free STMs \eqref{eqtheta}.
	Suppose that Assumption \ref{ass2} holds, and additionally we assume that 	
	\begin{equation}\label{wdxtj2}
	\left\|\alpha(U) \right\|_2^2 \leq \bar{K} \left\|U \right\|_2^2,
	\end{equation}
	for any $U \in \mathscr{F}^{sa}.$
Then we have the following statements.
	\begin{enumerate}
	\item If $ \theta \in [0, 1/2),$ set $\rho$ such that $0< \rho < \frac{2\bar{L}}{(1-2\theta)\bar{K}},$ then for any given step size $0< h \leq \rho,$ there exists a constant $C_{\ref{12wdx}}>0$ such that
	\begin{equation*}
		\left\|\bar{U}_n \right\|_2^2\le e^{-C_{\ref{12wdx}}nh} \left\|U_0 -\theta h \alpha(U_0)\right\|_2^2,
		\end{equation*}
		for any $ n\in \mathbb{N}$.

	\item If $\theta \in [1/2, 1]$, then for any given step size $h>0,$
there exists a constant $C'_{\ref{12wdx}}>0$ such that
			\begin{equation*}
		\left\|\bar{U}_n \right\|_2^2\le e^{-C'_{\ref{12wdx}}nh} \left\|U_0 -\theta h \alpha(U_0)\right\|_2^2,
		\end{equation*}
		for any $ n\in \mathbb{N}$.
		\end{enumerate}
\end{Theo}

\begin{proof}
For simplicity, we again assume that $k=1.$ Recall that 
\begin{equation}\label{eq:theta2}
\bar{U}_{n+1}- \theta\alpha(\bar{U}_{n+1})  h =\bar{U}_n+(1-\theta)\alpha(\bar{U}_{n})  h+ \beta (\bar{U}_n)  \Delta W_n\gamma(\bar{U}_n).
\end{equation} 
Thanks to \eqref{ohtj}, the left side of \eqref{eq:theta2} can be lower bounded as follows:
		\begin{equation}\label{th6.1}
		\begin{split}
		\left\|\bar{U}_{n+1}-\theta\alpha(\bar{U}_{n+1})  h\right\|_2^2
		&=\left\|\bar{U}_{n+1}\right\|_2^2-2\theta h \psi \left(  \bar{U}_{n+1} \alpha(\bar{U}_{n+1}) \right)+\theta^2\left\|\alpha(\bar{U}_{n+1}) \right\|_2^2 h^2 \\
		&\geqslant 
		\left(1+2\bar{L} \theta h\right) \left\|\bar{U}_{n+1}\right\|_2^2.
		\end{split}
		\end{equation}
On the other hand, similar to \eqref{eq:1}, the right side of \eqref{eq:theta2} can be estimated as follows: 
\begin{equation}\label{th6.2}
\begin{split}
& \left\| \bar{U}_n+(1-\theta)\alpha(\bar{U}_{n})  h+\beta (\bar{U}_n)  \Delta W_n\gamma(\bar{U}_n) \right\|_2^2\\
& = \left\| \bar{U}_n-\theta h \alpha(\bar{U}_{n}) \right\|_2^2  +  h^2 \left\| \alpha(\bar{U}_n)\right\|_2^2 + \left\| \beta(\bar{U}_n)  \Delta W_n\gamma(\bar{U}_n) \right\|_2^2+ 2 h \psi \left( (\bar{U}_n-\theta h \alpha(\bar{U}_{n}))\alpha( \bar{U}_n) \right)\\
&\quad + 2 h \psi \left(  \beta (\bar{U}_n)  \Delta W_n\gamma(\bar{U}_n)\alpha( \bar{U}_n) \right)+2 \psi \left( (\bar{U}_n-\theta h \alpha(\bar{U}_{n})) \beta (\bar{U}_n)  \Delta W_n\gamma(\bar{U}_n) \right)\\
& \leq \left\| \bar{U}_n-\theta h \alpha(\bar{U}_{n}) \right\|_2^2  + (1-2\theta) h^2 \left\| \alpha(\bar{U}_n)\right\|_2^2 + \left\| \beta(\bar{U}_n) \right\|_2^2\cdot \left\|  \gamma(\bar{U}_n)  \right\|_2^2 h+ 2 h \psi \left( \bar{U}_n\alpha( \bar{U}_n) \right)\\
&\leq \left\| \bar{U}_n-\theta h \alpha(\bar{U}_{n}) \right\|_2^2 +((1-2\theta)\bar{K}h -2\bar{L})h\left\| \bar{U}_n\right\|_2^2,
\end{split}
\end{equation}
where we have used \eqref{ohtj}, \eqref{wdxtj2}, and the fact that (see \eqref{UWU})
\begin{equation*}
\left\|\beta (\bar{U}_n)  \Delta W_n\gamma(\bar{U}_n) \right\|_2^2  \leq \left\| \beta(\bar{U}_n) \right\|_2^2\cdot \left\|  \gamma(\bar{U}_n)  \right\|_2^2 h.
\end{equation*}
Then, it follows from $\|U-V\|_2^2\leq2\|U\|_2^2+2\|V\|_2^2$ and \eqref{wdxtj2} that
\begin{equation*}
 \left\| \bar{U}_n-\theta h \alpha(\bar{U}_{n}) \right\|_2^2\leq 2\left\| \bar{U}_n\right\|_2^2+2 \theta^2h^2\left\| \alpha(\bar{U}_n)\right\|_2^2\leq(2+2\theta^2h^2\bar{K})\left\| \bar{U}_n\right\|_2^2,
 \end{equation*}
which implies
 \begin{equation}\label{th6.3}
\left\| \bar{U}_n\right\|_2^2\geq \frac{1}{2+2\theta^2h^2\bar{K}} \left\| \bar{U}_n-\theta h \alpha(\bar{U}_{n}) \right\|_2^2.
 \end{equation}

Now denote $f_n:=\left\|\bar{U}_{n}-\theta h \alpha(\bar{U}_{n}) \right\|_2^2$. 
Let $\theta \in [0, 1/2),$ for any $0< h \leq \rho< \frac{2\bar{L}}{(1-2\theta)\bar{K}},$ combining \eqref{eq:theta2}, \eqref{th6.2}, and \eqref{th6.3} leads to 
 \begin{equation}\label{th6.4}
f_{n+1}\leq \left(1+\frac{((1-2\theta)\bar{K}h -2\bar{L})h}{2+ 2\theta^2h^2\bar{K}} \right) f_n \leq(1-C_{\ref{12wdx}}h)f_n,
\end{equation}
where $C_{\ref{12wdx}}:=\frac{(2\theta-1)\bar{K} \rho+2\bar{L}}{2+2\theta^2h^2\bar{K}}>0$.
Applying $(1+x)\leq e^x$ to \eqref{th6.4} and iterating $n$ times, we obtain
\begin{equation*}\label{th6.5}
f_n \leq e^{-C_{\ref{12wdx}}nh}f_0.
\end{equation*}
Therefore, combining \eqref{th6.1}, we have
\begin{equation*}\label{eq:th6.2}
\begin{split}
 \left\|\bar{U}_{n}\right\|_2^2 
& \leq  \frac{1}{\left(1+2 L^{\prime}\theta h\right)}e^{-C_{\ref{12wdx}}nh}f_0 \leq e^{-C_{\ref{12wdx}}nh} f_0,
\end{split}
\end{equation*}
which completes our first statement.

For $\theta \in [1/2, 1]$, note that $ (1-2\theta) h^2 \| \alpha(\bar{U}_n) \|_2^2\leq0$ for any $h>0.$ Hence, \eqref{th6.2} reduces to 
\begin{equation}
\left\| \bar{U}_n+(1-\theta)\alpha(\bar{U}_{n})  h+\beta (\bar{U}_n)  \Delta W_n\gamma(\bar{U}_n) \right\|_2^2  \leq \left\| \bar{U}_n-\theta h \alpha(\bar{U}_{n}) \right\|_2^2 -2\bar{L}h\left\| \bar{U}_n\right\|_2^2.
\end{equation}
Similar to \eqref{th6.4}, we can obtain 
\begin{equation}
f_{n+1}\leq \left(1-\frac{2\bar{L}h}{2+ 2\theta^2h^2\bar{K}} \right) f_n =(1-C'_{\ref{12wdx}}h)f_n,
\end{equation}
where $C'_{\ref{12wdx}}:=\frac{2\bar{L}}{2+2\theta^2h^2\bar{K}}>0$.
Hence, for any given step size $h>0,$ we have 
\begin{equation*}
f_n \leq e^{-C'_{\ref{12wdx}}nh}f_0,
\end{equation*}
which completes our proof.
\end{proof}

The free STM with $\theta=1$ is called a free analogue of backward Euler method (free BEM). Our result shows that free BEM has better stability properties than the free EM. 

\section{Numerical experiment}\label{sec:experiment}

This section aims to present several numerical experiments to illustrate the convergence order and stability of the considered numerical methods. Note that the practical implementation of free STMs is to consider $N \times N$ matrices for large $N,$ and the accuracy of the numerical solutions is assessed by comparing their spectrum distribution approximations with those of the exact solutions. We refer to \cite[Section 5]{schluechtermannNumericalSolutionFree2022a} for more details. 

Note that the free Brownian motion can be regarded as the large $ N $ limit of the $ N\times N $  self-adjoint  Gaussian random matrix \cite{Kemp2016, schluechtermannNumericalSolutionFree2022a}. We take into account $ N $-dimensional random matrices $W^{(N)}:=(W_{ij} )$ with i.i.d. entries, i.e., 
$\{W_{ij}: i, j=1, \ldots, N\}$ are independent standard Gaussian random variables, i.e., $ W_{ij}\sim N(0, 1).$ Our algorithm is given as follows:
	
{\it Step 1}--Giving a partition of $ [0, T] $ into $ P=2^l$ $(l\in\mathbb{N} )$ intervals with constant step size $ h=T/P,$ and usually the step size $h$ is taken to be sufficiently small.
	
{\it Step 2}--Generating $ M\in \mathbb{N} $ different free Brownian motion paths where the increments $ \Delta W_n=\sqrt{\frac{h}{2N}} [W^{(N)}+(W^{(N)})^T]$ $(n=0,1, \ldots, P-1).$
	
{\it Step 3}--Finding the numerical solutions $\{\bar{U}_{n+1}^{(N)}\}_{n=0}^{P-1}$ by iterating the following matrix-valued equations:
\begin{equation*}\label{nonlineareq}
\bar{U}_{n+1}^{(N)} -\theta\alpha(\bar{U}_{n+1}^{(N)})  h= \bar{U}_{n}^{(N)} +(1-\theta)\alpha(\bar{U}_{n}^{(N)})  h+\beta(\bar{U}_{n}^{(N)})  \Delta W_{n} \gamma(\bar{U}_{n}^{(N)}).
\end{equation*}
If $\alpha$ is a linear function, our numerical methods are ``pseudo-implicit" and can be solved explicitly. If $\alpha$ is nonlinear, only a few cases are applicable; for instance, $\alpha$ is a polynomial \cite{Kratz1987NumericalSO,Dennis1978Algorithms,Davis1981Quadratic,PORCELLI2023349}.

{\it Step 4}--Choosing five different constants $ R<P $ and repeating {\it Step 1-3} to obtain the numerical solutions $ \bar{U}_{P/R}^{(N)} $ by letting $ h_R=R\times h $ instead of $ h $ and $ \Delta W_n^R=\sum_{i=(n-1)R}^{nR}\Delta W_i $ $ (n=0,1,\ldots,P/R-1) $ instead of $ \Delta W_n $ $ (n=0,1, \ldots, P-1).$

{\it Step 5}--Calculating the order of strong convergence according to \eqref{slj}. More precisely, if \eqref{slj} holds, then, taking logs, we have 
\begin{equation*}
\log \left\|\bar{U}_{P}^{(N)} -\bar{U}_{P/R}^{(N)} \right\|_2  \approx \log C+ \frac{1}{2}\log h. 
\end{equation*} 
We can plot our approximation to $ \|\bar{U}_{P}^{(N)} -\bar{U}_{P/R}^{(N)} \|_2 $ against $ h $ on a log-log scale and the slop is the order of strong convergence.


\subsection{Free Ornstein-Uhlenbeck Equation}
Consider the free Ornstein-Uhlenbeck (free OU) equation: 
\begin{equation}\label{ou}
{\rm d} U_t =\mu U_t  {\rm dt} +\sigma {\rm dW_t}, \; \mu, \sigma \in \mathbb{R}. 
\end{equation}
Suppose that the initial value $U_0= 0,$ it was known that the distribution of the exact solution to the free OU equation fulfills the semi-circular law \cite{karginFreeStochasticDifferential2011}.

Figure \ref{out} intuitively shows that the spectrum distribution of the numerical solution obtained by the free BEM. The data was generated using $h=2^{-10}$ as the step size (the following examples are treated in a similar manner). When $N$ tends to be large, the pattern in Figure \ref{out} tends to be a semi-circle (the radius is approximately equal to $1,$ which coincides with the theoretical result 0.9908 \cite{karginFreeStochasticDifferential2011}).  
\begin{figure}[!htbp]
	\centering
	\begin{minipage}{0.32\linewidth}
		\vspace{3pt}
		\centerline{\includegraphics[width=\textwidth]{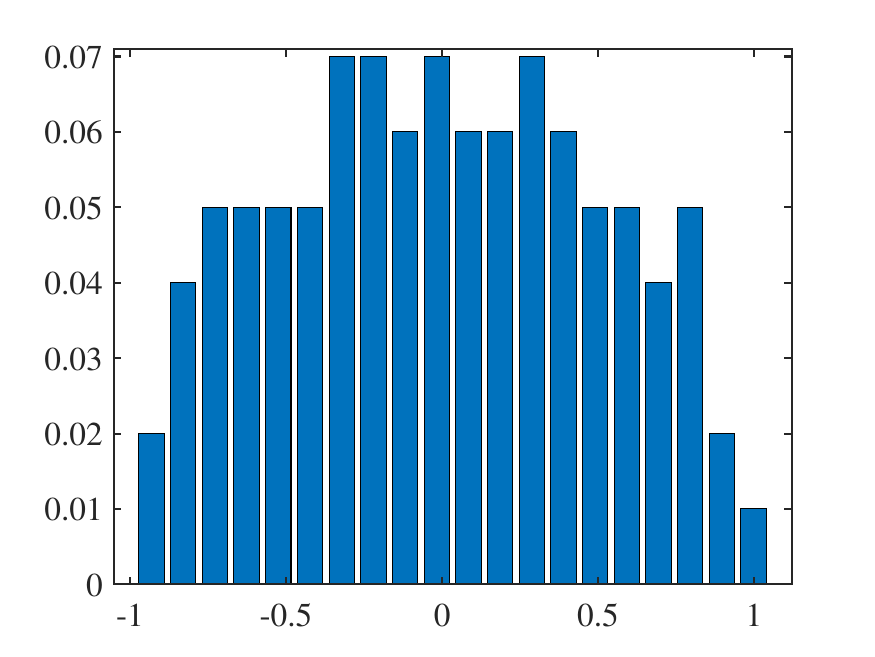}}
		\centerline{{\scriptsize $T=1, N=100$}}
	\end{minipage}
	\begin{minipage}{0.32\linewidth}
		\vspace{3pt}
		\centerline{\includegraphics[width=\textwidth]{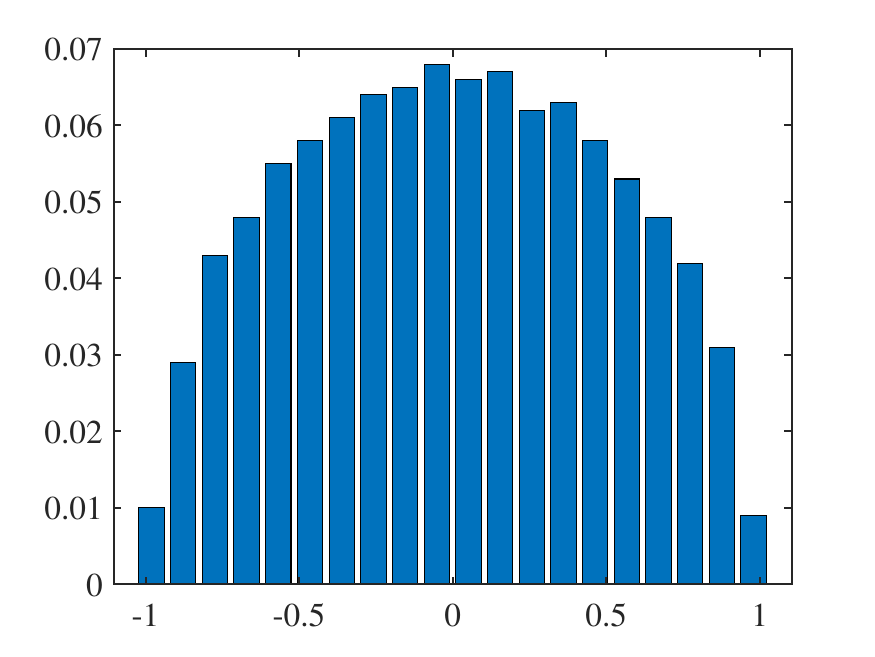}}
		\centerline{{\scriptsize $T=1,N=1000$}}
	\end{minipage}
	\caption{The spectrum distribution of the numerical solution $ \bar{U}_P^{(N)} $  of \eqref{ou}  generated by free BEM with $ \mu=-2$ and $\sigma=1$.
	}\label{out}
\end{figure}

For showing the strong convergence order, the unavailable exact solution is identified with a numerical approximation generated by the corresponding free STMs with a fine step size $ h =2^{-14} $. The expectations are approximated by the Monte Carlo simulation with $M=250$ different free Brownian motion paths. Numerical approximations to \eqref{ou}  are generated by the free STMs with $N=10$ and five different equidistant step sizes $ 2^{-11}, 2^{-10}, 2^{-9}, 2^{-8}, 2^{-7}$, i.e., $ h_R =R \times h $ with $ R=2^3, 2^4, 2^5, 2^6, 2^7 $ (The following examples are handled similarly). See Figure \ref{ous} for the strong convergence order. The order is close to $1$ and it does not conflict with the expected order $1/2$; however, we do not know how to explain it in theory.

\begin{figure}[!htbp]
	\centering
	\centering
	\begin{minipage}{0.32\linewidth}
		\centering
		\vspace{3pt}
		\centerline{\includegraphics[width=\textwidth]{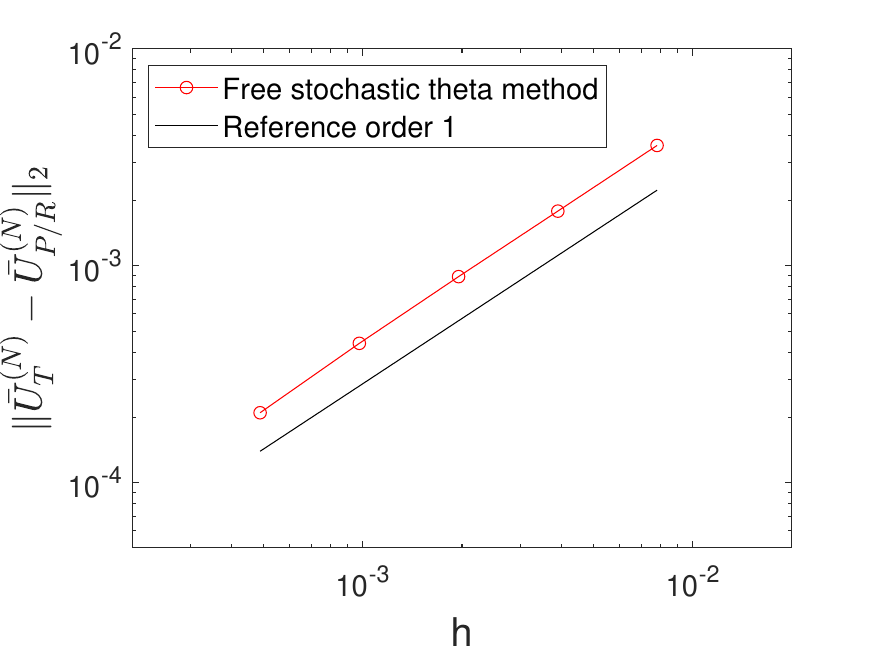}}
		\centerline{{\scriptsize$ \theta =0$}}
	\end{minipage}
	\begin{minipage}{0.32\linewidth}
		\vspace{3pt}
		\centerline{\includegraphics[width=\textwidth]{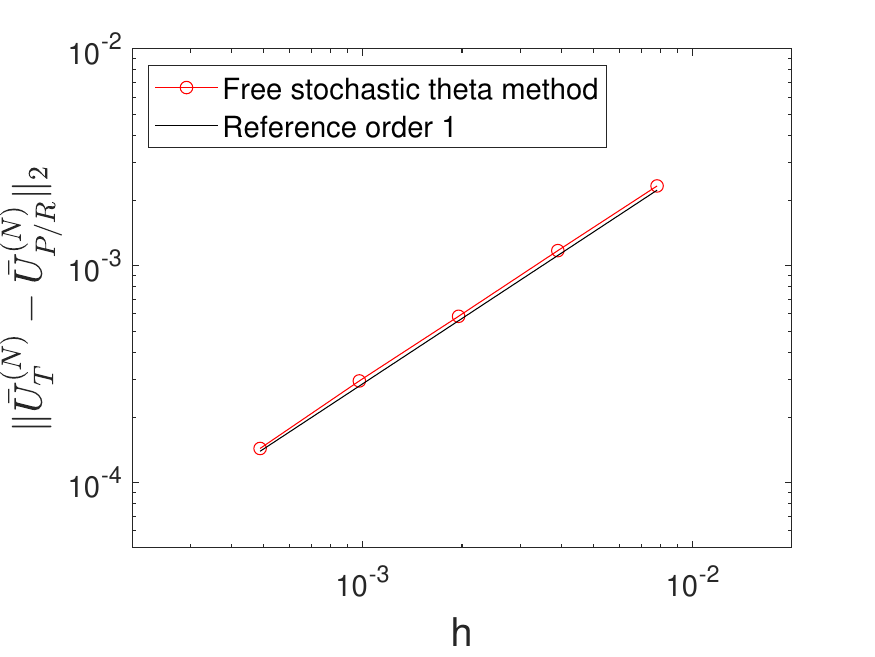}}
		\centerline{{\scriptsize$ \theta =0.5$}}
	\end{minipage}
	\begin{minipage}{0.32\linewidth}
		\vspace{3pt}
		\centerline{\includegraphics[width=\textwidth]{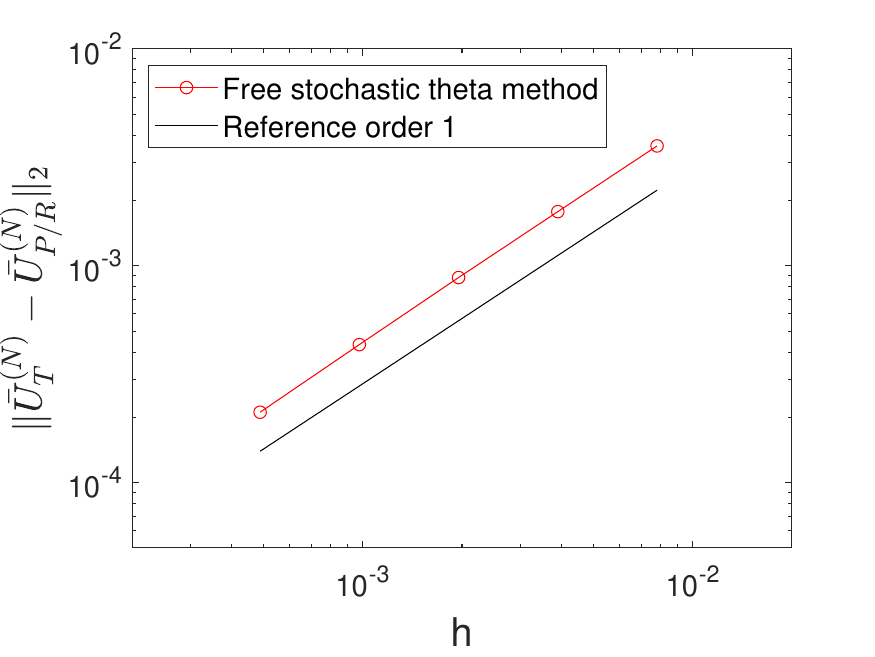}}
		\centerline{{\scriptsize$ \theta =1$}}
	\end{minipage}
	\caption{Strong convergence order simulations for the free STMs applied to \eqref{ou} with $\mu=-2$ and $\sigma=1$.}\label{ous}
\end{figure}

Since the free OU equation \eqref{ou} does not have a trivial solution \cite{karginFreeStochasticDifferential2011}, it motivates us to consider the exponential stability in mean square of the perturbation solution (we refer to Appendix \ref{sec:appendix} for the definition and the theoretical result). 
 Let $ \bar{U}^{(N)}_n $ and $ \bar{V}^{(N)}_n$ be the numerical solutions that are generated by two different initial values $\bar{U}^{(N)}_0 =E$ and $ \bar{V}^{(N)}_0 =2E,$ respectively, where $E$ is the $N \times N$ identity matrix. Let $\mu<0$. For $\theta \in [0, 1/2),$ set $h_{\max}:= \frac{-2}{(1-2\theta)\mu}$, the free STMs are exponentially stable in mean square for any step size $0<h<h_{\max}.$ And for $\theta \in [1/2, 1],$ the free STMs are exponentially stable in mean square for any given step size $h>0$ (see Theorem \ref{12wdxrd} in Appendix \ref{sec:appendix}).

We choose $N=100, \mu=-4,$ and $\sigma=1$. Figure \ref{ouw} demonstrates the exponential stability of the free STMs with $h<h_{\max}$ for different $\theta=0,0.25,0.4$ and the free STMs with different step sizes ($ h=0.25,0.5,1,2,4$) for $\theta=0.5,1$. On the premise of inheriting the exponential stability of the original equation, the free EM requires $16.11s$ with $ h = 0.25$ and $T=20$, while the free BEM only needs $1.05s$ with $ h = 4$. Thus, the free BEM is more efficient than the free EM. Figure \ref{ouw} also shows the free STMs with $\theta=0, 0.25, 0.4$ are not stable with large step sizes $h>h_{\max}$. 

\begin{figure}[!htbp]
	\centering
	\begin{minipage}{0.32\linewidth}
		\vspace{3pt}
		\centerline{\includegraphics[width=\textwidth]{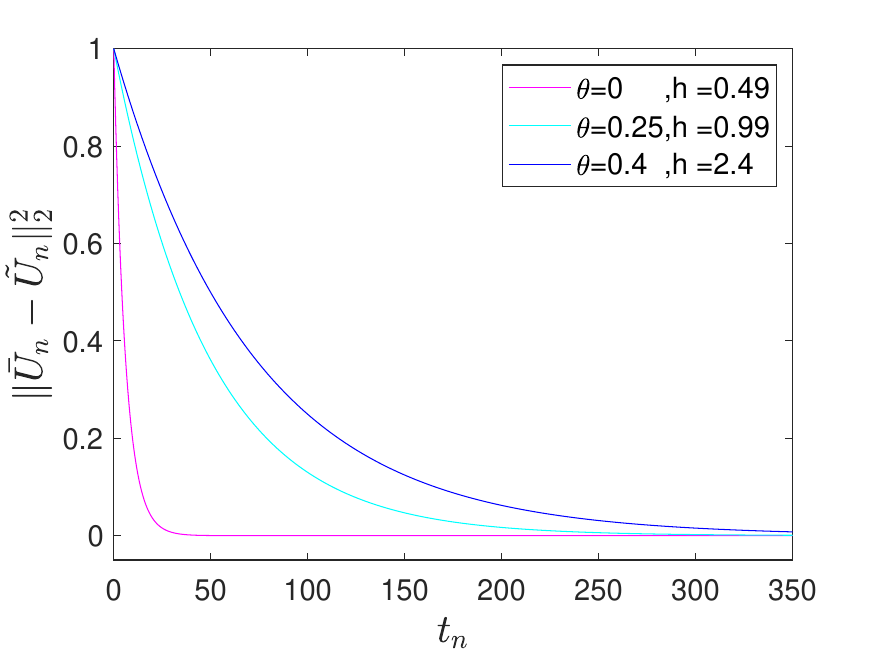}}
		\centerline{{\scriptsize\textit{free STMs with $0\leq\theta<0.5$}}}
	\end{minipage}
	\begin{minipage}{0.32\linewidth}
		\vspace{3pt}
		\centerline{\includegraphics[width=\textwidth]{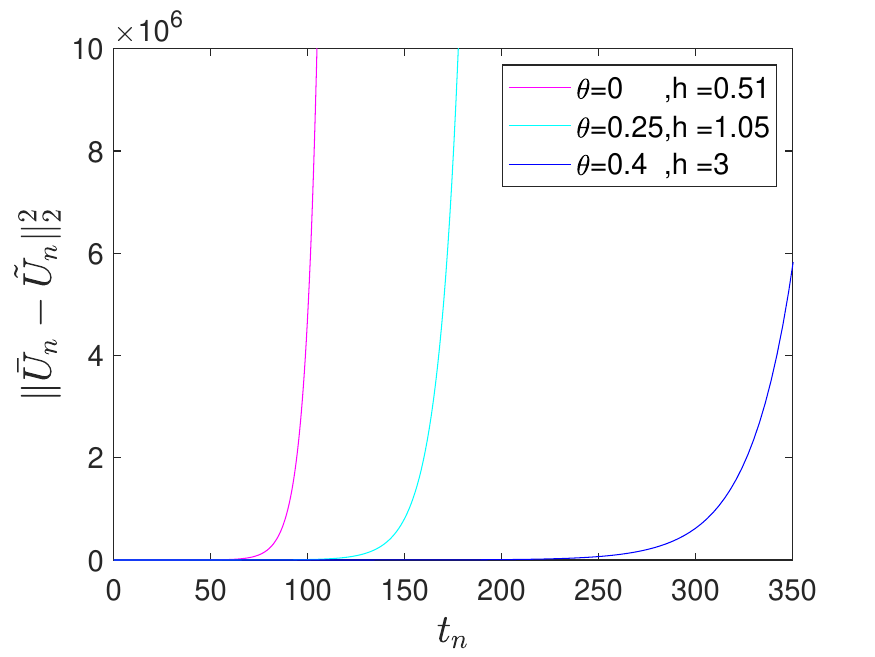}}
		\centerline{{\scriptsize\textit{ free STMs with $0\leq\theta<0.5$}}}
	\end{minipage}
	
	\begin{minipage}{0.32\linewidth}
		\vspace{3pt}
		\centerline{\includegraphics[width=\textwidth]{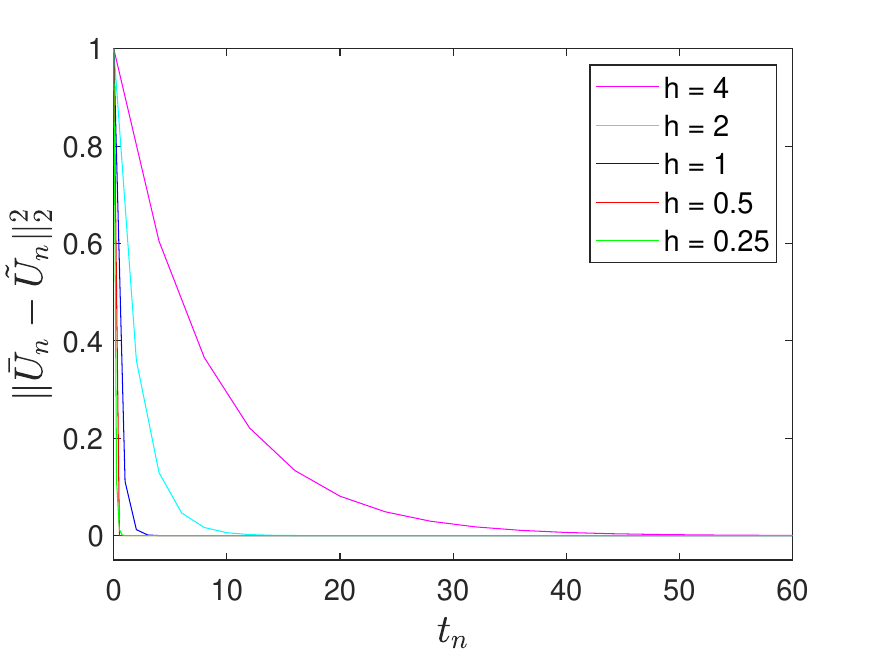}}
		\centerline{{\scriptsize\textit{ free STM with $\theta=0.5$}}}
	\end{minipage}
	\begin{minipage}{0.32\linewidth}
		\vspace{3pt}
		\centerline{\includegraphics[width=\textwidth]{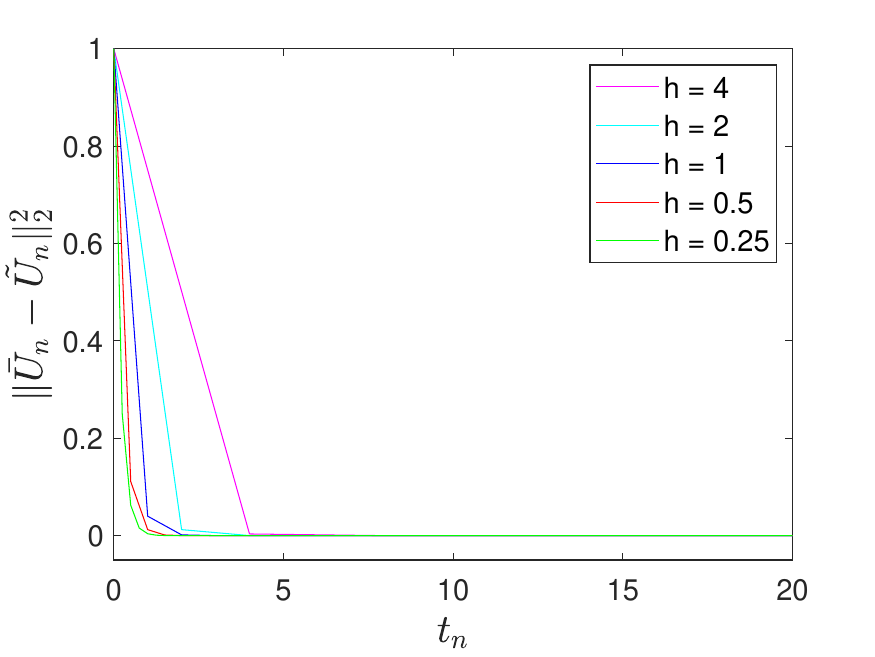}}
		\centerline{{\scriptsize\textit{ free BEM}}}
	\end{minipage}
	\caption{The exponential stability in mean square of free STMs applied to \eqref{ou}.}\label{ouw}
\end{figure}

\subsection{Free Geometric Brownian Motion I }
Consider the free Geometric Brownian Motion (free GBM) I
\begin{equation}\label{geo1}
{\rm d} U_t=\mu U_t {\rm dt}  + \sqrt{U_t} {\rm d} W_t \sqrt{U_t}, \; \mu \in \mathbb{R}. 
\end{equation} 
Suppose that the initial value is $U_0={\bf 1},$ where ${\bf 1}$ is the unit operator in $\mathscr{F}$. We refer to \cite{karginFreeStochasticDifferential2011} for the exact solution. We note that $\sqrt{t}$ is not locally operator Lipschitz in the $L^2 (\psi)$ norm; however, our numerical experiment still shows the strong convergence (The same situation appears in Example \ref{section:cir}). Figure \ref{g1t} shows that the ability of the free BEM to reproduce the spectrum distribution of the numerical solution $ \bar{U}_{n}^{(N)} $. Since the density of the spectrum distribution of $ U_t $ is supported on the interval $ [R_1(t),R_2(t)],$ where $ R_{i}(t)=\frac{r_{i}(t)}{r_{i}(t)+1}e^{(\mu -1-r_{i}(t))t} $ and $r_{i}(t)=(-1 \pm \sqrt{1+\frac{4}{t}})/2$ ($i=1,2$) \cite{karginFreeStochasticDifferential2011}, it has been observed from Figure \ref{g1qj} that the evolution of the supporting interval of the spectrum distribution of $ \bar{U}_n^{(N)} $ fits well with the theoretical values. Figure \ref{g1s} shows that the free STMs give errors that decrease proportionally to $h^{1/2}$, as shown in Theorem \ref{thm:converges-stm}.

\begin{figure}[h]
	\centering
	\begin{minipage}{0.32\linewidth}
		\vspace{3pt}
		\centerline{\includegraphics[width=\textwidth]{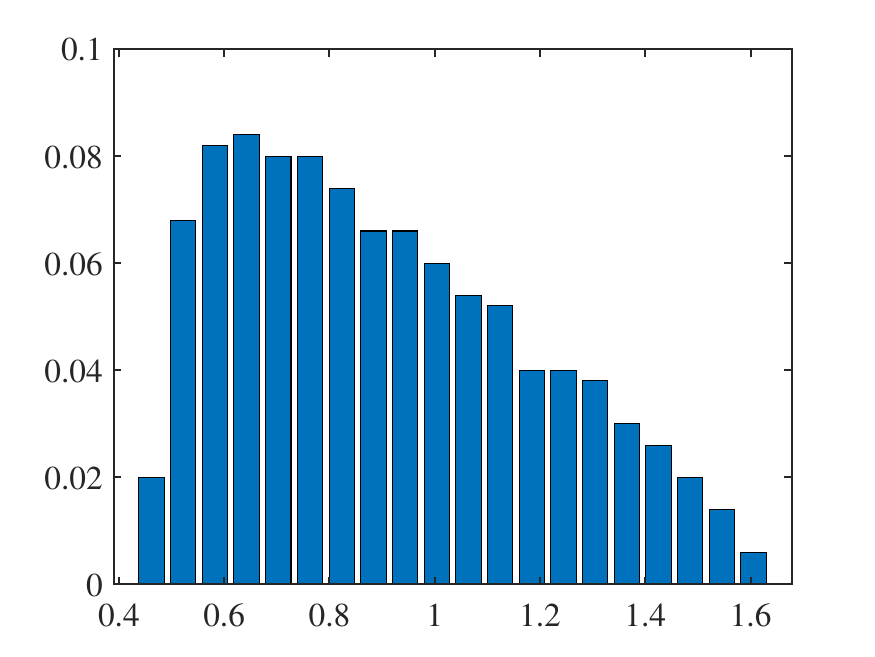}}
		\centerline{{\scriptsize$ T=100h, N=500 $}}
	\end{minipage}
	\begin{minipage}{0.32\linewidth}
		\vspace{3pt}
		\centerline{\includegraphics[width=\textwidth]{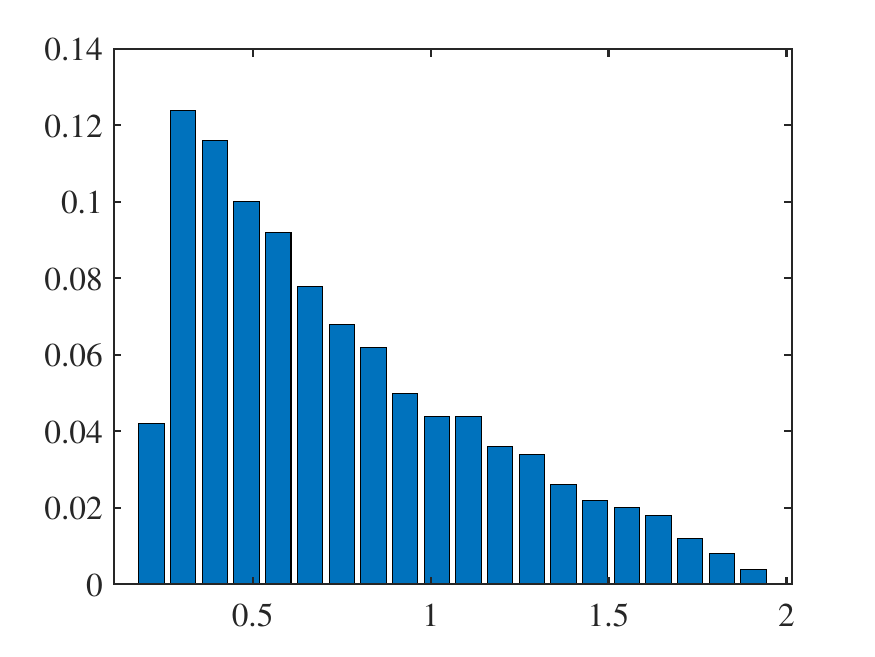}}
		\centerline{{\scriptsize$ T=300h, N=500 $}}
	\end{minipage}
	\begin{minipage}{0.32\linewidth}
		\vspace{3pt}
		\centerline{\includegraphics[width=\textwidth]{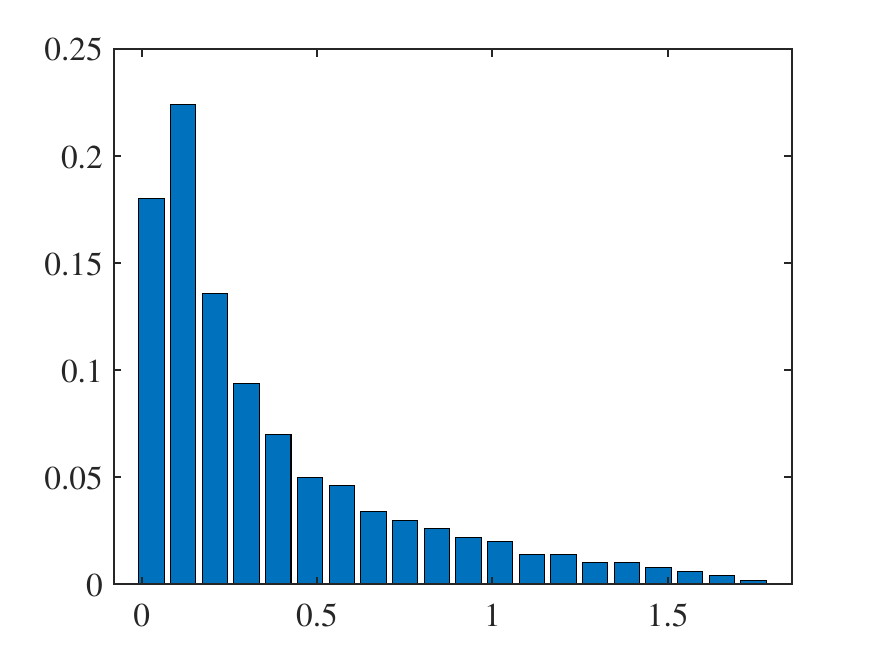}}
		\centerline{{\scriptsize$ T=1024h, N=500 $}}
	\end{minipage}
	\caption{The spectrum distribution of the numerical solution $ \bar{U}_P^{(N)} $ of \eqref{geo1} generated by free BEM with $ \mu=-1$.}\label{g1t}
\end{figure}

\begin{figure}[!htbp]
	\centering
	\includegraphics[width=4.5cm]{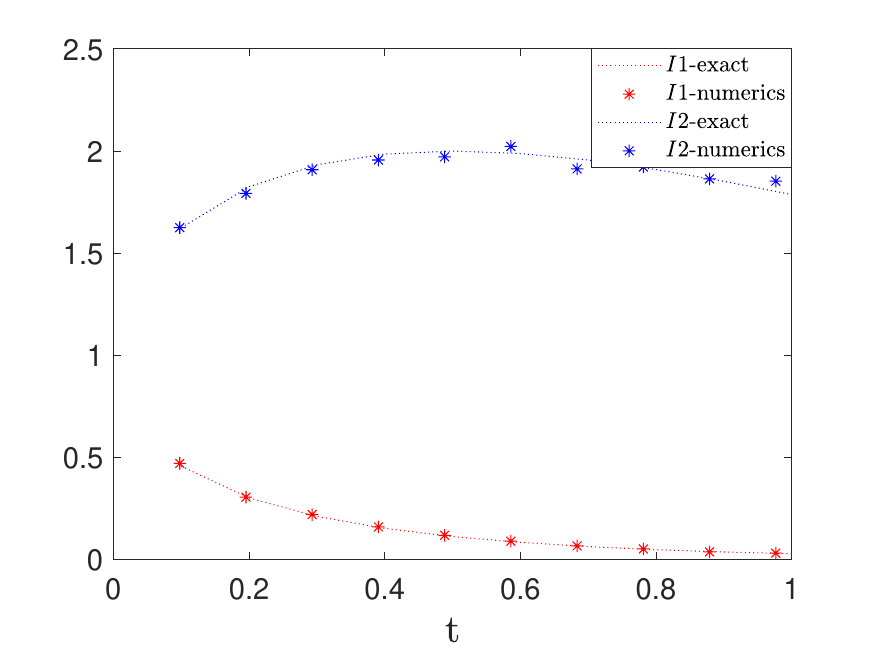}
	\caption{Comparison of boundaries of support interval $ [I1=R_1(t),I2=R_2(t)] $ of the density of spectrum distribution of \eqref{geo1} with $ \mu=-1, h =2^{-10},$ and $N=500$.}\label{g1qj}
\end{figure}

\begin{figure}[!htbp]
	\centering
	\centering
	\begin{minipage}{0.32\linewidth}
		\vspace{3pt}
		\centerline{\includegraphics[width=\textwidth]{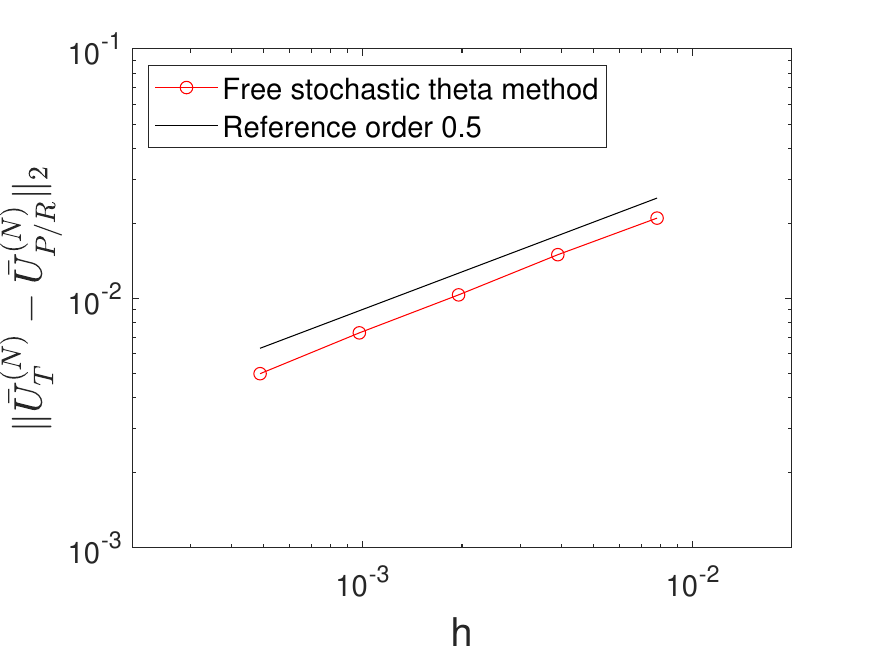}}
		\centerline{{\scriptsize$ \theta =0$}}
	\end{minipage}
	\begin{minipage}{0.32\linewidth}
		\vspace{3pt}
		\centerline{\includegraphics[width=\textwidth]{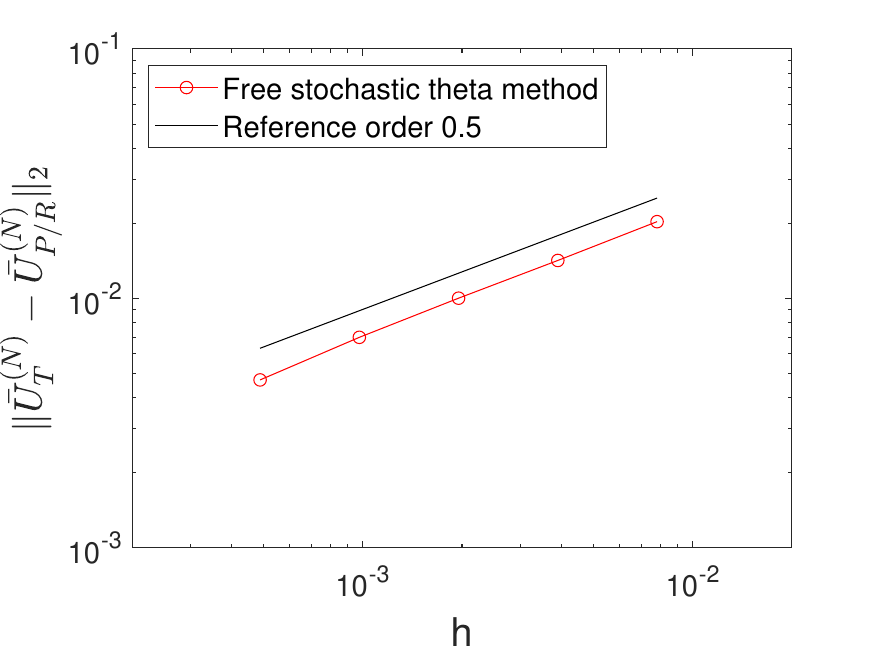}}
		
		\centerline{{\scriptsize$ \theta =0.5$}}
	\end{minipage}
	\begin{minipage}{0.32\linewidth}
		\vspace{3pt}
		\centerline{\includegraphics[width=\textwidth]{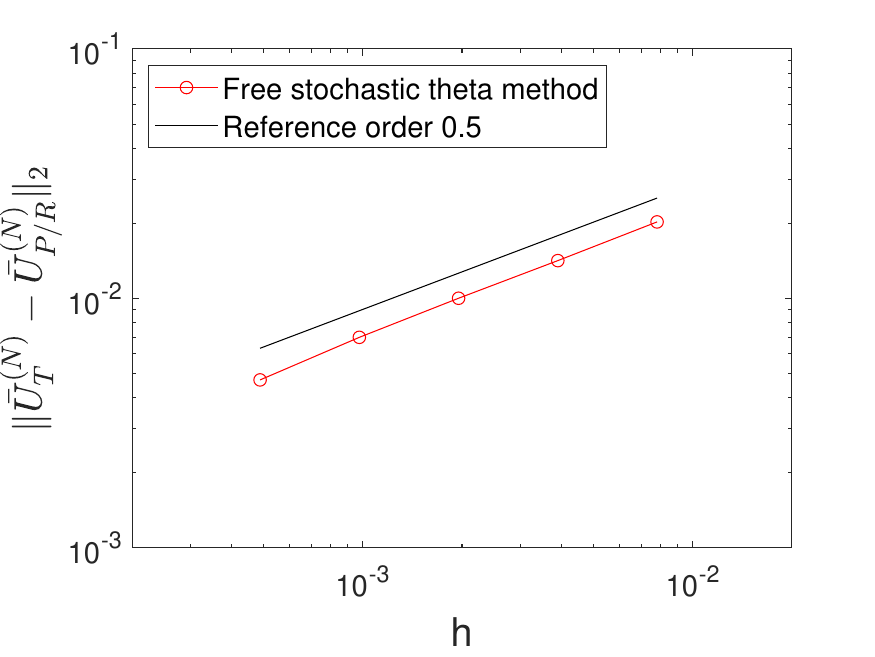}}
		
		\centerline{{\scriptsize$ \theta =1$}}
	\end{minipage}
	\caption{Strong convergence order simulations for the free STMs applied to \eqref{geo1} with $\mu=-1$.}\label{g1s}
\end{figure}

Obviously, the free GBM I has a trivial solution. Thus, by Theorem \ref{12wdx}, we set $\bar{L}=-\mu-1/2, \bar{K}=\mu^2$ and let $\mu<-1/2.$ It appears that the free STMs with $\theta \in [0,1/2)$ are exponentially stable when $h<\frac{-2\mu-1}{(1-2\theta)\mu^2}=:h_{\max}$.  And for $\theta \in [1/2,1],$ the free STMs are exponentially stable for any given $h > 0.$

We choose $N=100$ and $ \mu=-4$. Figure \ref{g1w} shows the exponential stability of the free STMs with $h<h_{\max}$ for different $\theta=0,0.25,0.4$ and the free STMs with different step sizes ($ h=0.25,0.5,1,2,4$) for $\theta=0.5,1$. Under the premise of inheriting the exponential stability in mean square of the original equation, the free EM requires $501.00s$ with $ h=0.25$ and $T=20$, while the free BEM only requires $29.04s$ with $ h=4$. Consequently, the free BEM is more efficient than the free EM. Figure \ref{g1w} also indicates that the free STMs with $\theta=0,0.25,0.4$ are not stable with large step sizes $ h>h_{\max} $.

\begin{figure}[!htbp]
	\centering
	\centering
	\begin{minipage}{0.32\linewidth}
		\vspace{3pt}
		\centerline{\includegraphics[width=\textwidth]{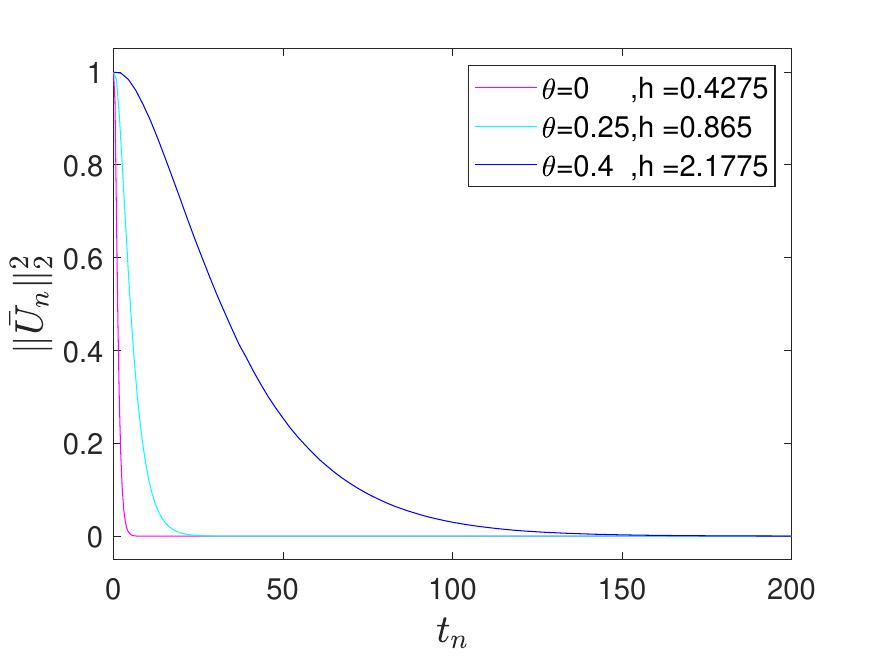}}
		\centerline{{\scriptsize\textit{free STMs with $0\leq\theta<0.5$}}}
	\end{minipage}
	\begin{minipage}{0.32\linewidth}
		\vspace{3pt}
		\centerline{\includegraphics[width=\textwidth]{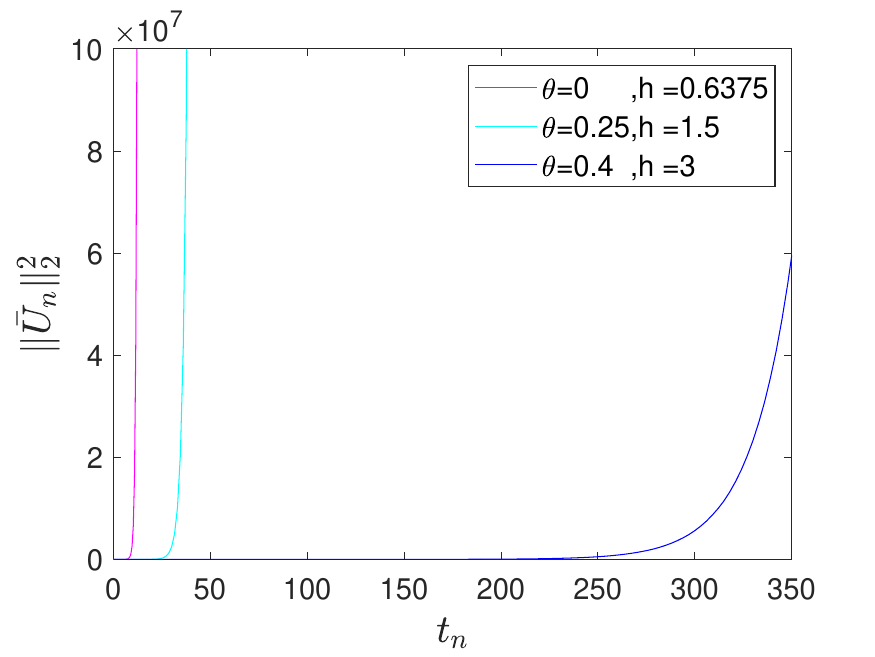}}
		\centerline{{\scriptsize\textit{ free STMs with $0\leq\theta<0.5$}}}
	\end{minipage}
	
	\begin{minipage}{0.32\linewidth}
		\vspace{3pt}
		\centerline{\includegraphics[width=\textwidth]{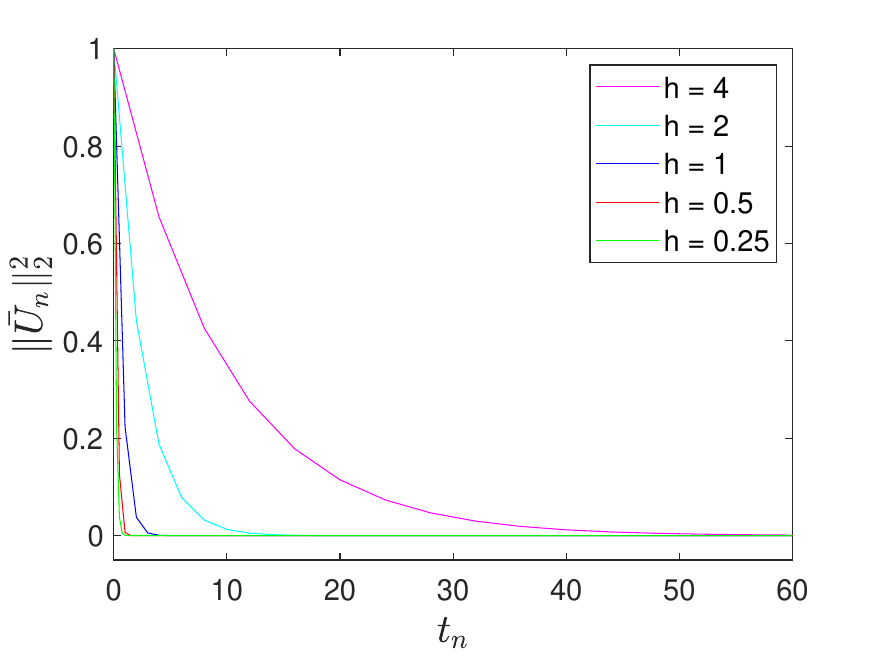}}
		\centerline{{\scriptsize\textit{ free STM with $\theta=0.5$}}}
	\end{minipage}
	\begin{minipage}[h]{0.32\linewidth}
		\centering
		\includegraphics[width=\textwidth]{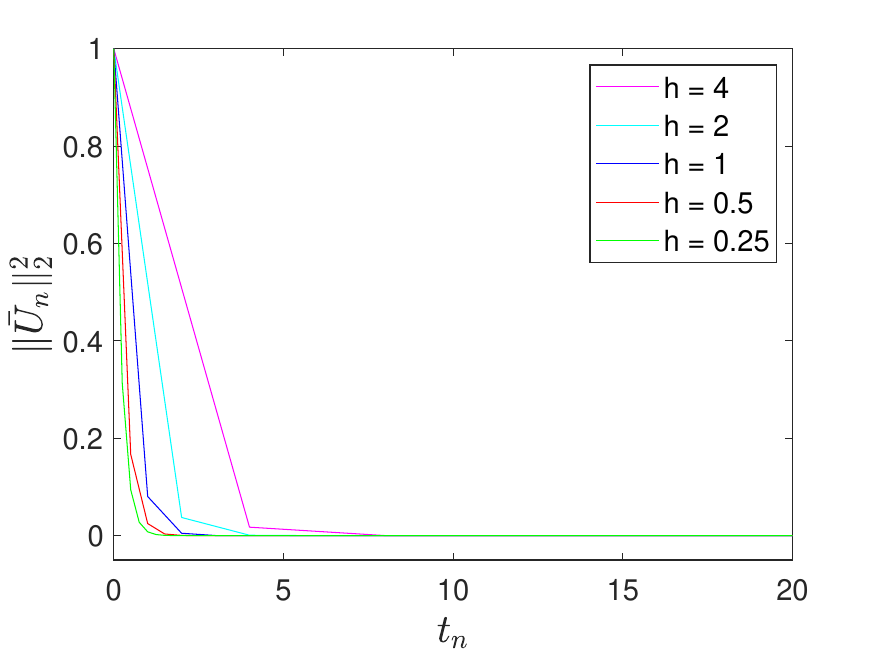}
		\centerline{{\scriptsize\textit{free BEM}}}
	\end{minipage}
	\caption{The exponential stability in mean square of free STMs applied to \eqref{geo1}.}\label{g1w}
\end{figure}

\subsection{Free Geometric Brownian Motion II }

Consider the free GBM II 
\begin{equation}\label{geo}
{\rm d} U_t =\mu U_t {\rm dt} +U_t {\rm dW_t} + {\rm dW_t} U_t, \; \mu \in \mathbb{R}. 
\end{equation} 
Suppose that $U_0 = {\bf 1}$. It is difficult to obtain the explicit solution; however, the expectation and variance of the exact solution are $e^{\mu t}$ and $2e^{2\mu t}(e^{2t}-1)$, respectively (see \cite{karginFreeStochasticDifferential2011}). See Figure \ref{g2t} for the spectrum distribution of the numerical approximation (by using the free BEM). And Table \ref{g2b} shows that the expectation (resp. variance) of the spectrum distribution of the numerical solution tends to the expectation (resp. variance) of the exact solution for sufficiently lager $N$. Figure \ref{g2s} shows that free STMs give errors that decrease proportionally to $h^{1/2}$ as expected.
\begin{figure}[h]
	\centering
	\begin{minipage}{0.32\linewidth}
		\vspace{3pt}
		\centerline{\includegraphics[width=\textwidth]{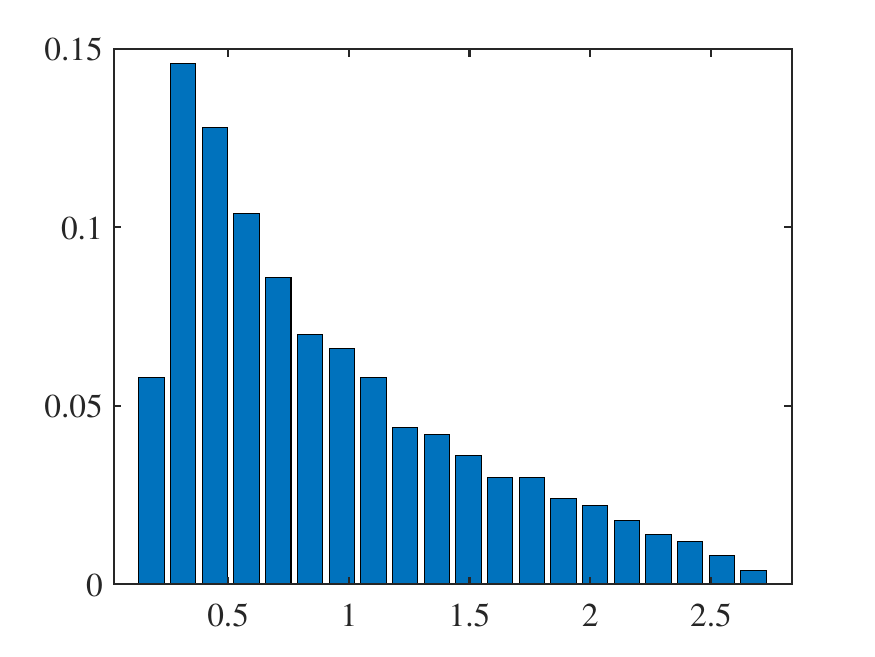}}
		\centerline{{\scriptsize$ T=100h, N=500 $}}
	\end{minipage}
	\begin{minipage}{0.32\linewidth}
		\vspace{3pt}
		\centerline{\includegraphics[width=\textwidth]{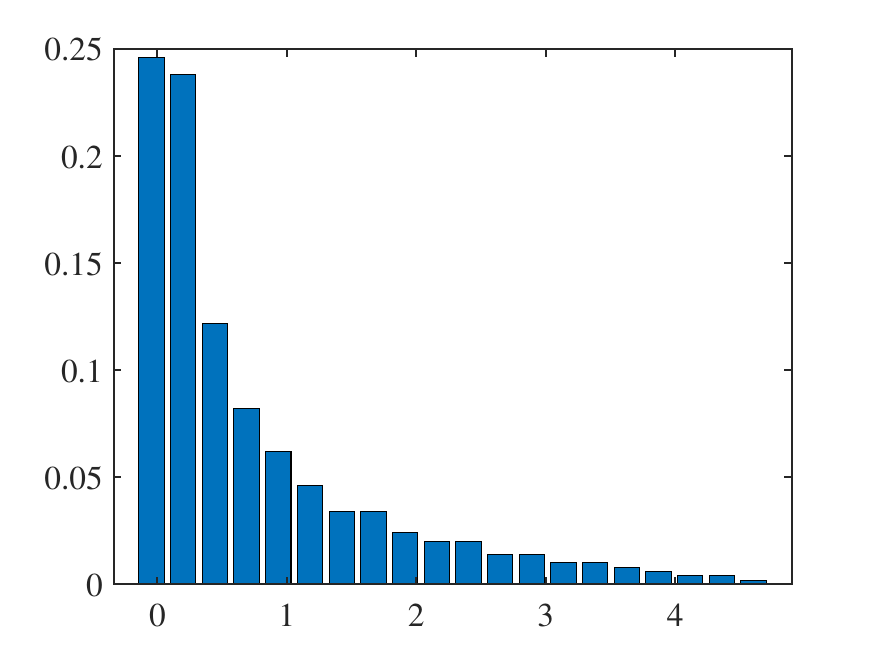}}
		\centerline{{\scriptsize$ T=300h, N=500 $}}
	\end{minipage}
	\begin{minipage}{0.32\linewidth}
		\vspace{3pt}
		\centerline{\includegraphics[width=\textwidth]{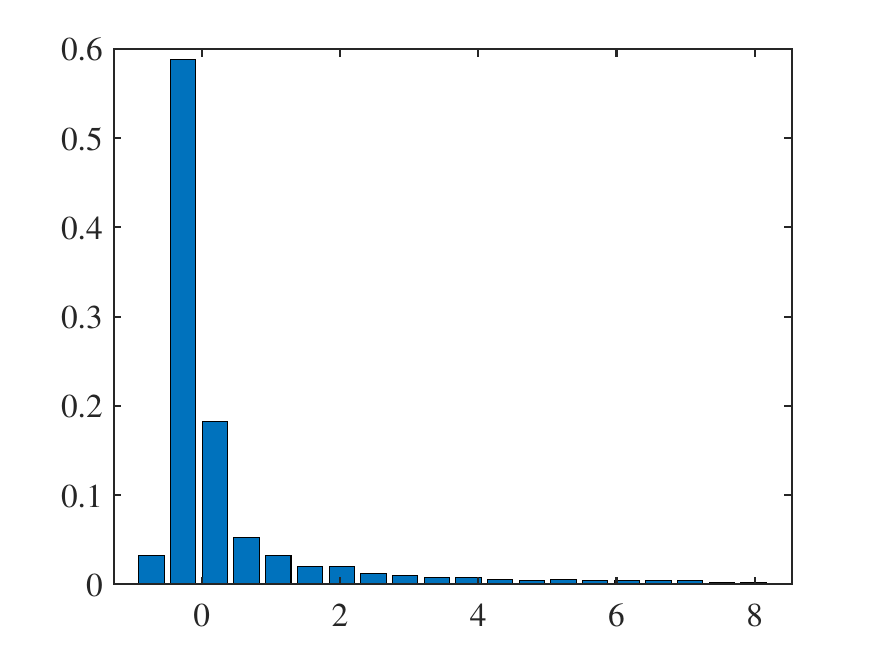}}
		\centerline{{\scriptsize$ T=1024h, N=500 $}}
	\end{minipage}
	\caption{The spectrum distribution of the numerical solution $ \bar{U}_P^{(N)} $ of \eqref{geo} generated by free BEM with $\mu=-1$.}\label{g2t}
\end{figure}

\begin{table}[!htbp]
\caption{The expectation and variance of the solutions to \eqref{geo} with $\mu=-1, h =2^{-10}, $ and $N=500$. \label{g2b}}%
\begin{tabular*}{\columnwidth}{@{\extracolsep\fill}llll@{\extracolsep\fill}}
\toprule
                                          &$T=100h$&$T=300h$&$T=1024h$\\
\midrule
     $\mathbb{E}[tr_N(\bar{U}_P^{(N)})]$   & $0.9094$  & $0.7481$  &$0.3680$\\
    $\psi(U_T)$                     & $0.9070$  & $0.7460$  &$0.3679$\\
\midrule
    $^{1}$ $var_N(\bar{U}_P^{(N)})$ & $0.3564$  & $0.8994$  &$1.7354$\\
    $^{2}$ $var(U_T)$                    & $0.3548$  & $0.8868$  &$1.7293$\\
\bottomrule
\end{tabular*}
\begin{tablenotes}%
\item $^{1}$ $var_N(\bar{U}_P^{(N)}):=\mathbb{E}[tr_N(\bar{U}_P^{(N)})-\mathbb{E}[tr_N(\bar{U}_P^{(N)})]]^2$.
\item $^{2}$ $var(U_T):=\psi|U_T-\psi(U_T)|^2$.
\end{tablenotes}
\end{table}

\begin{figure}[!htbp]
	\centering
	\centering
	\begin{minipage}{0.32\linewidth}
		\vspace{3pt}
		\centerline{\includegraphics[width=\textwidth]{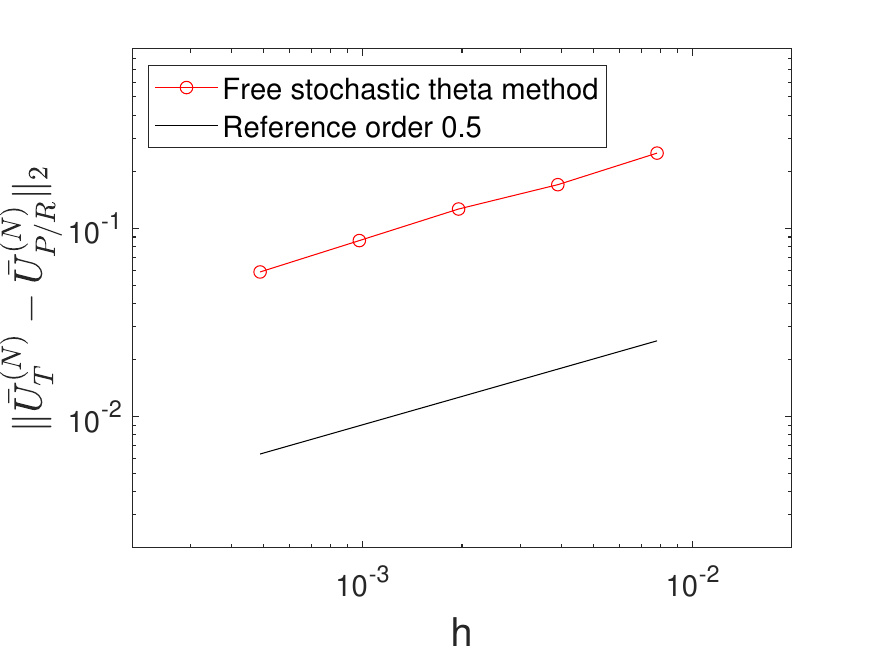}}
		\centerline{{\scriptsize$ \theta =0$}}
	\end{minipage}
	\begin{minipage}{0.32\linewidth}
		\vspace{3pt}
		\centerline{\includegraphics[width=\textwidth]{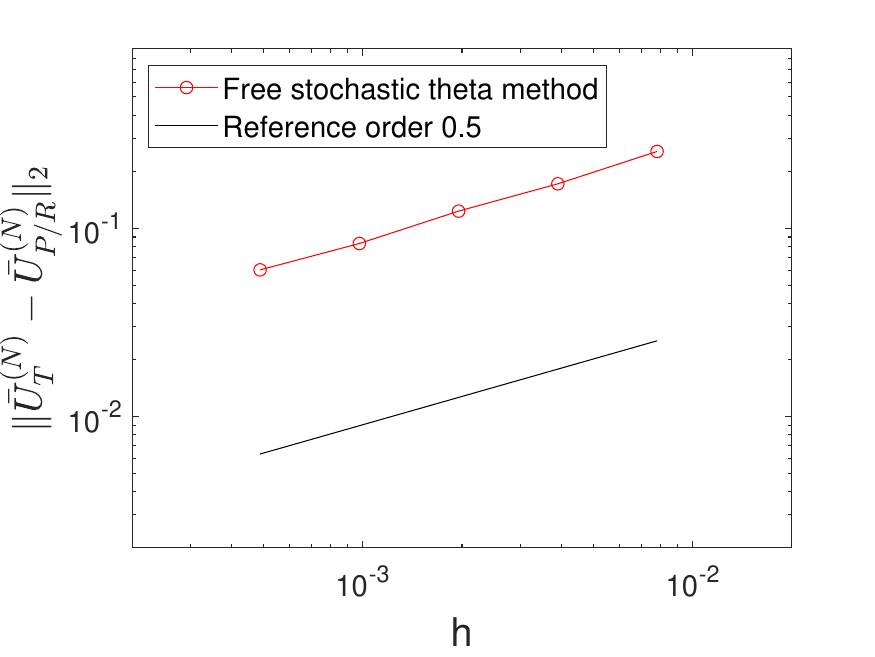}}		
		\centerline{{\scriptsize$ \theta =0.5$}}
	\end{minipage}
	\begin{minipage}{0.32\linewidth}
		\vspace{3pt}
		\centerline{\includegraphics[width=\textwidth]{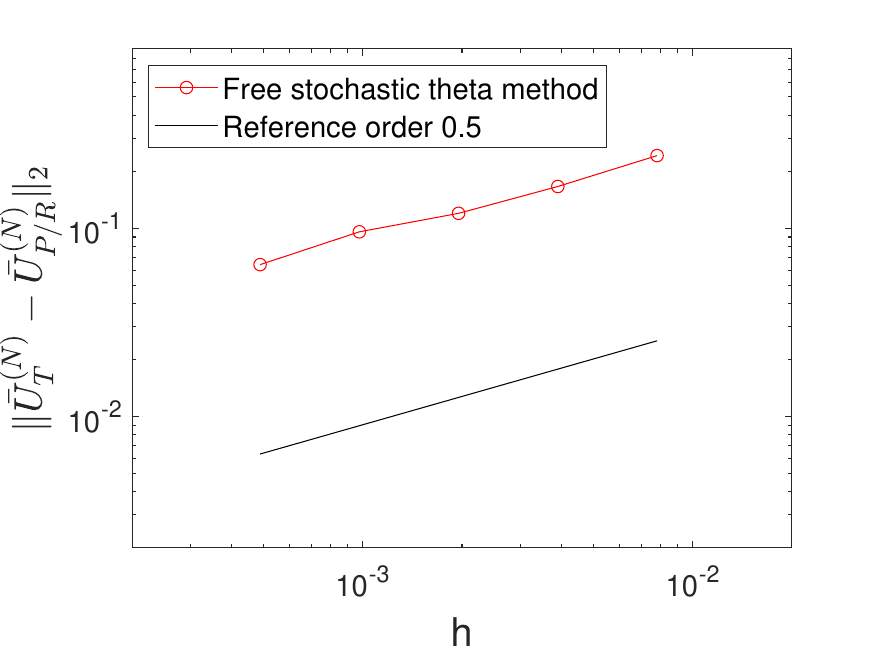}}		
		\centerline{{\scriptsize$ \theta =1$}}
	\end{minipage}
	\caption{Strong convergence order simulations for the free STMs applied to \eqref{geo} with $\mu=-1$.}\label{g2s}
\end{figure}

Evidently, the free GBM II has a trivial solution. According to Theorem \ref{12wdx}, we set $\bar{L}=-\mu-2,\ \bar{K}=\mu^2,$ and let $\mu <-2.$ Then for $\theta \in [0, 1/2),$ the free STMs are exponentially stable when $h<\frac{-2\mu-4}{(1-2\theta)\mu^2}=:h_{\max}$. And for $\theta \in [1/2, 1],$ the free STMs are exponentially stable for any given $h>0.$ 

We choose $ N=100$ and $ \mu=-4$. Figure \ref{g2w} shows the exponential stability in mean square of the free STMs with $h<h_{\max}$ for different $\theta=0,0.25,0.4$ and the free STMs with different step sizes ($ h=0.25,0.5,1,2,4$) for different $\theta=0.5,1$. Figure \ref{g2w} also indicates that the free STMs with $\theta=0,0.25,0.4$ are not stable with large step sizes $h>h_{\max} $. Moreover, on the premise of inheriting the exponential stability of original equation, the free EM requires $35.12s$ with $h=0.24$ and $T=20$, while the free BEM requires only $4.70s$ with $h=4$.
\begin{figure}[!htbp]
			\centering
	\begin{minipage}{0.32\linewidth}
		\vspace{3pt}
		\centerline{\includegraphics[width=\textwidth]{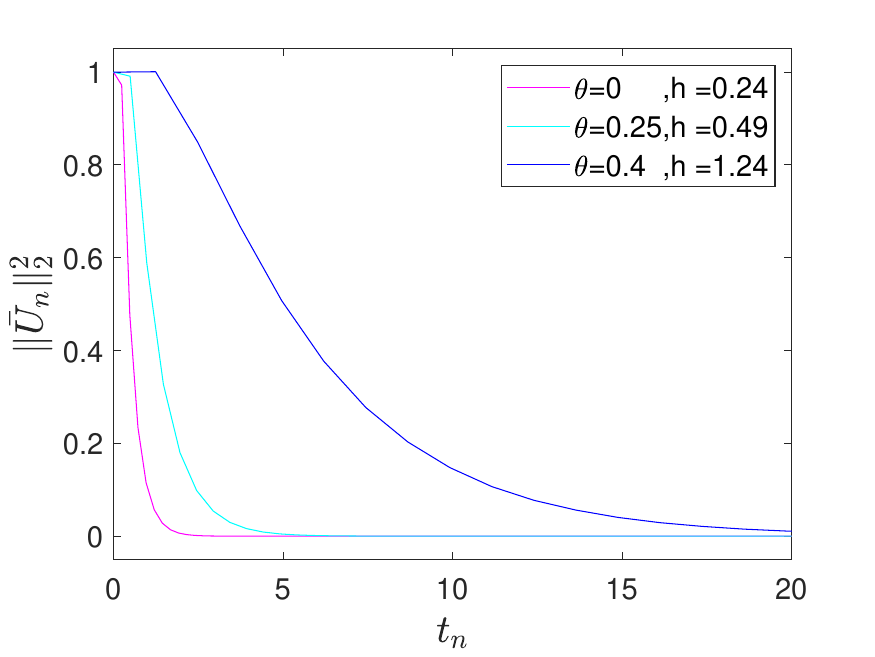}}
		\centerline{{\scriptsize\textit{free STMs with $0\leq\theta<0.5$}}}
	\end{minipage}
	\begin{minipage}{0.32\linewidth}
		\vspace{3pt}
		\centerline{\includegraphics[width=\textwidth]{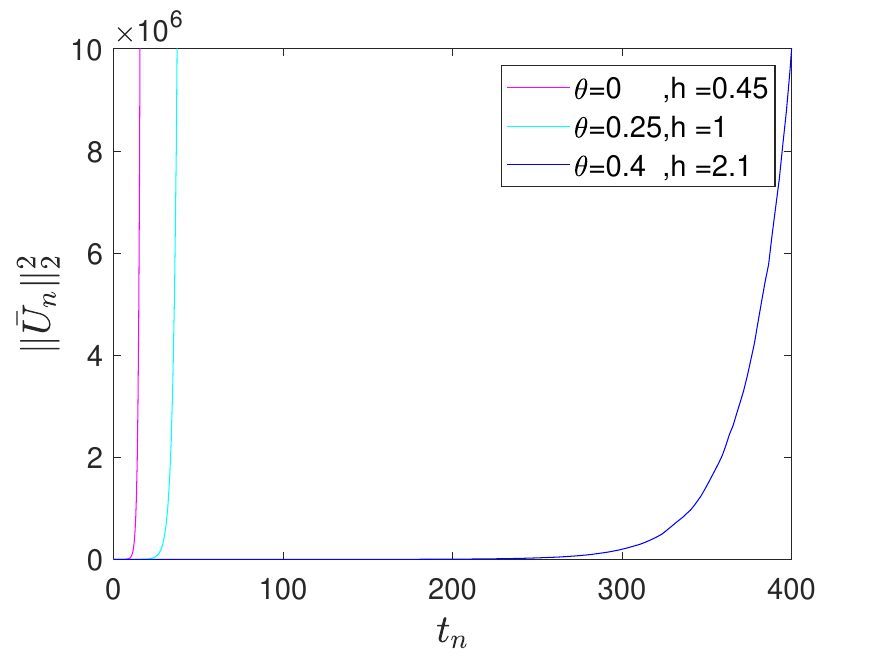}}
		\centerline{{\scriptsize\textit{ free STMs with $0\leq\theta<0.5$}}}
	\end{minipage}
	
	\begin{minipage}{0.32\linewidth}
		\vspace{3pt}
		\centerline{\includegraphics[width=\textwidth]{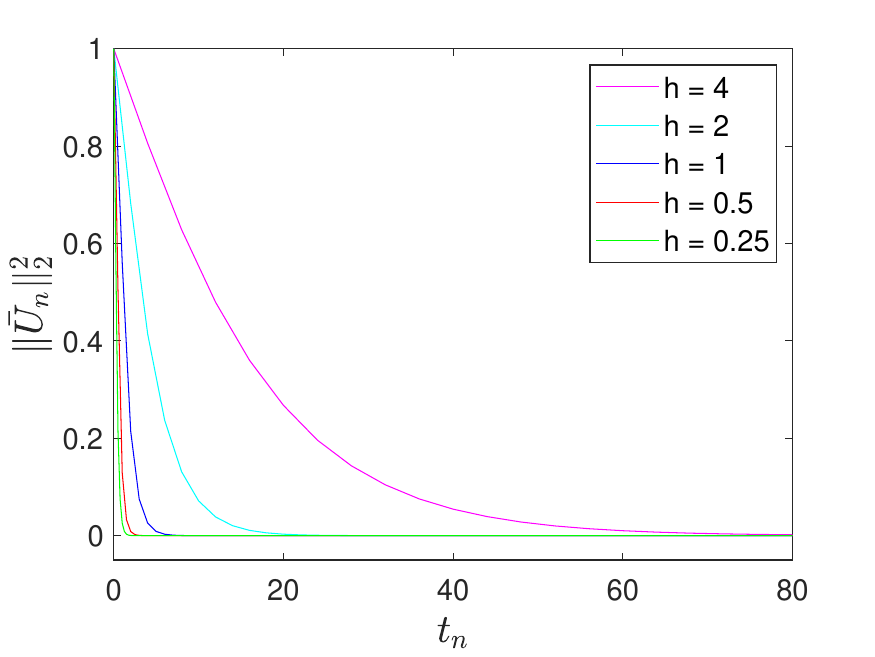}}
		\centerline{{\scriptsize\textit{ free STM with $\theta=0.5$}}}
	\end{minipage}
	\begin{minipage}{0.32\linewidth}
		\vspace{3pt}
		\centerline{\includegraphics[width=\textwidth]{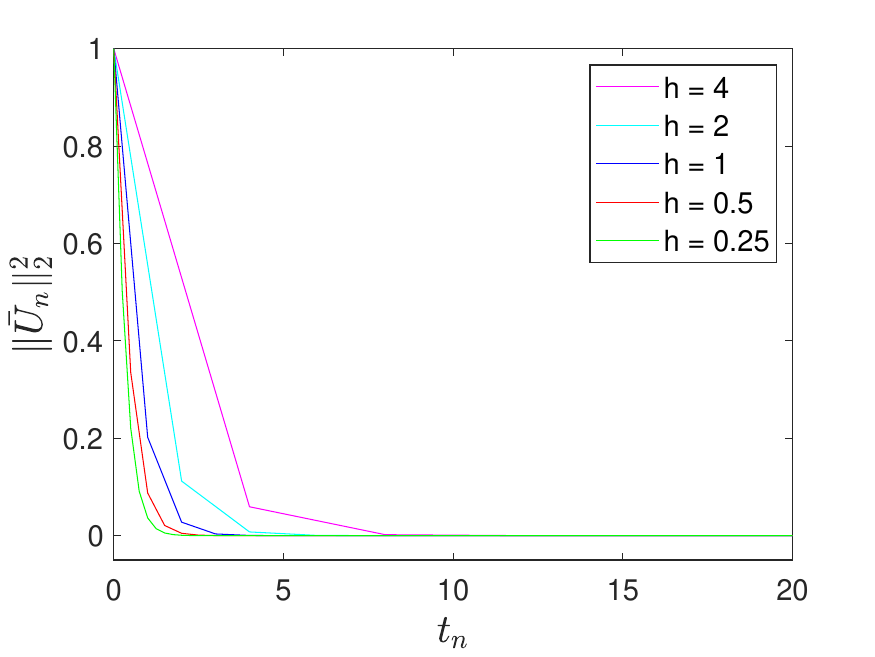}}
		\centerline{{\scriptsize\textit{free BEM}}}
	\end{minipage}
	\caption{The exponential stability in mean square of free STMs applied to \eqref{geo}.}\label{g2w}
\end{figure}

\subsection{Free Cox-Ingersoll-Ross Equation}\label{section:cir}
Consider the free Cox-Ingersoll-Ross (free CIR) equation
\begin{equation}\label{cir}
{\rm d} U_t =(\alpha-\beta U_t ){\rm dt} + \frac{\sigma}{2}\sqrt{U_t} {\rm d} W_t+ \frac{\sigma}{2} {\rm d} W_t \sqrt{U_t}, 
\end{equation} 
where $ \alpha,\beta, \sigma \ge 0 $ and satisfy the Feller condition $ 2\alpha \ge \sigma^2 $. Suppose that $U_0 = {\bf 1}.$ It follows that $\psi(U_t)=\frac{e^{-\beta t}}{\beta}(\beta+\alpha(e^{\beta t}-1))$. Figure \ref{cirt} shows the spectrum distribution of the numerical solution obtained by the free BEM. And Table \ref{cirb} shows that the expectation of the spectrum distribution of the numerical solution tends to the expectation of the exact solution for sufficiently lager $N$. Figure \ref{cirs} shows that the mean square errors decrease at a slope approaching {1/2}, which confirms the predicted convergence result.
\begin{figure}[!htbp]
	\centering
	\centering
	\begin{minipage}{0.32\linewidth}
		\vspace{3pt}
		\centerline{\includegraphics[width=\textwidth]{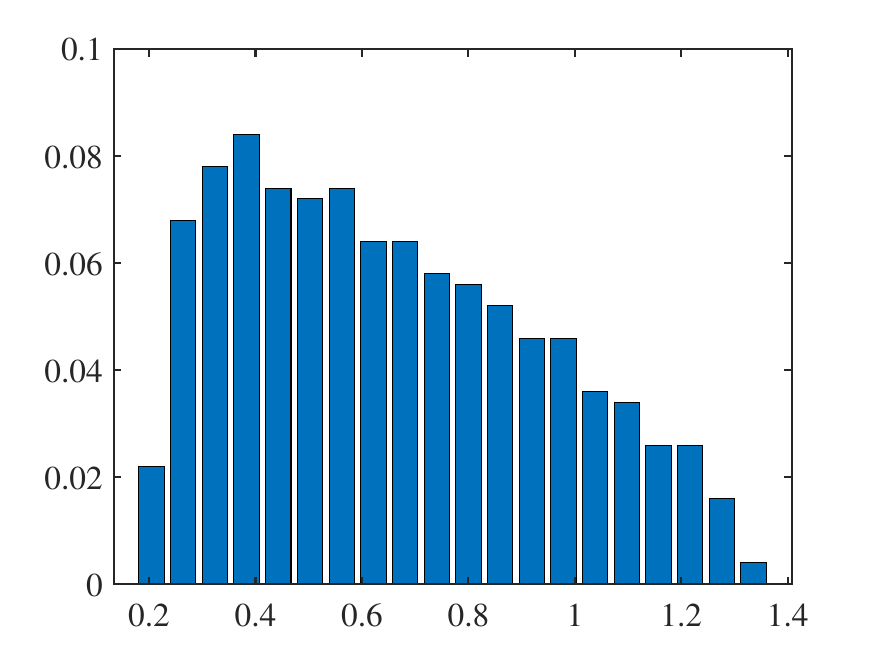}}
		\centerline{{\scriptsize$ T=300h, N=500 $}}
	\end{minipage}
	\begin{minipage}{0.32\linewidth}
		\vspace{3pt}
		\centerline{\includegraphics[width=\textwidth]{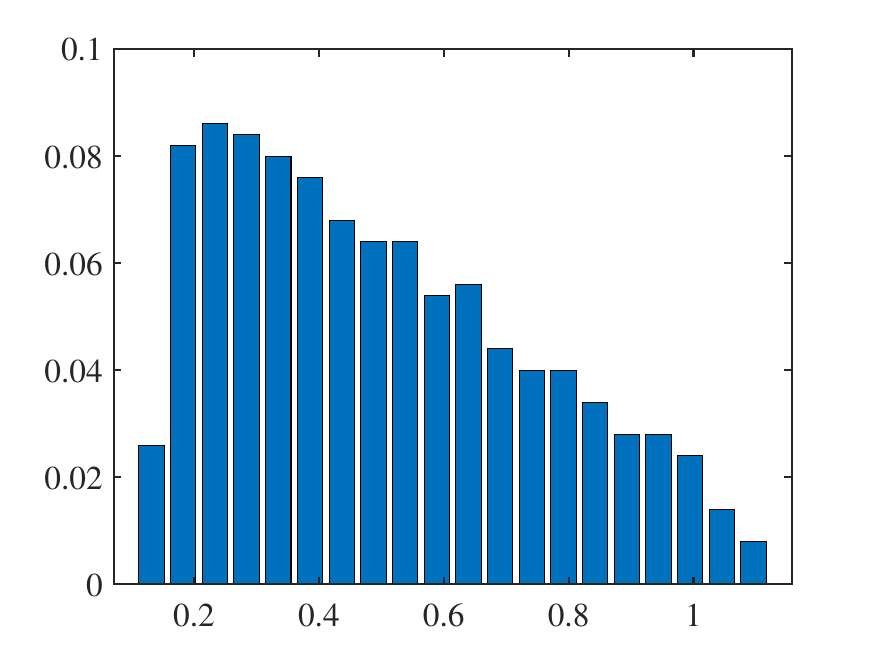}}
		\centerline{{\scriptsize$ T=4000h, N=500 $}}
	\end{minipage}
	\caption{The spectrum distribution of the numerical solution $ \bar{U}_P^{(N)} $ of \eqref{cir} generated by free BEM with $ \alpha=2, \beta=4, $ and $\sigma=1$.}\label{cirt}
\end{figure}

\begin{table}[!htbp]
\caption{The expectation of the solutions to \eqref{cir} with $\alpha=2, \beta=4, \sigma=1, h =2^{-10},$ and $N=500$. \label{cirb}}%
\begin{tabular*}{\columnwidth}{@{\extracolsep\fill}llll@{\extracolsep\fill}}
\toprule
                                             &$T=300h$&$T=1024h$\\

\midrule
    $\mathbb{E}[tr_N(\bar{U}_P^{(N)})]$   & $0.6557$  &$0.5002$\\
    $\psi(U_T)$                      & $0.6549 $ &$0.5000$\\
\bottomrule
\end{tabular*}
\end{table}

\begin{figure}[!htbp]
	\centering
	\centering
	\begin{minipage}{0.32\linewidth}
		\vspace{3pt}
		\centerline{\includegraphics[width=\textwidth]{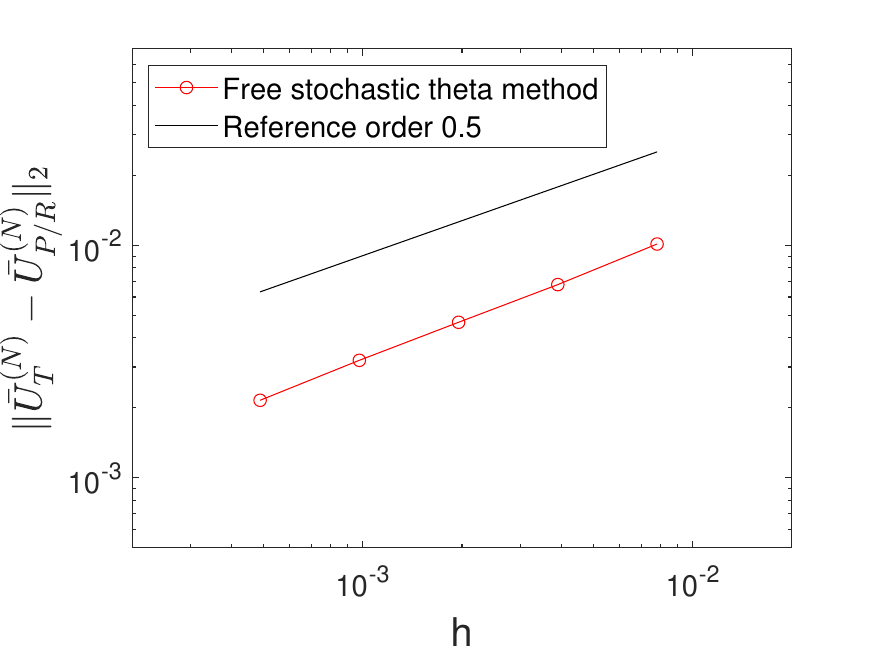}}
		\centerline{{\scriptsize$ \theta =0$}}
	\end{minipage}
	\begin{minipage}{0.32\linewidth}
		\vspace{3pt}
		\centerline{\includegraphics[width=\textwidth]{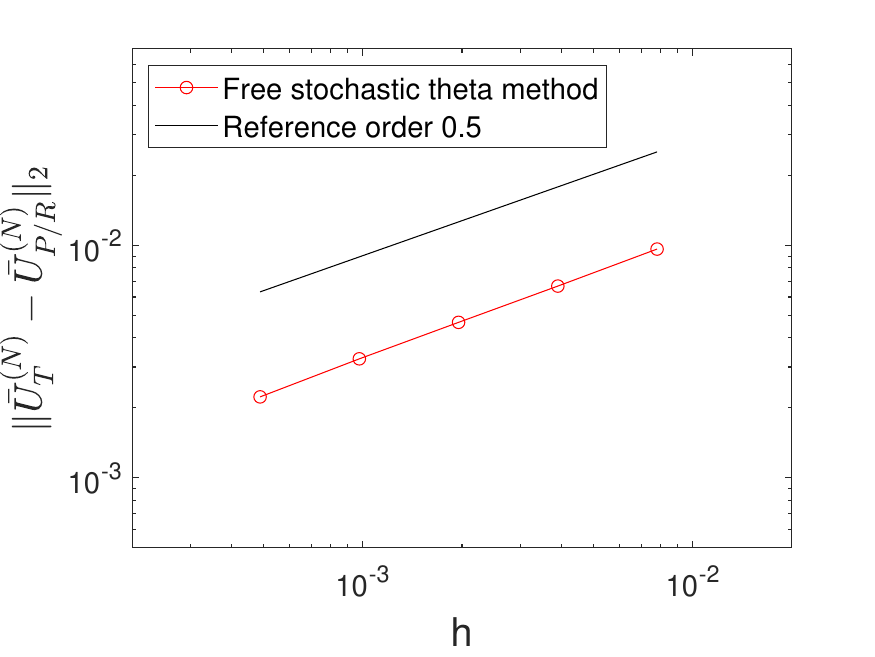}}
		
		\centerline{{\scriptsize$ \theta =0.5$}}
	\end{minipage}
	\begin{minipage}{0.32\linewidth}
		\vspace{3pt}
		\centerline{\includegraphics[width=\textwidth]{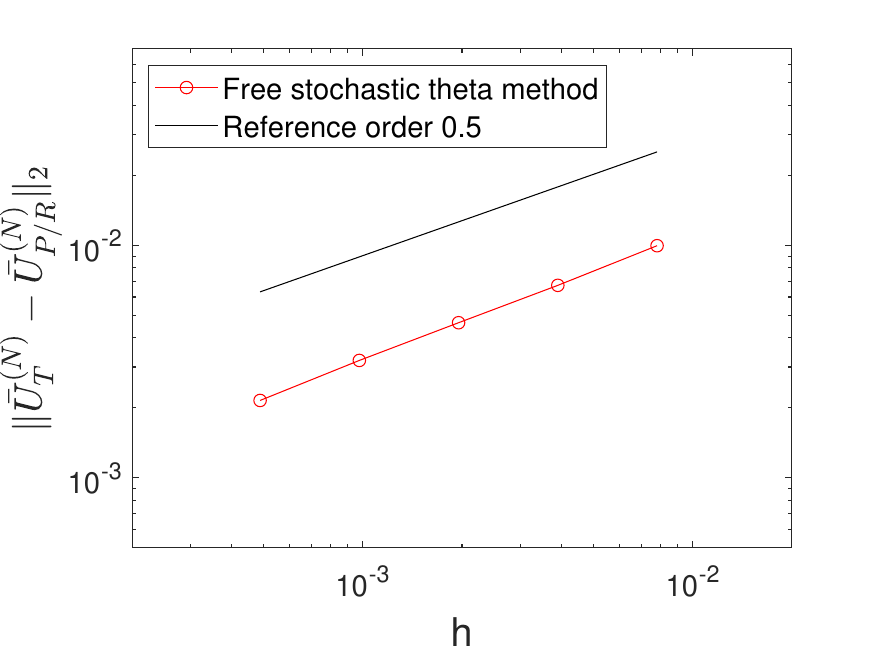}}
		
		\centerline{{\scriptsize$ \theta =1$}}
	\end{minipage}
	\caption{Strong convergence order simulations for the free STMs applied to \eqref{cir} with $\alpha=2, \beta=4,$ and $\sigma=1$.}\label{cirs}
\end{figure}

Since the free CIR equation \eqref{cir} does not have a trivial solution, we also should consider the exponential stability in the perturbative sense (see Definition \ref{def:rdjn} in Appendix \ref{sec:appendix}). Unfortunately, the parameters of the free CIR equation does not fulfill Assumption \ref{ass4}, so Theorem \ref{12wdxrd1} cannot be applied. However, our numerical experiment still shows the stability. Let $ \bar{U}^{(N)}_n $ and $ \bar{V}^{(N)}_n$ be the numerical solutions that are generated by two different initial values $\bar{U}^{(N)}_0 =E$ and $ \bar{V}^{(N)}_0 =2E,$ respectively. Choosing $N=100, \alpha=2, \beta=4, \sigma=1,$ Figure \ref{cirw} demonstrates that the free EM are not stable with $h=1$, while the free STMs with $\theta=0.5,1$ are stable for different step sizes. On the premise of inheriting the exponential stability of the original equation, the free EM requires $1162.32s$ with $ h = 0.25$ and $T=20$, while the free BEM only needs $23.45s$ with $ h = 4$.
\begin{figure}[!htbp]
			\centering
	\begin{minipage}{0.32\linewidth}
		\vspace{3pt}
		\centerline{\includegraphics[width=\textwidth]{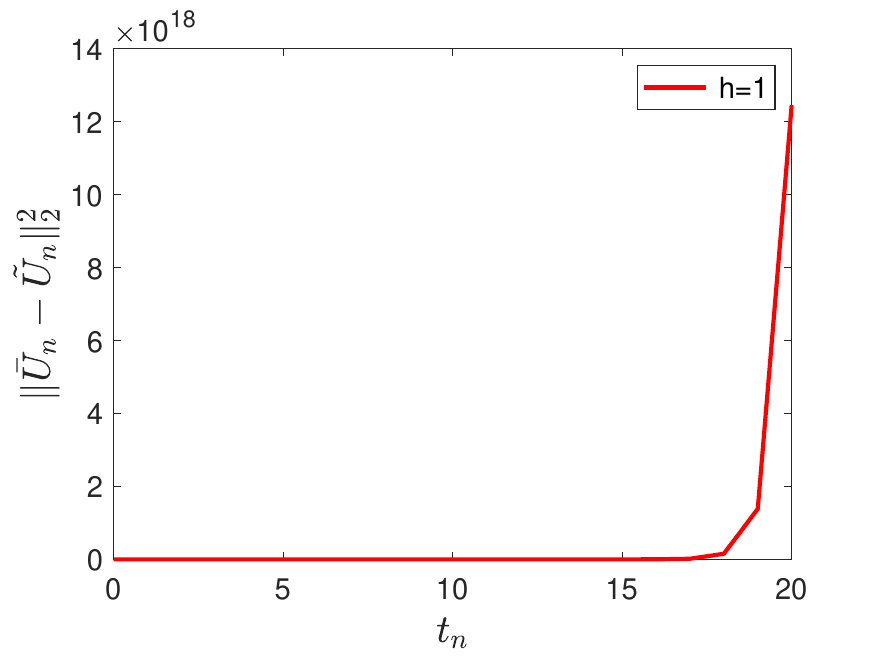}}
		\centerline{{\scriptsize\textit{free EM}}}
	\end{minipage}	
	\begin{minipage}{0.32\linewidth}
		\vspace{3pt}
		\centerline{\includegraphics[width=\textwidth]{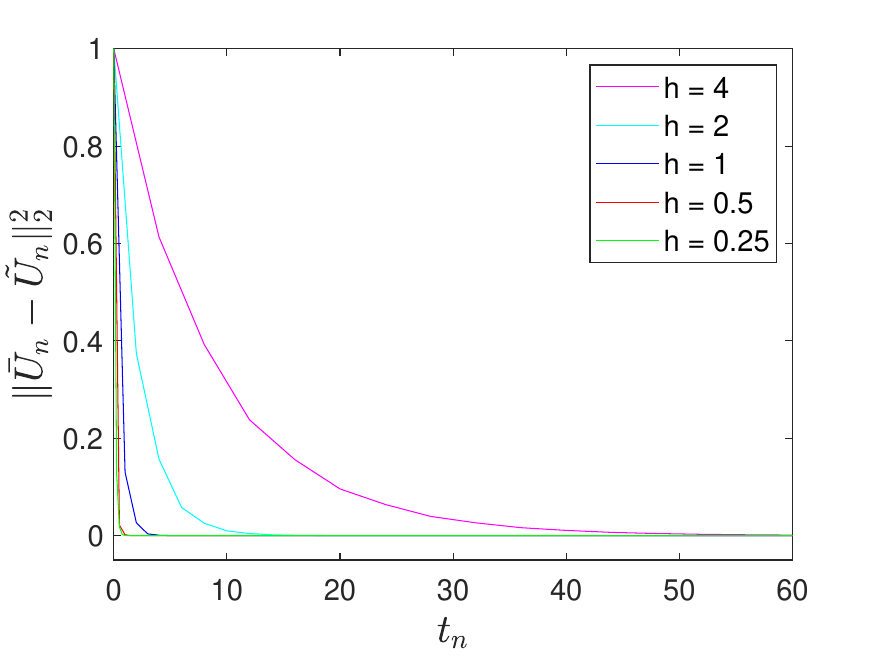}}
		\centerline{{\scriptsize\textit{ free STM with $\theta=0.5$}}}
	\end{minipage}
	\begin{minipage}{0.32\linewidth}
		\vspace{3pt}
		\centerline{\includegraphics[width=\textwidth]{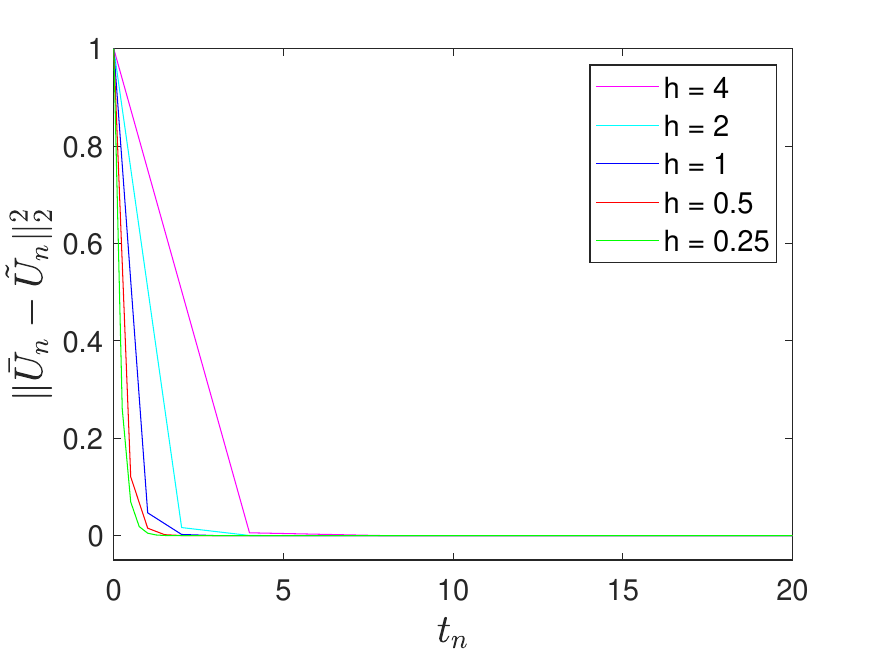}}
		\centerline{{\scriptsize\textit{free BEM}}}
	\end{minipage}
	\caption{The exponential stability in mean square of free STMs applied to \eqref{cir}.}\label{cirw}
\end{figure}

\subsection{An example for nonlinear $\alpha$}
To the best of our knowledge, the functions $\alpha$ of present-day known examples are usually linear (see the above mentioned examples). In this case, our numerical methods are “pseudo-implicit” and can be solved explicitly. It is interesting to consider a free SDE with a nonlinear $\alpha.$ Consider the following free SDE
\begin{equation}\label{usinu}
{\rm d} U_t =(\mu U_t+\nu \sin(\psi(U_t)){\bf 1} ){\rm dt} + \sigma U_t {\rm d} W_t+ \sigma {\rm d} W_t U_t,
\end{equation} 
where $ \mu,\nu,\sigma \in \mathbb{R} $. Suppose that $U_0= {\bf 1}.$ There exists an unique global solution to \eqref{usinu} (see \eqref{fSDEdanbian}); however, it is difficult to obtain the exact solution. Figure \ref{usinut} shows the spectrum distribution of the numerical solution obtained by the free BEM. Figure \ref{usinus} indicates that the mean square errors decrease at a slope approaching {1/2}, which confirms the predicted convergence result.
 \begin{figure}[!htbp]
	\centering
	\centering
	\begin{minipage}{0.32\linewidth}
		\vspace{3pt}
		\centerline{\includegraphics[width=\textwidth]{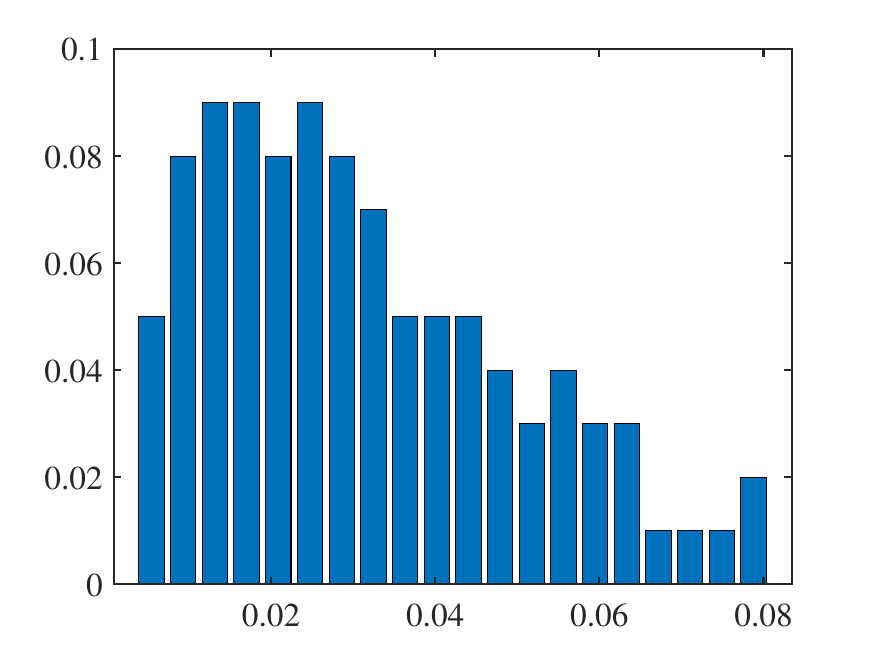}}
		\centerline{{\scriptsize$ T=300h, N=100 $}}
	\end{minipage}
	\begin{minipage}{0.32\linewidth}
		\vspace{3pt}
		\centerline{\includegraphics[width=\textwidth]{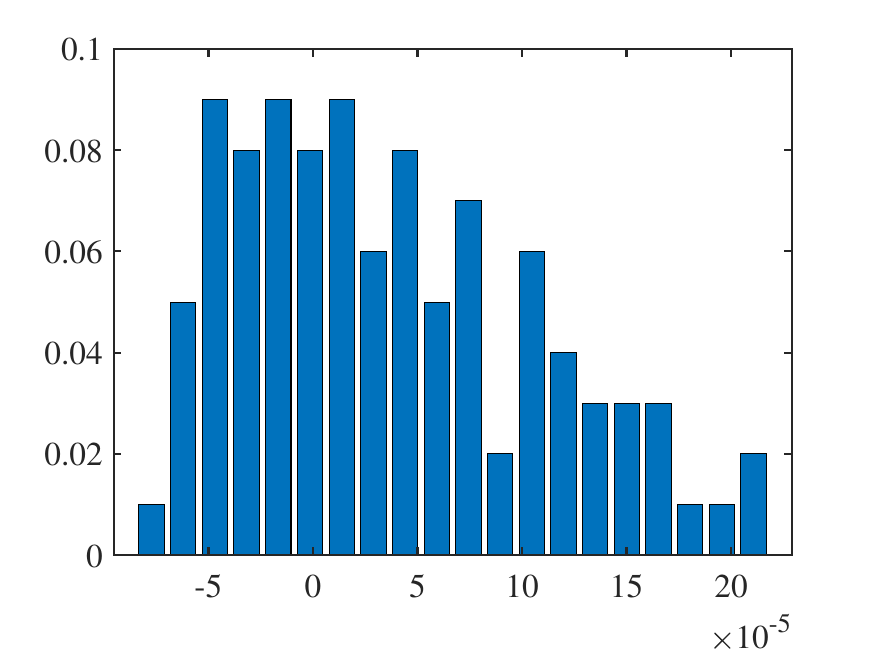}}
		\centerline{{\scriptsize$ T=1024h, N=100 $}}
	\end{minipage}
	\begin{minipage}{0.32\linewidth}
		\vspace{3pt}
		\centerline{\includegraphics[width=\textwidth]{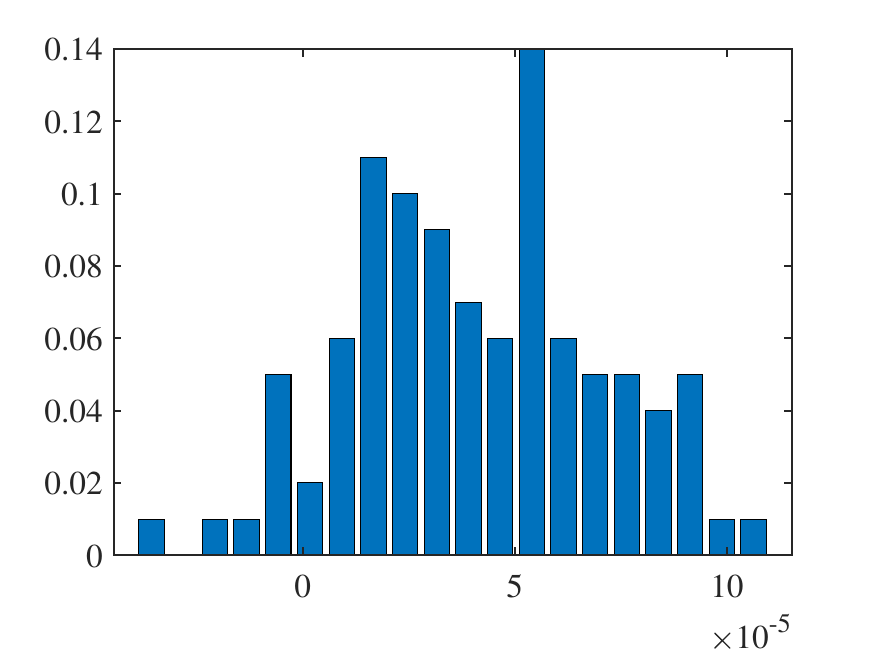}}
		\centerline{{\scriptsize$ T=4000h, N=100 $}}
	\end{minipage}
	\caption{The spectrum distribution of the numerical solution $ \bar{U}_P^{(N)} $ of \eqref{usinu} generated by free BEM with $ \mu=-8, \nu=-4, $ and $\sigma=1$.}\label{usinut}
\end{figure}

\begin{figure}[!htbp]
	\centering
	\centering
	\begin{minipage}{0.32\linewidth}
		\vspace{3pt}
		\centerline{\includegraphics[width=\textwidth]{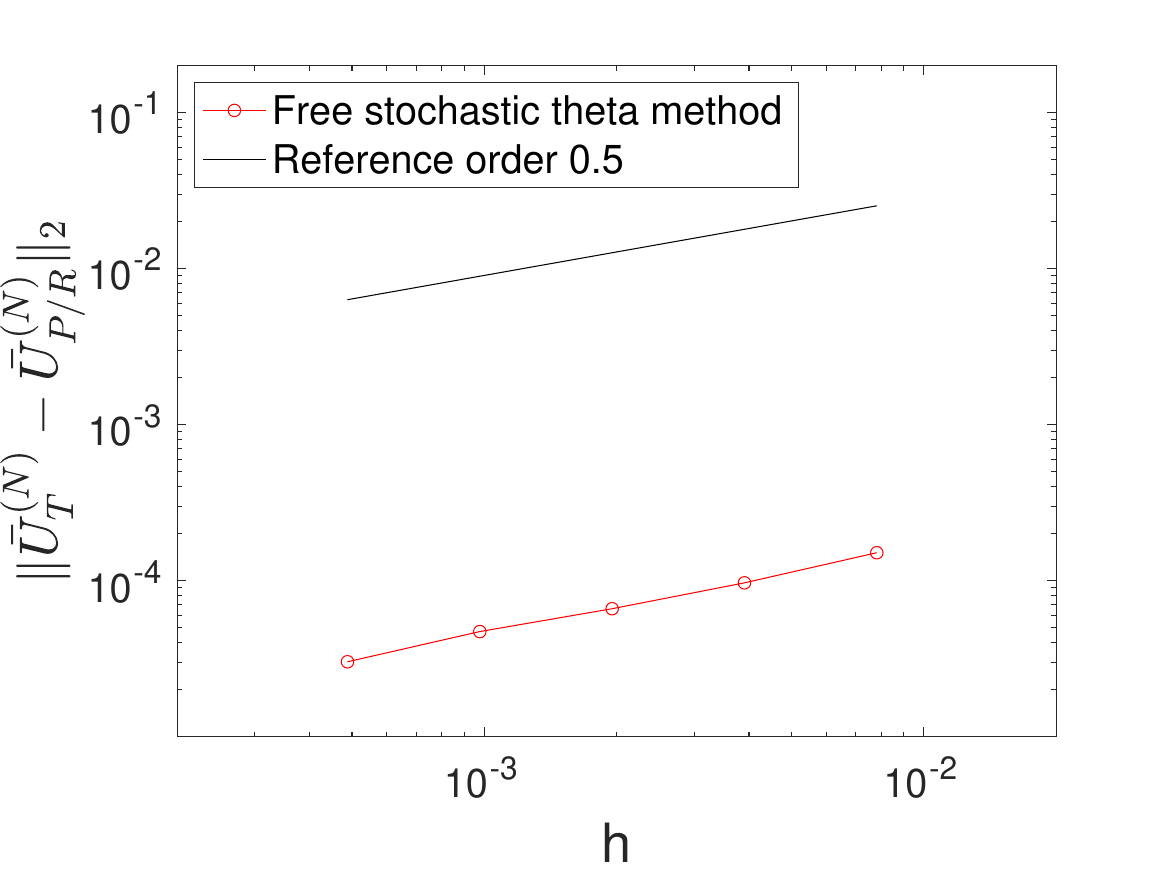}}
		\centerline{{\scriptsize$ \theta =0$}}
	\end{minipage}
	\begin{minipage}{0.32\linewidth}
		\vspace{3pt}
		\centerline{\includegraphics[width=\textwidth]{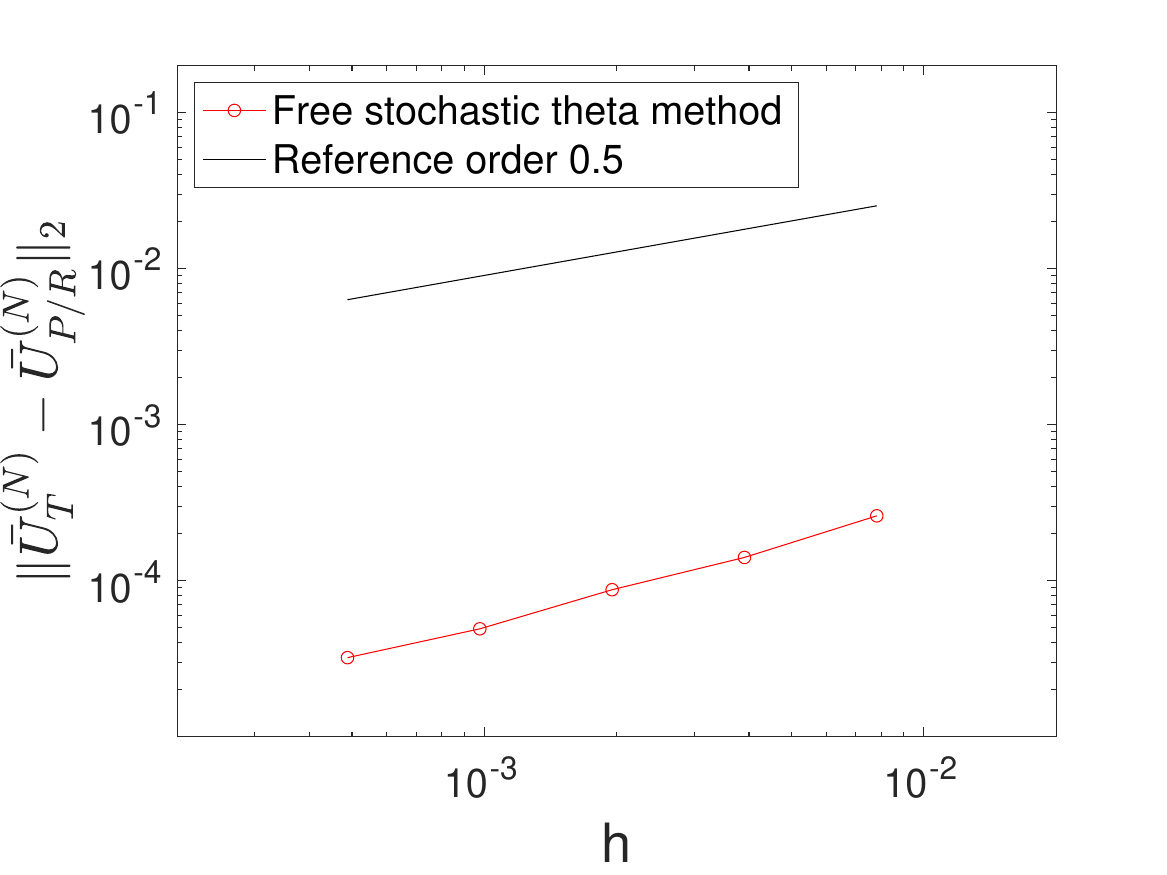}}
		
		\centerline{{\scriptsize$ \theta =0.5$}}
	\end{minipage}
	\begin{minipage}{0.32\linewidth}
		\vspace{3pt}
		\centerline{\includegraphics[width=\textwidth]{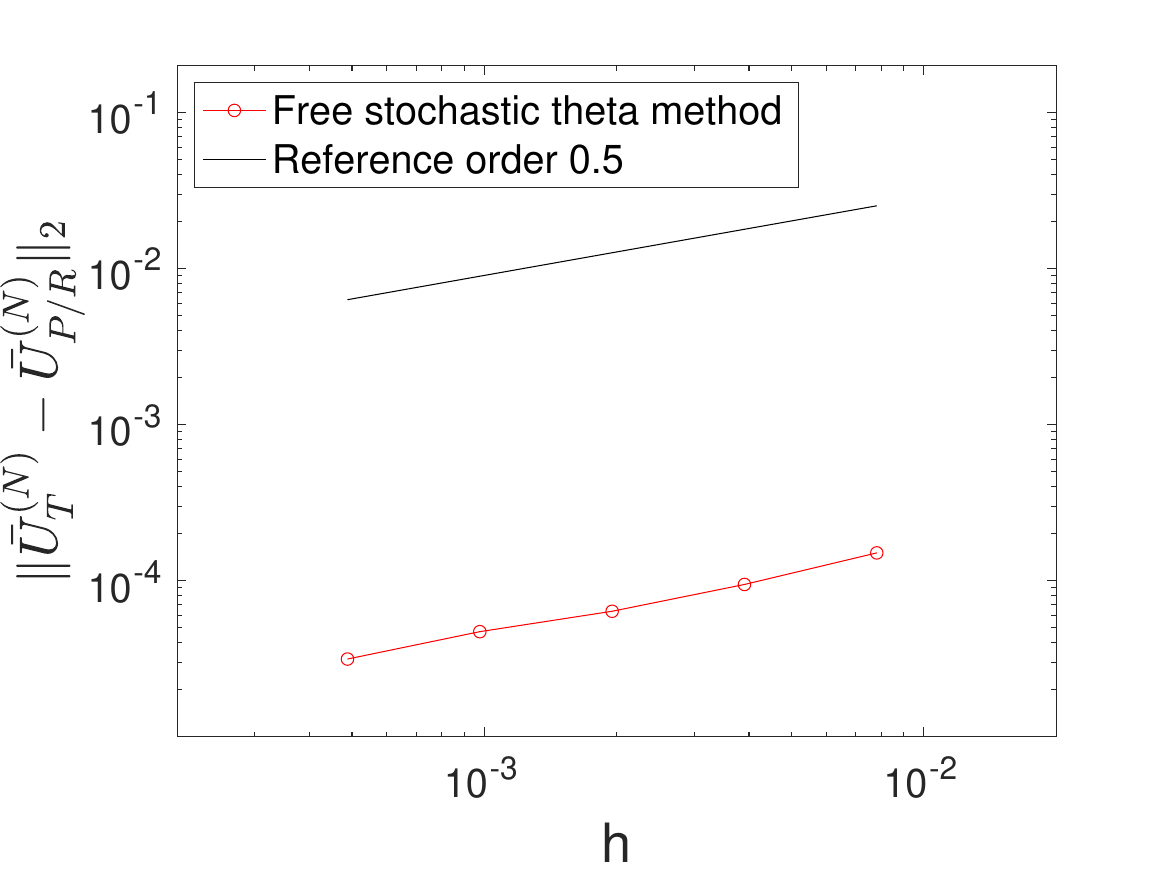}}
		
		\centerline{{\scriptsize$ \theta =1$}}
	\end{minipage}
	\caption{Strong convergence order simulations for the free STMs applied to \eqref{usinu} with $ \mu=-8, \nu=-4, $ and $\sigma=1$.}\label{usinus}
\end{figure}

Note that the free SDE \eqref{usinu} has a trivial solution. By Theorem \ref{12wdx}, we set $\bar{L}=-\mu-|\nu|-2|\sigma|,\ \bar{K}=2(\mu^2+\nu^2),$ and let $\mu+|\nu|+2|\sigma|<0.$ Then for $\theta \in [0,1/2),$ the free STMs are exponentially stable when $h<\frac{-2\mu-2|\nu|-4|\sigma|}{2(1-2\theta)(\mu^2+\nu^2)}:=h_{\max}$. And for $\theta \in [1/2, 1],$ the free STMs are exponentially stable for any given $h > 0.$

We choose $N=100, \mu=-8, \nu=-4,$ and $\sigma=1.$ Figure \ref{usinuw} shows the exponential stability of the free STMs with $h<h_{\max}$ for different $\theta=0,0.25,0.4$ and the free STMs with different step sizes ($h=0.25,0.5,1,2,4$) for $\theta=0.5,1$. On the premise of inheriting the exponential stability of the original equation, the free EM requires $522.48s$ with $ h = 0.015$ and $T=20$, while the free BEM only needs $5.43s$ with $ h = 4$. Figure \ref{usinuw} also demonstrates that the free STMs with $\theta=0,0.25,0.4$ are not stable with large step sizes $ h>h_{\max}$. 
\begin{figure}[!htbp]
			\centering
	\begin{minipage}{0.32\linewidth}
		\vspace{3pt}
		\centerline{\includegraphics[width=\textwidth]{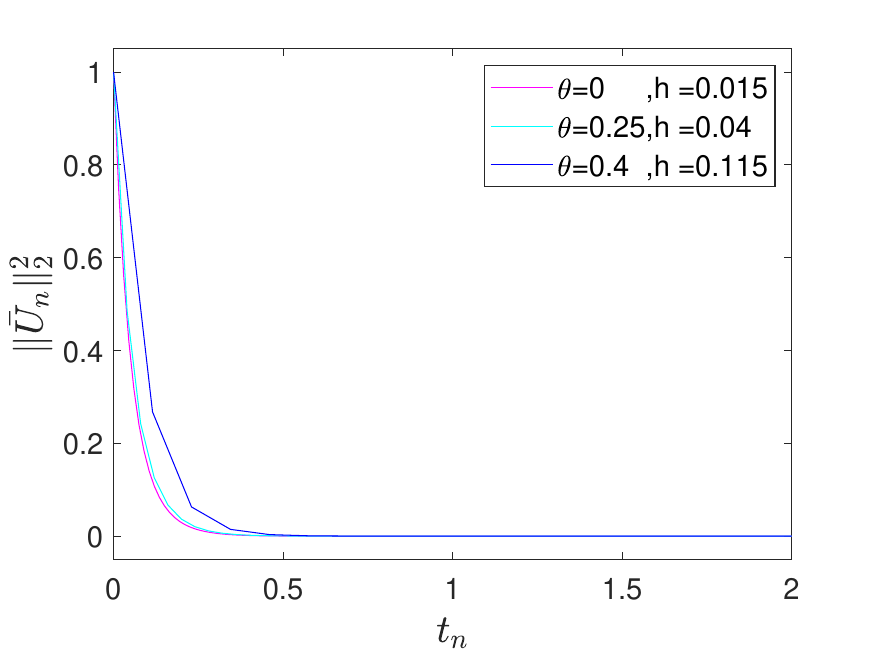}}
		\centerline{{\scriptsize\textit{free STMs with $0\leq\theta<0.5$}}}
	\end{minipage}
	\begin{minipage}{0.32\linewidth}
		\vspace{3pt}
		\centerline{\includegraphics[width=\textwidth]{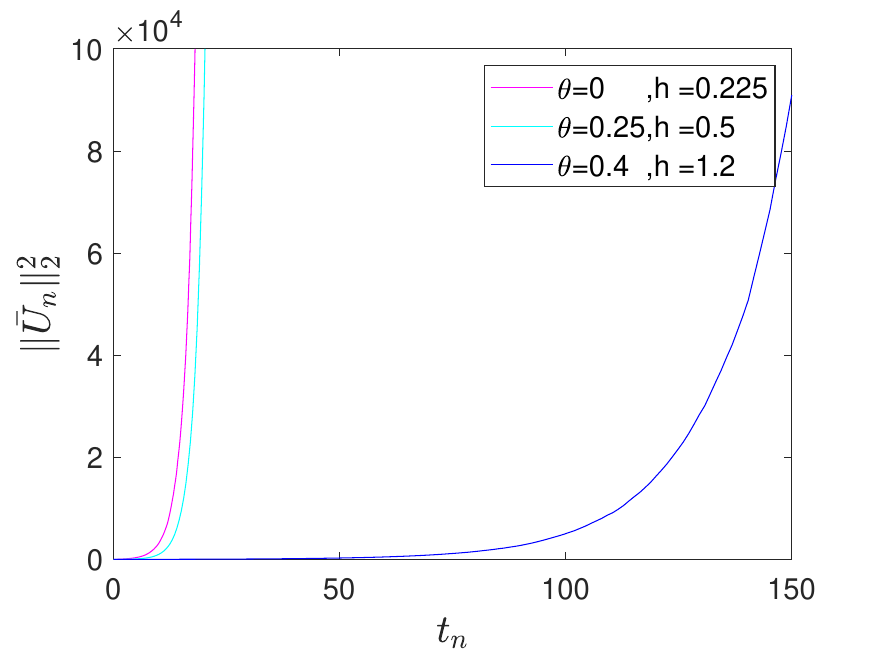}}
		\centerline{{\scriptsize\textit{ free STMs with $0\leq\theta<0.5$}}}
	\end{minipage}
	
	\begin{minipage}{0.32\linewidth}
		\vspace{3pt}
		\centerline{\includegraphics[width=\textwidth]{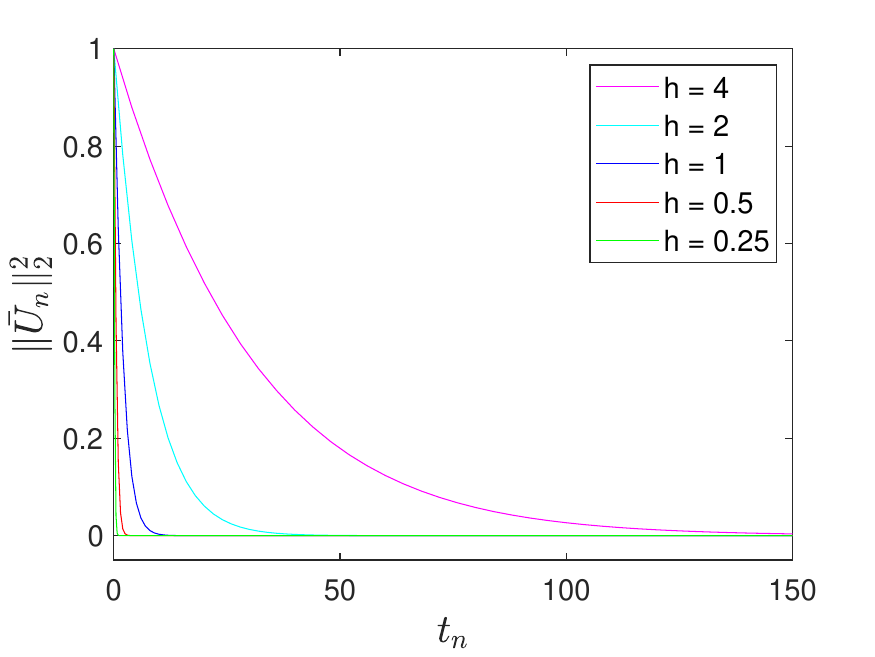}}
		\centerline{{\scriptsize\textit{ free STM with $\theta=0.5$}}}
	\end{minipage}
	\begin{minipage}{0.32\linewidth}
		\vspace{3pt}
		\centerline{\includegraphics[width=\textwidth]{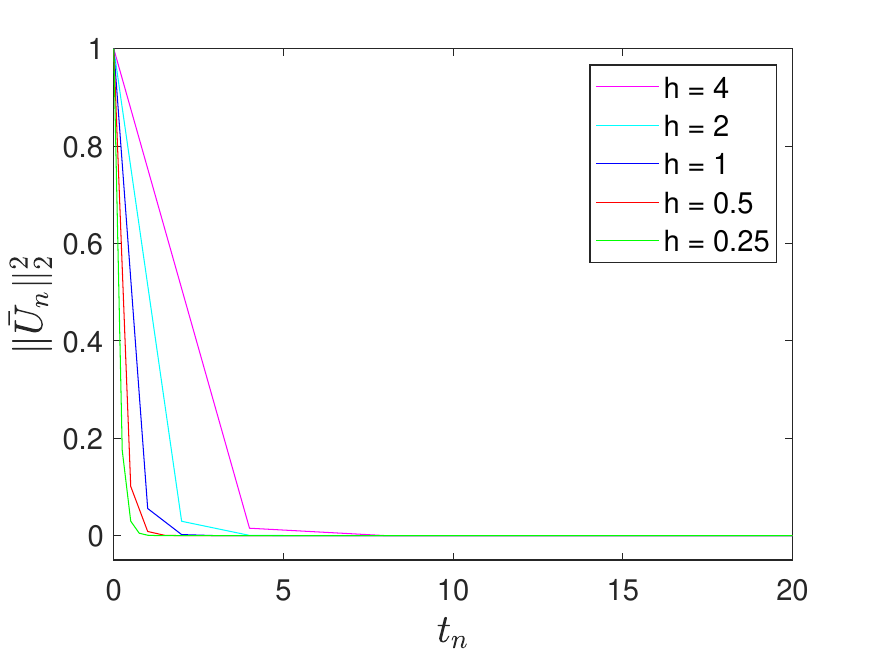}}
		\centerline{{\scriptsize\textit{free BEM}}}
	\end{minipage}
	\caption{The exponential stability in mean square of free STMs applied to \eqref{usinu}.}\label{usinuw}
\end{figure}

\vspace{2mm}

\section{Conclusions}
\label{sec:conclusions}
In this paper, we derived a family of free STMs with parameter $\theta \in[0,1]$ for numerically solving free SDEs. We recovered the standard convergence order $1/2$ of free STMs for free SDEs and derived conditions of exponential stability in mean square of the exact solution and the numerical solution. We obtained a step size restriction for the free STMs. It is shown that free STM with $\theta \in [1/2, 1]$ can inherit the exponential stability of the original equations for any given step size. Therefore, free STM with $\theta \in [1/2, 1]$  is superior to the free EM given previously in the literature. The theoretical analysis and numerical simulations demonstrated that the free STM with $\theta \in [1/2, 1]$ (especially the free BEM) is very promising. It is shown that the numerical approximations generated by the free STMs to free SDEs with coefficients in the stochastic term that do not satisfy the locally operator Lipschitz condition also converge to the exact solution with order $1/2$ and behave exponentially stable. In addition, the solution to free SDEs should be self-adjoint, and the eigenvalues of the numerical solutions should be real. We obtained a few eigenvalues with a very small imaginary part when simulating free GBM I \eqref{geo1}. Looking for numerical methods that are structure-preserving will be considered in future work.
We end our conclusions with the following open questions:

${\rm (a)}.$  The operator Lipschitz conditions for the coefficients $\alpha,\beta,$ and $\gamma$ are crucial to proving the existence and uniformly boundedness of the numerical solution. It is interesting to explore the numerical methods for the free SDEs with weak conditions on the coefficients, in particular for the ones with sublinear or superlinear drift coefficients.

${\rm (b)}.$ For the free OU equation \eqref{ou}, the numerical experiment shows that the strong convergence order is $1$ for the free STMs. Is it possible to prove this order in theory? Referring to the classical case, we may introduce the free analogue of Milstein method by using the non-commutative stochastic Taylor expansion \cite[Corollary 5.8]{Azamov2009} (see also \cite{schluechtermannNumericalSolutionFree2022a}). Is it possible to show the strong convergence of the free Milstein method? When our paper was under peer review, Schl\"{u}chtermann and Wibmer announced an interesting paper on arXiv regarding the free Milstein method \cite{wibmer_milstein_2025}.

${\rm (c)}.$ It can be shown that the spectral distribution of the free GBM I \eqref{geo1} shrinks to zero exponentially fast \cite{karginFreeStochasticDifferential2011} when $\mu<0.$ But by Theorem \ref{12wdx}, we require $\mu < -1/2,$ i.e., $\bar{L}>0$ for the stability. So there is a gap between $0$ and $-1/2.$ A similar phenomenon also happens for the free GBM II \eqref{geo}. Is it possible to fill this gap in theory? 

\section*{Acknowledgments}
We would like to thank Profs. Siqing Gan, Yong Jiao, Xiaojie Wang, and Lian Wu for their helpful discussion. We also would like to thank the anonymous reviewers for their helpful remarks that improved this paper.

\section*{Funding}
This work was supported by the Natural Science Foundation of Changsha (Nos. kq2502101, kq2402192), the National Natural Science Foundation of China (Nos. 12371417, 12471394, 12031004, 12071488), Postdoctoral Fellowship Program of CPSF (No. GZC20233194), and the Fundamental Research Funds for the Central Universities of Central South University (No. 2025ZZTS0604).

\bibliographystyle{plain}
\bibliography{ref}

\section{Appendix}\label{sec:appendix}

Since the free OU equation does not have a trivial solution, it motivates us to introduce the exponential stability in mean square of the perturbation solution (see \cite{DAmbrosio2023} for the classical case).
\begin{Defi}
Let $\{U_t\}_{t\geq 0}$ and $\{V_t\}_{t\geq 0}$ be two solutions to the free SDEs \eqref{eq:fSDE} with different initial values $U_0$ and $V_0$, respectively. The free SDEs \eqref{eq:fSDE} are said to be exponentially stable in mean square if there exist constants $c_1, c_2 >0$ such that, for all initial value $U_0,V_0 \in \mathscr{F}^{sa},$
\begin{equation}
		\left\|U_t-V_t \right\|_2^2 \le  c_1 e^{-c_2 t} \left\|U_0-V_0 \right\|_2^2,
		\end{equation}
for all $t \geq 0.$
\end{Defi}

	\begin{Prop}\label{ou:stable-solution}
		Let $\{U_t\}_{t\geq 0}$ and $\{V_t\}_{t\geq 0}$ be two solutions to the free OU equation \eqref{ou} with different initial values $U_0$ and $V_0$, respectively. If $\mu<0$,
		then there exists a constant $C'_{\ref{ou:stable-solution}}>0,$ such that 
		for all $ t\geq 0 $ 
		\begin{equation*}
		\left\|U_t-V_t \right\|_2^2\le e^{-C'_{\ref{ou:stable-solution}} t} \left\|U_0 -V_0\right\|_2^2.
		\end{equation*}
	\end{Prop}

\begin{proof}
Since $\{U_t\}_{t\geq 0}$ and $\{V_t\}_{t\geq 0}$ are two solutions to the free OU equation with different initial values $U_0$ and $V_0$, respectively, we can obtain 
\begin{equation}
{\rm d}(U_t-V_t)=\mu (U_t-V_t) {\rm d}t.
\end{equation}
For all $t \geq 0,$ the free It\^{o} formula \eqref{eq:free-ito} readily implies that (see also \cite{Kemp2016})
		\begin{equation}\label{itor}
		\begin{split}
		&{\rm d} (e^{-\mu t} (U_t-V_t)^2) \\
		&=e^{-\mu t} \cdot {\rm d} (U_t-V_t)^2 + (U_t-V_t)^2 \cdot {\rm d} e^{-\mu t}  \\
		&=e^{-\mu t} \left((U_t-V_t) \cdot {\rm d} (U_t-V_t) + {\rm d} (U_t-V_t)\cdot (U_t-V_t)+ {\rm d} (U_t-V_t) \cdot {\rm d} (U_t-V_t)\right)+ (U_t-V_t)^2 \cdot {\rm d} e^{-\mu t}\\
		&=e^{-\mu t} \left(\mu(U_t-V_t)^2  {\rm dt}+\mu(U_t-V_t)^2 {\rm dt}-\mu (U_t-V_t)^2 {\rm d} t\right).
		\end{split}
		\end{equation}
		Integrating from $0$ to $t$ on both sides of \eqref{itor} gives
		\begin{equation*}
			e^{-\mu t} (U_t -V_t)^2 - (U_0 -V_0)^2=\mu \int_0^t e^{-\mu s} (U_s-V_s)^2  {\rm d s}.
		\end{equation*}
Hence, note that $\mu<0,$ we have 
\begin{equation*}
\left\|U_t -V_t \right\|_2^2 \leq e^{\mu t} \left\|U_0-V_0 \right\|_2^2,
\end{equation*}
for all $t \geq 0.$
\end{proof}
		
Analogue to the classical case \cite{chenConvergenceStabilityBackward2020,MR4589707}, we have the following definition of exponential stability in mean square for a numerical method which produces the numerical solutions $\bar{U}_n$ and $\bar{V}_n$ with different initial values $U_0$ and $V_0$:
\begin{Defi}\label{def:rdjn}
For a given step size $h>0$, a numerical method is said to be exponentially stable in mean square if there exist constants $C_1, C_2>0$, such that with initial value $U_0,V_0\in \mathscr{F}^{sa}$,
		\begin{equation} \label{wdx}
		\left\|\bar{U}_n-\bar{V}_n \right\|_2^2\le C_1 e^{- C_2 nh} \left\|U_0-V_0 \right\|_2^2,
		\end{equation}
		for any $ n \in \mathbb{N}.$
	\end{Defi}

\begin{Theo}\label{12wdxrd}
Let $\{\bar{U}_n\}$ and $\{\bar{V}_n\}$ be numerical solutions to the free OU equation \eqref{ou}. 
	Suppose that $\mu <0.$
	Then we have the following statements.
	\begin{itemize}
	\item If $ \theta \in [0, 1/2),$ set $\rho$ such that $0< \rho < \frac{-2}{(1-2\theta)\mu},$ then for any given step size $0< h \leq \rho,$ there exists a constant $C_{\ref{12wdxrd}}>0$ such that		
	\begin{equation*}
		\left\|\bar{U}_n -\bar{V}_n\right\|_2^2\le e^{-C_{\ref{12wdxrd}}nh} \left\|U_0-V_0 -\theta \mu h (U_0- V_0)\right\|_2^2,
		\end{equation*}
		for any $ n\in \mathbb{N}$.
	\item  If $\theta \in [1/2, 1]$, then for any given step size $h>0,$
there exists a constant $C'_{\ref{12wdxrd}}>0$ such that
	\begin{equation*}
		\left\|\bar{U}_n -\bar{V}_n\right\|_2^2\le e^{-C'_{\ref{12wdxrd}}nh} \left\|U_0-V_0 -\theta \mu h (U_0-V_0)\right\|_2^2,
		\end{equation*}
		for any $ n\in \mathbb{N}$.
	\end{itemize}		
\end{Theo}

\begin{proof}
Recall that for the free OU, the free STMs are given by 
\begin{equation}\label{eq:theta3}
\bar{U}_{n+1}-\bar{V}_{n+1}- \theta\mu h (\bar{U}_{n+1} - \bar{V}_{n+1}) =\bar{U}_n-\bar{V}_n+(1-\theta) \mu h (\bar{U}_{n} - \bar{V}_{n}).
\end{equation} 
Thus, we can bound the left side of \eqref{eq:theta3} as follows:
		\begin{equation}\label{th8.1}
		\begin{split}
		\left\| \bar{U}_{n+1}-\bar{V}_{n+1}- \theta\mu h (\bar{U}_{n+1} - \bar{V}_{n+1})\right\|_2^2 & =(1- \theta \mu h)^2 \left\|\bar{U}_{n+1}-\bar{V}_{n+1}\right\|_2^2\\
		 & \geq  \left(1- 2 \theta \mu  h\right) \left\|\bar{U}_{n+1} - \bar{V}_{n+1}\right\|_2^2.
		\end{split}
		\end{equation}
Similar to \eqref{th6.2}, the right side of \eqref{eq:theta3} can be estimated as follows: 
\begin{equation}\label{th8.2}
\begin{split}
& \left\| \bar{U}_n-\bar{V}_n+(1-\theta) \mu h (\bar{U}_{n} - \bar{V}_{n}) \right\|_2^2\\
& = \left\| \bar{U}_n-\bar{V}_n-\theta \mu h(\bar{U}_{n} - \bar{V}_{n})\right\|_2^2  + (1-2\theta) \mu^2 h^2 \left\| \bar{U}_{n} -\bar{V}_{n}\right\|_2^2+ 2 \mu  h  \left\| \bar{U}_{n} -\bar{V}_{n}\right\|_2^2\\
&= \left\| \bar{U}_n-\bar{V}_n-\theta \mu h(\bar{U}_{n} -\bar{V}_{n})\right\|_2^2+((1-2\theta) \mu^2 h + 2 \mu)h  \left\| \bar{U}_n-\bar{V}_n\right\|_2^2.
\end{split}
\end{equation}
Moreover, it follows from $\|U-V\|_2^2\leq2\|U\|_2^2+2\|V\|_2^2$ that 
\begin{equation*}
\begin{split}
 \left\| \bar{U}_n-\bar{V}_n-\theta \mu h (\bar{U}_{n} - \bar{V}_{n}) \right\|_2^2 & \leq 2\left\| \bar{U}_n-\bar{V}_n\right\|_2^2+2 \theta^2 \mu^2 h^2\left\| \bar{U}_n - \bar{V}_n \right\|_2^2\\
 & \leq(2+2 \theta^2 \mu^2 h^2)\left\| \bar{U}_n-\bar{V}_n\right\|_2^2,
 \end{split}
 \end{equation*}
which implies
 \begin{equation}\label{th8.3}
\left\| \bar{U}_n-\bar{V}_n\right\|_2^2\geq \frac{1}{2+2 \theta^2\mu^2 h^2}  \left\| \bar{U}_n-\bar{V}_n-\theta\mu h (\bar{U}_{n} - \bar{V}_{n}) \right\|_2^2.
 \end{equation}

Now denote $f_n:=\left\|\bar{U}_n-\bar{V}_n-\theta \mu h(\bar{U}_{n} -\bar{V}_{n}) \right\|_2^2$. 
Let $\theta \in [0, 1/2),$ for any $0< h \leq \rho< \frac{-2}{(1-2\theta)\mu},$ combining \eqref{eq:theta3}, \eqref{th8.2}, and \eqref{th8.3} leads to 
 \begin{equation}\label{th8.4}
f_{n+1}\leq \left(1+\frac{((1-2\theta) \mu^2 h +2\mu)h}{2+ 2\theta^2 \mu^2 h^2} \right) f_n \leq(1-C_{\ref{12wdxrd}}h)f_n,
\end{equation}
where $C_{\ref{12wdxrd}}:=\frac{(2\theta-1)\mu^2 \rho- 2\mu}{2+2\theta^2h^2\mu^2}>0$.
Applying $(1+x)\leq e^x$ to \eqref{th8.4} and iterating $n$ times, we obtain
\begin{equation*}\label{th8.5}
f_n \leq e^{-C_{\ref{12wdxrd}}nh}f_0.
\end{equation*}
Therefore, combining \eqref{th8.1}, we have
\begin{equation*}
\begin{split}
 \left\|\bar{U}_{n} -\bar{V}_n\right\|_2^2 
& \leq  \frac{1}{\left(1-2 \theta \mu h\right)}e^{-C_{\ref{12wdxrd}}nh}f_0 \leq e^{-C_{\ref{12wdxrd}}nh} f_0,
\end{split}
\end{equation*}
which completes our first statement.

For $\theta \in [1/2, 1]$, note that $ (1-2\theta) \mu^2 h^2 \| \bar{U}_n-\bar{V}_n \|_2^2\leq0$ for any $h>0.$ Hence, \eqref{th8.2} reduces to 
\begin{equation}
\left\| \bar{U}_n-\bar{V}_n+(1-\theta) \mu h (\bar{U}_{n} - \bar{V}_{n}) \right\|_2^2  \leq \left\| \bar{U}_n-\bar{V}_n-\theta \mu h(\bar{U}_{n} -\bar{V}_{n})\right\|_2^2+ 2 \mu h \left\| \bar{U}_n-\bar{V}_n\right\|_2^2.
\end{equation}
Similar to \eqref{th8.4}, we can obtain 
\begin{equation}
f_{n+1}\leq \left(1+ \frac{2 \mu h}{2+ 2\theta^2 \mu^2 h^2} \right) f_n =(1-C'_{\ref{12wdxrd}}h)f_n,
\end{equation}
where $C'_{\ref{12wdxrd}}:=\frac{-2\mu}{2+2\theta^2h^2 \mu^2}>0$.
Hence, for any given step size $h>0,$ we have 
\begin{equation*}
f_n \leq e^{-C'_{\ref{12wdxrd}}nh}f_0,
\end{equation*}
which completes our proof.
\end{proof}

In fact, Proposition \ref{ou:stable-solution} and Theorem \ref{12wdxrd} can be generalized to the following free SDEs 
\begin{equation}\label{eq:fSDEdb}
{\rm d} U_t =\alpha (U_t) {\rm dt} +\sum_{i=1}^k (\beta^i (U_t)  {\rm d} W_t + {\rm d} W_t \gamma^i (U_t)).
\end{equation}
We remark that the proofs are quite similar to the proofs of Proposition \ref{prop:stable-solution} and Theorem \ref{12wdx}, and we omit the details.

\begin{Assu}\label{ass4}	
There exists a constant $ \bar{L} >0,$ such that for any $ U \in \mathscr{F}^{sa},$ 
	\begin{equation}\label{wdxtjrd}
	\psi\left( (U-V) (\alpha(U)-\alpha(V)) \right) +\frac{k}{2}\sum_{i=1}^k\left(\left\|\beta^i(U)-\beta^i(V) \right\|_2^2+\left\|\gamma^i(U)-\gamma^i(V) \right\|_2^2\right)\leq -\bar{L} \left\| U-V \right\|_2^2.
	\end{equation}
\end{Assu}
	
	\begin{Prop}\label{stable-solution}
		Let $\{U_t\}_{t\geq 0}$ and $\{V_t\}_{t\geq 0}$ be two solutions to the free SDEs \eqref{eq:fSDEdb} with different initial values $U_0$ and $V_0$, respectively. Suppose that Assumption \ref{ass4} holds,
		then there exists a constant $C'_{\ref{stable-solution}}>0,$ such that 
		for all $ t\geq 0 $ 
		\begin{equation*}
		\left\|U_t-V_t \right\|_2^2\le e^{-C'_{\ref{stable-solution}} t} \left\|U_0 -V_0\right\|_2^2.
		\end{equation*}
	\end{Prop}

\begin{Theo}\label{12wdxrd1}
Let $\{\bar{U}_n\}$ and $\{\bar{V}_n\}$ be numerical solutions to the free SDEs \eqref{eq:fSDEdb} with different initial values $U_0$ and $V_0$, respectively.
Suppose that Assumption \ref{ass4} holds, and additionally we assume that $\alpha$ is operator Lipschitz in the $L^2(\psi)$ norm with $\bar{K}: =L_\alpha^2.$
Then we have the following statements.
	\begin{itemize}
	\item If $\theta \in [0, 1/2),$ set $\rho$ such that $0< \rho < \frac{2\bar{L}}{(1-2\theta)\bar{K}},$ then for any given step size $0< h \leq \rho,$ there exists a constant $C_{\ref{12wdxrd1}}>0$ such that
	\begin{equation*}
		\left\|\bar{U}_n -\bar{V}_n\right\|_2^2\le e^{-C_{\ref{12wdxrd1}}nh} \left\|U_0-V_0 -\theta h (\alpha(U_0)-\alpha(V_0))\right\|_2^2,
		\end{equation*}
		for any $ n\in \mathbb{N}$.
		\item If $\theta \in [1/2, 1]$, then for any given step size $h>0$, there exists a constant $C'_{\ref{12wdxrd1}}>0$ such that
		\begin{equation*}
		\left\|\bar{U}_n -\bar{V}_n\right\|_2^2\le e^{-C'_{\ref{12wdxrd1}}nh} \left\|U_0-V_0 -\theta h (\alpha(U_0)-\alpha(V_0))\right\|_2^2,
		\end{equation*}
		for any $ n\in \mathbb{N}$.
	\end{itemize}		
\end{Theo}

\end{document}